\documentclass[reqno,12pt]{amsart}
\usepackage{amsmath,amsfonts,amsthm,amssymb,indentfirst}
\usepackage[all,poly]{xy} 
\usepackage{color}
\usepackage[pagebackref,colorlinks]{hyperref}

\usepackage[all,poly]{xy}
\usepackage{mathrsfs}
\usepackage{enumerate}

\theoremstyle{plain}
\newtheorem{lemma}{Lemma}[section] 
\newtheorem{theorem}[lemma]{Theorem}
\newtheorem{corollary}[lemma]{Corollary}
\newtheorem{proposition}[lemma]{Proposition}

\theoremstyle{definition}
\newtheorem{remark}[lemma]{Remark}
\newtheorem{example}[lemma]{Example}

\newtheorem{definition}[lemma]{Definition}

\newtheorem{question}[lemma]{Question}

\newcommand{\Zset}{\mathbb Z}

\newcommand{\Cset}{\mathbb C}

\newcommand{\M}{\operatorname{\mathbb M}}
\newcommand{\gr}{\operatorname{gr}}
\newcommand{\ol}{\overline}

\newcommand{\End}{\operatorname{End}}
\newcommand{\Hom}{\operatorname{Hom}}

\newcommand{\A}{\mathcal A}
\newcommand{\B}{\mathcal B}

\newcommand{\V}{\mathcal V}

\newcommand{\Stab}{\operatorname{Stab}}

\newcommand{\POG}{\mathbf {POG}}
\newcommand{\OG}{\mathbf {OG}}
\newcommand{\GrURings}{\mathbf {GrURings}}
\newcommand{\GrRings}{\mathbf {GrRings}}
\newcommand{\GrMat}{\mathbf {GrMat}}

\title[Simplicial and dimension groups with group action]{Simplicial and dimension groups with group action and their realization}

\author{Lia Va\v s}
\address{Department of Mathematics, Physics and Statistics, University of the Sciences, Philadelphia, PA 19104, USA}
\email{l.vas@usciences.edu}



\thanks{The author is very grateful to the referee for a careful read, a thoughtful report, and for making some highly valuable suggestions.}

\begin{document}

\begin{abstract} 
We define simplicial and dimension $\Gamma$-groups, the generalizations of simplicial and dimension groups to the case when these groups have an action of an arbitrary group $\Gamma.$ Assuming that the integral group ring of $\Gamma$ is Noetherian, we show that every dimension $\Gamma$-group is isomorphic to a direct limit of a directed system of simplicial $\Gamma$-groups and that the limit can be taken in the category of ordered groups with order-units or generating intervals.  

We adapt Hazrat's definition of the Grothendieck $\Gamma$-group $K_0^{\Gamma}(R)$ for a $\Gamma$-graded ring $R$ to the case when $\Gamma$ is not necessarily abelian. If $G$ is a pre-ordered abelian group with an action of $\Gamma$ which agrees with the pre-ordered structure, we say that $G$ is {\em realized} by a $\Gamma$-graded ring $R$ if $K_0^{\Gamma}(R)$ and $G$ are isomorphic as pre-ordered $\Gamma$-groups with an isomorphism which preserves order-units or generating intervals.  

We show that every simplicial $\Gamma$-group with an order-unit can be realized by a graded matricial ring over a $\Gamma$-graded division ring. If the integral group ring of $\Gamma$ is Noetherian, we realize a countable dimension $\Gamma$-group with an order-unit or a generating interval by a $\Gamma$-graded ultramatricial ring over a $\Gamma$-graded division ring. We also relate our results to graded rings with involution which give rise to Grothendieck $\Gamma$-groups with actions of both $\Gamma$ and $\mathbb Z_2$. We adapt the Realization Problem for von Neumann regular rings to graded rings and concepts from this work and discuss some other questions.
\end{abstract}

\maketitle

\section{Introduction}

A pre-ordered abelian group $G$ is said to be {\em realized} by a ring $R$ if the Grothendieck group $K_0(R)$ and $G$ are isomorphic. If $G$ is a countable {\em dimension group}, i.e. $G$ is a countable abelian group which is partially ordered, directed, unperforated, and has the interpolation property, then $G$ can be realized by an ultramatricial algebra $R$ over a given field. Moreover, if $G$ is considered together with an order-unit or a generating interval, the isomorphism $K_0(R)\cong G$ can be chosen to preserve order-units or generating intervals. A fundamental result on dimension groups states that every dimension group is a direct limit of {\em simplicial groups}, direct sums of finitely many copies of the infinite cyclic group $\Zset,$ and those groups can be realized by matricial algebras over a field.   

The main goal of this paper is to generalize these classic results to the case when simplicial and dimension groups have an additional action of a group $\Gamma$. Such additional action is present for the Grothendieck groups of $\Gamma$-graded rings for example. The classic results correspond to the case when $\Gamma$ is trivial. 

If a ring $R$ is graded by a group $\Gamma,$ the Grothendieck $\Gamma$-group $K_0^{\Gamma}(R)$ is constructed analogously to the ordinary Grothendieck group $K_0(R),$
by considering finitely generated {\em graded} projective modules instead of finitely generated projective modules.
The group $K_0^{\Gamma}(R)$ carries more information than $K_0(R)$ and distinguishes some rings which have identical $K_0$-groups as illustrated by examples in \cite{Roozbeh_Annalen} and \cite{Roozbeh_Lia}. The Grothendieck $\Gamma$-group was introduced in \cite{Roozbeh_Annalen} (now also in \cite{Roozbeh_book}) for the case when $\Gamma$ is abelian. In section \ref{section_Gamma_groups}, we adapt the definitions and results from  \cite{Roozbeh_book} and \cite{Roozbeh_Lia} to the general case when $\Gamma$ is not necessarily abelian.

The $\Gamma$-grading of $R$ induces a $\Gamma$-action on $K_0^{\Gamma}(R)$ which agrees with the pre-order. We say that  $K_0^{\Gamma}(R)$ is a pre-ordered $\Gamma$-group. If $G$ is a pre-ordered $\Gamma$-group, we say that $G$ is {\em realized} by a $\Gamma$-graded ring $R$ if $K_0^{\Gamma}(R)$ and $G$ are isomorphic as pre-ordered $\Gamma$-groups. If $G$ is considered with an order-unit or a generating interval, the isomorphism is required to preserve them.  

We define ordered $\Gamma$-group generalizations of simplicial, ultrasimplicial, and dimension groups called {\em simplicial $\Gamma$-groups}, {\em ultrasimplicial $\Gamma$-groups} and {\em dimension $\Gamma$-groups} respectively. 

In section \ref{section_simplicial_group}, we define a simplicial $\Gamma$-group. The existence of a $\Zset$-module basis of a nontrivial simplicial group is reflected in our definition of a nontrivial simplicial $\Gamma$-group $G$ as follows: there are a positive integer $n$ and elements $x_1,\ldots, x_n,$ 
fixed by the same subgroup $\Delta$ of $\Gamma,$ which generate $G$ as a $\Zset[\Gamma]$-module and such that a $\Zset[\Gamma]$-linear combination of $x_1,\ldots,x_n$ is zero exactly when the coefficients map to zero under the natural $\Zset[\Gamma]$-module map $\Zset[\Gamma]\to\Zset[\Gamma/\Delta]$ (Definitions \ref{simplicial_cone_definition} and \ref{simplicial_definition}). The set $\{x_1,\ldots, x_n\}$ is called {\em a simplicial $\Gamma$-basis} and $\Delta$ is its stabilizer. If $\Delta$ is a normal subgroup, $G$ is a free $\Zset[\Gamma/\Delta]$-module and the elements $x_1,\ldots, x_n$ correspond to a $\Zset[\Gamma/\Delta]$-basis of $G$ (Proposition \ref{Delta_normal_case}). The main result of section \ref{section_simplicial_group} is Theorem \ref{realization_simplicial} stating that every simplicial $\Gamma$-group can be realized by a graded matricial ring over a $\Gamma$-graded division ring. In this case, the basis stabilizer $\Delta$ corresponds to the graded support of the graded division ring.  

In section \ref{section_limits}, we show that appropriate categories of ordered $\Gamma$-groups are closed under formation of direct limits (Proposition \ref{direct_limits}) which enables us to introduce an ultrasimplicial $\Gamma$-group as a direct limit of a sequence of simplicial $\Gamma$-groups with simplicial $\Gamma$-bases 
stabilized by the same subgroup of $\Gamma.$ The functor $K_0^\Gamma$ agrees with direct limits (Proposition \ref{direct_limits_K0}) and, as a corollary, the Grothendieck $\Gamma$-group of a graded ultramatricial ring over a graded division ring is an ultrasimplicial $\Gamma$-group (Corollary \ref{dir_lim_corollary}). Showing the converse, that every ultrasimplicial $\Gamma$-group can be obtained in this way, requires more work and we set to achieve this in the rest of section \ref{section_limits}. First, we show that the functor $K_0^\Gamma$ is full when restricted to the category of graded matricial rings over a graded division ring (Proposition \ref{fullness}). Then we prove the main result of section \ref{section_limits}, Theorem \ref{realization_ultrasimplicial}, stating that every ultrasimplicial $\Gamma$-group can be realized by a graded ultramatricial ring over a graded division ring. 

In section \ref{section_dimension_group}, we define a dimension $\Gamma$-group. In the case when $\Gamma$ is trivial, a dimension group $G$ is defined as a partially ordered, directed, unperforated abelian group with interpolation. If $G_1$ is a simplicial group and $g_1: G_1\to G$ an order-preserving homomorphism, then there are a simplicial group $G_2$ and order-preserving homomorphisms $g_{12}: G_1\to G_2$ and $g_2: G_2\to G$ such that $g_2g_{12}=g_1$ and $\ker g_1=\ker g_{12}.$ This statement, sometimes referred to as the {\em Shen criterion}, enables one to inductively construct a directed system of simplicial groups whose direct limit is any given dimension group $G.$ The Shen criterion follows from the {\em Strong Decomposition Property} (SDP). If  $\Delta$ is a subgroup of $\Gamma$, we formulate a generalization (SDP$_\Delta$) of the (SDP) for ordered and directed $\Gamma$-groups (Definition \ref{SDP_Delta_definition}). The presence of the subgroup $\Delta$ in this definition can be explained by the fact that every simplicial $\Gamma$-group comes equipped with a subgroup $\Delta$ of $\Gamma$ which stabilizes its simplicial $\Gamma$-basis.  

We define a dimension $\Gamma$-group as a directed and ordered $\Gamma$-group which satisfies (SDP$_\Delta$) for some subgroup $\Delta$ of $\Gamma$ which is contained in the stabilizer of the group (Definition \ref{dimension_definition}). If $\Gamma$ is trivial, this definition is equivalent with the usual definition. 
With the property of being unperforated appropriately generalized, we show that a dimension $\Gamma$-group is unperforated with respect to a subgroup of $\Gamma$ and  that it has the interpolation property (Propositions \ref{unperforated} and \ref{interpolation}). In Question \ref{question_converse}, we ask whether the converse holds, like in the trivial case. 

An ultrasimplicial $\Gamma$-group $G$ satisfies (SDP$_\Delta$) for some subgroup $\Delta$ of $\Gamma.$ However, any such $\Delta$ may not necessarily be contained in the stabilizer of $G$. If $G$ can be formed via simplicial $\Gamma$-groups with a {\em normal} stabilizer, then $G$ is a dimension $\Gamma$-group (Proposition \ref{ultrasimplicial_is_dimension}). In particular, if $\Gamma$ is abelian, every  ultrasimplicial $\Gamma$-group is a countable dimension $\Gamma$-group. 
We show the converse in the case when $\Zset[\Gamma]$ is Noetherian so, in that case, the statement that ``simplicial groups are building blocks of every dimension group'' still holds. More specifically, with this assumption, Theorem \ref{any_dim_gr} states that every dimension $\Gamma$-group is a direct limit of a directed system of simplicial $\Gamma$-groups with simplicial $\Gamma$-bases stabilized by a normal subgroup of $\Gamma$. Moreover, the limit can be taken in the category of ordered groups with order-units or generating intervals. The proof of Theorem \ref{any_dim_gr} uses our generalized Shen criterion, Proposition \ref{telescoping}, and the proof of Proposition \ref{telescoping} uses the assumption that $\Zset[\Gamma]$ is a Noetherian ring. In Question \ref{question_Noetherian}, we ask whether Theorem \ref{any_dim_gr} holds without this assumption. 

Returning to the question of realizability, we show Theorem \ref{realization_dim_gr} stating that every countable dimension $\Gamma$-group with an order-unit or a generating interval can be realized by a $\Gamma$-graded ultramatricial ring over a $\Gamma$-graded division ring if $\Zset[\Gamma]$ is Noetherian. Thus, if $\Zset[\Gamma]$ is Noetherian, $G$ is a countable dimension $\Gamma$-group exactly when it is realizable by a graded ultramatricial ring over a graded division ring $K$ with normal support (Corollary \ref{realization_corollary}). If $\Gamma$ is a finitely generated abelian group, the class of ultrasimplicial $\Gamma$-groups and the class of countable dimension $\Gamma$-groups coincide (Corollary \ref{realization_corollary_abelian}). 

We also formulate our results for graded rings with involution and for $\Gamma$-$\Zset_2$-bigroups. Thus, our results are applicable to group rings, Leavitt path algebras, and ultramatricial algebras over $\Gamma$-graded fields.   

In section \ref{section_open_problems}, we list some remarks and discuss open questions (including Questions \ref{question_Noetherian} and \ref{question_converse}). We state the {\em Realization Problem} for von Neumann regular rings in terms of graded von Neumann regular rings and we refer to this problem as the {\em $\Gamma$-Realization Problem}. We also note one context in which a pre-ordered $\Gamma$-group structure can arise on the standard $K_0$-group of a ring not the Grothendieck $\Gamma$-group of a graded ring.     

In appendix \ref{appendix_ultramatricial}, we relax the assumption of \cite[Theorem 5.2.4]{Roozbeh_book} that the grade group is abelian and show Theorem \ref{classification} stating that $K_0^\Gamma$ classifies graded ultramatricial algebras over a field graded by an arbitrary group $\Gamma.$  The connecting maps in \cite[Theorem 5.2.4]{Roozbeh_book} are assumed to be unit-preserving while we do not make this assumption. 

In appendix \ref{appendix_extensions}, we show that an ultrasimplicial $\Gamma$-group has a dimension $\Gamma$-group extension by an appropriate one-dimensional simplicial $\Gamma$-group (Proposition \ref{extension_of_ultrasimplicial}) just like in the case when $\Gamma$ is trivial. If $\Zset[\Gamma]$ is Noetherian, the same statement holds for a dimension $\Gamma$-group (Theorem \ref{extension_of_dimension}).

\section{Pre-ordered \texorpdfstring{$\Gamma$}{TEXT}-groups and Grothendieck \texorpdfstring{$\Gamma$}{}-groups}\label{section_Gamma_groups}

Throughout the paper, $\Gamma$ denotes an arbitrary group unless otherwise stated. We use multiplicative notation for the operation of $\Gamma$ and $1_\Gamma,$ or $1$ if it is clear from context, for the identity element. After a short preliminary on $\Gamma$-graded rings, in this section we adapt the theory of Grothendieck $\Gamma$-groups, defined in \cite{Roozbeh_Annalen} in the case when $\Gamma$ is abelian, to the general case when $\Gamma$ is arbitrary. 
We then consider the categories of pre-ordered $\Gamma$-groups and  pre-ordered $\Gamma$-$\Zset_2$-bigroups and the structure of the Grothendieck $\Gamma$-groups of graded rings and graded $*$-rings.

\subsection{Graded rings preliminaries}\label{subsection_preliminaries} 
A ring $R$ is a \emph{$\Gamma$-graded ring} if $R=\bigoplus_{ \gamma \in \Gamma} R_{\gamma}$ for additive subgroups $R_{\gamma}$ of $R$ such that $R_{\gamma}  R_{\delta} \subseteq R_{\gamma\delta}$ for all $\gamma, \delta \in \Gamma$. If it is clear from the context that $R$ is graded by $\Gamma$, $R$ is said to be a graded ring. The elements of $\bigcup_{\gamma \in \Gamma} R_{\gamma}$ are the \emph{homogeneous elements} of $R.$ 
A unital graded ring $R$ is a \emph{graded division ring} if every nonzero homogeneous element has a multiplicative inverse. If a graded division ring $R$ is commutative then $R$ is a {\em graded field}. We also recall that the {\em support} $\Gamma_R$ of a graded ring $R$ is defined by  
\[\Gamma_R=\{\gamma\in\Gamma\,|\, R_\gamma\neq 0\}.\]
A $\Gamma$-graded ring $R$ is \emph{trivially graded} if $\Gamma_R=\{1\}.$ 

We adopt the standard definitions of graded ring homomorphisms and isomorphisms, graded left and right $R$-modules, graded module homomorphisms, graded algebras, graded left and right ideals, graded left and right free and projective modules as defined in \cite{NvO_book} and \cite{Roozbeh_book}. Since we intend to adapt results proven in the case when $\Gamma$ is abelian to the case when $\Gamma$ is arbitrary, we review definitions of some concepts which are particularly left-right sensitive in a bit more details. 

If $M$ is a graded left $R$-module and $\gamma\in\Gamma,$ the $\gamma$-\emph{shifted or $\gamma$-suspended} graded left $R$-module $M(\gamma)$ is defined as the module $M$ with the $\Gamma$-grading given by $$M(\gamma)_\delta = M_{\delta\gamma}$$ for all $\delta\in \Gamma.$ The order of the terms in the product $\delta\gamma$ in this definition is such that $R_\varepsilon M(\gamma)_\delta\subseteq M(\gamma)_{\varepsilon\delta}$ for any $\delta,\varepsilon\in\Gamma$ so that $M(\gamma)$ is indeed a graded left $R$-module for any $\gamma\in\Gamma$. By definition, $M(\gamma)(\delta)=M(\delta\gamma)$ for any $\gamma,\delta\in\Gamma.$ 

Analogously, if $M$ is a graded right $R$-module, the $\gamma$-shifted graded right $R$-module, $(\gamma)M$ is defined as the module $M$ with the $\Gamma$-grading given by $$(\gamma)M_\delta = M_{\gamma\delta}$$ for all $\delta\in \Gamma.$ Thus,  $(\delta)(\gamma)M=(\gamma\delta)M$ for any $\gamma,\delta\in\Gamma.$ 

A graded left module of the form
$R(\gamma_1)\oplus\ldots\oplus R(\gamma_n)$ for $\gamma_1, \ldots,\gamma_n\in\Gamma$ is graded free. Conversely, any finitely generated graded free left $R$-module is of this form. Analogously, every finitely generated graded free right $R$-module is of the form $(\gamma_1)R\oplus\ldots\oplus (\gamma_n)R$ for $\gamma_1, \ldots,\gamma_n\in\Gamma$ (both \cite{NvO_book} and \cite{Roozbeh_book} contain details).  

If $M$ and $N$ are graded left $R$-modules and $\gamma\in\Gamma$, then $\Hom_R(M,N)_\gamma$ usually denotes the following. 
$$\Hom_R(M,N)_\gamma=\{f\in \Hom_R(M, N)\,|\,f(M_\delta)\subseteq N_{\delta\gamma}\}$$ 
If $M$ and $N$ are graded right $R$-modules and $\gamma\in\Gamma$, then $\Hom_R(M,N)_\gamma$ denotes the following. 
$$\Hom_R(M,N)_\gamma=\{f\in \Hom_R(M, N)\,|\,f(M_\delta)\subseteq N_{\gamma\delta}\}$$ 
In both cases, if $M$ is finitely generated (which is the case we will exclusively be interested in), then $\Hom_R(M,N)=\bigoplus_{\gamma\in\Gamma} \Hom_R(M,N)_\gamma$ (both \cite{NvO_book} and \cite{Roozbeh_book} contain details).  

For a $\Gamma$-graded unital ring $R$ and $\gamma_1,\dots,\gamma_n\in \Gamma$, $\M_n(R)(\gamma_1,\dots,\gamma_n),$ or $\M_n(R)(\ol \gamma)$ for short, denotes the ring of matrices $\M_n(R)$ with the $\Gamma$-grading given by  
\begin{center}
$(r_{ij})\in\M_n(R)(\gamma_1,\dots,\gamma_n)_\delta\;\;$ if $\;\;r_{ij}\in R_{\gamma_i\delta\gamma_j^{-1}}$
\end{center}
for $i,j=1,\ldots, n.$ This definition agrees with \cite{NvO_book}. We note that the definition of $\M_n(R)(\gamma_1,\dots,\gamma_n)$ in \cite[Section 1.3]{Roozbeh_book} is different. In \cite{Roozbeh_book}, $\M_n(R)(\gamma_1,\dots,\gamma_n)$ is the ring of matrices $\M_n(R)$ such that $(r_{ij})\in\M_n(R)(\gamma_1,\dots,\gamma_n)_\delta$ if $r_{ij}\in R_{\gamma_i^{-1}\delta\gamma_j}$ for $i,j=1,\ldots, n.$ The definitions are equivalent since
\begin{center}
$\M_n(R)(\gamma_1,\dots,\gamma_n)\;$ from \cite{Roozbeh_book} $\;\;$ is  $\;\;\M_n(R)(\gamma_1^{-1},\dots,\gamma_n^{-1})\;$ from \cite{NvO_book}. 
\end{center}

Using the first definition, 
\[\Hom_R(F,F)\cong_{\gr} \;\M_n(R)(\gamma_1,\dots,\gamma_n)\]
as graded rings where $F$ is the graded free right module $(\gamma_1)R\oplus \dots \oplus (\gamma_n)R.$ Using the second definition, the above formula still holds but for $F=(\gamma_1^{-1})R\oplus \dots \oplus (\gamma_n^{-1})R.$ 

We shall use the definition from \cite{NvO_book} in this paper. Our decision is based on the fact that this definition has been in circulation longer and, consequently, has been used in more publications than the definition from \cite{Roozbeh_book}.   

A \emph{graded matricial ring} over a $\Gamma$-graded ring $R$ is a finite direct sum of graded matrix rings of the form $\M_{n}(R)(\overline \gamma)$ for $\overline \gamma\in  \Gamma^{n}$. A \emph{graded ultramatricial ring} over $R$ is a ring graded isomorphic to a direct limit of a sequence of graded matricial rings over $R$.

It is direct to check that the graded homomorphisms between finitely generated graded free left and right $R$-modules below
\[\bigoplus_{i=1}^n R(\gamma_i^{-1})\to \bigoplus_{j=1}^m R(\delta_i^{-1})\hskip2cm \bigoplus_{j=1}^m (\delta_j)R\to \bigoplus_{i=1}^n (\gamma_i)R \]
are in bijective correspondence with the elements $(r_{ij})$ of $\M_{n\times m}(R)$ such that 
$r_{ij}\in R_{\gamma_i\delta_j^{-1}}$ (see \cite[Section 1.3.4]{Roozbeh_book}). The abelian group of such matrices is denoted by 
$\M_{n\times m}(R)[\gamma_1,\ldots,\gamma_n][\delta_1,\ldots,\delta_m].$

\subsection{Graded rings with involution} 
A ring involution $*$ is an anti-automorphism of order two. 
A ring $R$ with an involution $*$ is said to be an {\em involutive ring} or a {\em $*$-ring}.
In this case, the matrix ring  $\M_n(R)$ also becomes an involutive ring with $(r_{ij})^*=(r_{ji}^*)$ and we refer to this involution as the \emph{$*$-transpose.} 
If a $*$-ring $R$ is also a $K$-algebra for some commutative, involutive ring $K,$ then $R$ is a {\em
$*$-algebra} if $(kr)^*=k^*r^*$ for $k\in K$ and $r\in R.$

A $\Gamma$-graded ring $R$ with involution is said to be a \emph{graded $*$-ring} if $R_\gamma ^*\subseteq R_{\gamma^{-1}}$ for every $\gamma\in \Gamma.$ In this case, the $*$-transpose makes $\M_n(R)(\gamma_1,\dots,\gamma_n)$ into a graded $*$-ring. A graded ring homomorphism which is also a $*$-homomorphism ($f(r^*)=f(r)^*$) is said to be a {\em graded $*$-homomorphism}. 

A \emph{graded matricial $*$-ring} over a $\Gamma$-graded $*$-ring $R$ is graded matricial ring with the $*$-transpose in each coordinate as its involution.  
A \emph{graded ultramatricial $*$-ring} over $R$ is a graded ultramatricial ring which is graded $*$-isomorphic to a direct limit of a sequence of graded matricial $*$-rings over $R$.

\subsection{The Grothendieck \texorpdfstring{$\Gamma$}{TEXT}-group}\label{subsection_Grothendieck}

If $R$ is a $\Gamma$-graded unital ring, let $\V^{\Gamma}_l(R)$ denote the monoid of graded isomorphism classes $[P]$ of finitely generated graded projective left $R$-modules $P$ with $[P]+[Q]=[P\oplus Q]$ as the addition operation. The group $\Gamma$ acts on the monoid $\V^{\Gamma}_l(R)$ by \[(\gamma, [P])\mapsto [P(\gamma)]\] and the action agrees with the addition. We use the term {\em $\Gamma$-monoid} for a monoid with an action of $\Gamma$ which agrees with the monoid structure. Analogously, let 
$\V^{\Gamma}_r(R)$ denote the monoid of graded isomorphism classes $[Q]$ of finitely generated graded projective right $R$-modules $Q$ with the direct sum as the addition operation and the left $\Gamma$-action given by \[(\gamma, [Q])\mapsto [(\gamma^{-1})Q].\]

The formulas 
$(\gamma)\Hom_R(P, R)=\Hom_R(P(\gamma^{-1}), R)$  and $\Hom_R(Q, R)(\gamma)=\Hom_R((\gamma^{-1})Q, R)$
(see \cite[\S 2.4]{NvO_book} or  \cite[\S 1.2.3]{Roozbeh_book}) imply that the operation of taking the dual $\Hom_R( \,\underline{\hskip.3cm}
\,, R)$ induces a $\Gamma$-monoid isomorphism between 
$\V^{\Gamma}_l(R)$ and $\V^{\Gamma}_r(R)$, so we identify these two monoids and use $\V^{\Gamma}(R)$ to denote them.  

In addition, $\V^{\Gamma}(R)$ can also be represented using the equivalence classes of homogeneous idempotent matrices, as has been done in \cite[\S 3.2]{Roozbeh_book}. In particular, the definitions and results of \cite[\S 3.2]{Roozbeh_book} carry to the case when $\Gamma$ is not necessarily abelian as follows. The action of $\gamma$ on an equivalence class of a homogeneous idempotent $p\in \M_n(R)(\gamma_1,\dots,\gamma_n)_1$ is an equivalence class of $p$ when $p$ is considered as an element of $\M_n(R)(\gamma_1\gamma^{-1},\dots,\gamma_n\gamma^{-1})_1.$ This agrees with the above action on $\V^{\Gamma}(R)$ since if $P$ is a finitely generated graded projective  left module which is a direct summand of $\bigoplus_{i=1}^n R(\gamma_i^{-1})$ and $p\in \M_n(R)(\gamma_1,\ldots,\gamma_n)_1$ is a graded homomorphism which corresponds to the projection onto $P,$ then $P(\gamma)$ is a direct summand of $\bigoplus_{i=1}^n R(\gamma_i^{-1})(\gamma)=\bigoplus_{i=1}^n R(\gamma\gamma_i^{-1})$
and $p\in \M_n(R)(\gamma_1\gamma^{-1},\ldots,\gamma_n\gamma^{-1})_1$ corresponds to the projection onto $P(\gamma).$ Thus, we can identify $\V^{\Gamma}(R)$ with the $\Gamma$-monoid of graded equivalence classes of homogeneous idempotent matrices. 

The \emph{Grothendieck $\Gamma$-group}  $K_0^{\Gamma}(R)$ is defined as the group completion of  the $\Gamma$-monoid $\V^{\Gamma}(R)$ which naturally inherits 
the action of $\Gamma$ from $\V^{\Gamma}(R)$. If $\Gamma$ is trivial, $K_0^{\Gamma}(R)$ is the usual $K_0$-group. 

The author prefers to use the terminology ``Grothendieck $\Gamma$-group'' over ``graded Grothendieck group'' used in \cite{Roozbeh_book}. This is because $K_0^\Gamma(R)$ of a graded ring $R$ is not itself graded by $\Gamma$ making the use of ``the graded Grothendieck group'' misleading. Secondly,  if a ring $R$ is naturally graded by two different groups $\Gamma_1$ and $\Gamma_2$, the notation  $K_0^{\gr}(R),$ used in both cases, does not specify which group we are referring to while writing the groups as $K_0^{\Gamma_1}(R)$ and $K_0^{\Gamma_2}(R)$ clearly specifies them. Moreover, a $\Gamma$-action on the Grothendieck group of a ring may arise from other structures, not necessarily grading, as section \ref{subsection_smash} illustrates. For these reasons, we chose to use the notation $K_0^{\Gamma}$ instead of $K_0^{\gr}$. 

We also note that $K_0^{\Gamma}$ is indeed a functor. The formulas $(\gamma)(P\otimes_R Q)=(\gamma)P\otimes_R Q$  and $(P\otimes_R Q)(\gamma)=P\otimes_R Q(\gamma)$ hold for a $\Gamma$-graded ring $R$, a graded right $R$-module $P$ and a graded left $R$-module $Q.$ Thus,  
a graded homomorphism of graded rings $\phi: R\to S$ induces a $\Gamma$-monoid map $\V^{\Gamma}(R)\to \V^{\Gamma}(S)$ given by $[P]\mapsto [P\otimes_R S]$
(respectively $[Q]\mapsto [S\otimes_R Q]$) which extends to a $\Zset[\Gamma]$-module homomorphism $K_0^{\Gamma}(\phi): K_0^{\Gamma}(R)\to K_0^{\Gamma}(S).$ 

\subsection{Graded Morita theory}\label{subsection_Morita}
In \cite[Proposition 2.1.1. and Corollary 2.1.2]{Roozbeh_book}, a Morita equivalence between appropriate categories of $\Gamma$-graded rings and algebras is shown in the case when $\Gamma$ is abelian. In \cite[Remark 2.1.6]{Roozbeh_book}, this equivalence is expanded to the case when $\Gamma$ is not necessarily abelian. We state these claims in the paragraph below. 

If $R$ is a $\Gamma$-graded unital ring and $\ol \gamma = (\gamma_1, \ldots, \gamma_n)$ is in $\Gamma^n,$ then the categories of graded right $R$-modules and the graded right $\M_n(R)(\ol \gamma)$-modules are equivalent. The functors $\Phi$ and $\Psi$ below are mutually inverse functors realizing the equivalence.
\[\Phi: M \mapsto M \otimes_R \bigoplus_{i=1}^n R(\gamma_i^{-1})\hskip1cm \Psi: N \mapsto  N\otimes_{\M_n(R)(\ol \gamma)} \bigoplus_{i=1}^n (\gamma_i)R\] 
In particular, the following graded $R$-bimodules and  graded  $\M_n(R)(\ol\gamma)$-bimodules are graded isomorphic for any $\ol \gamma\in \Gamma^n.$
\[R\cong_{\gr} \bigoplus_{i=1}^n R(\gamma_i^{-1})\otimes_{\M_n(R)(\ol \gamma)} \bigoplus_{i=1}^n (\gamma_i)R\hskip1cm \M_n(R)(\ol\gamma)\cong_{\gr}  \bigoplus_{i=1}^n (\gamma_i)R\otimes_{R}\bigoplus_{i=1}^n R(\gamma_i^{-1})\]
The functors $\Phi$ and $\Psi$ commute with the shift functors so $\Phi$ induces a $\Zset[\Gamma]$-module isomorphism 
$K_0^{\Gamma}(R)\cong K_0^{\Gamma}(\M_n(R)(\ol\gamma))$ 
for any $\Gamma$-graded unital ring $R$ and $\ol \gamma\in \Gamma^n.$ 

The following lemma generalizes \cite[Lemma 5.1.2]{Roozbeh_book} to the case when $\Gamma$ is not necessarily abelian and when a $\Gamma$-graded division ring $K$ is not necessarily commutative. 
\begin{lemma}
If $K$ is a $\Gamma$-graded division ring, $R$ is a graded matricial ring over $K$, and $P$ and $Q$ are two finitely generated graded projective $R$-modules, then 
$P\cong_{\gr} Q$ if and only if $[P]=[Q]$ in $K_0^{\Gamma}(R).$ 
\label{iso_carries_from_K0} 
\end{lemma}
The proof of \cite[Lemma 5.1.2]{Roozbeh_book} uses the equivalence of the categories above and 
\cite[Proposition 3.7.1]{Roozbeh_book}. We formulate this last result in the case when $\Gamma$ and $K$ are not necessarily commutative in section \ref{subsection_division_ring}. This directly translates the proof of \cite[Lemma 5.1.2]{Roozbeh_book} into a proof of Lemma \ref{iso_carries_from_K0}. 

When presented for arbitrary $\Gamma$ using multiplicative notation, \cite[Proposition 1.3.16]{Roozbeh_book} states that for a $\Gamma$-graded unital ring $R$ and $\gamma_1,\ldots,\gamma_n,$ $\delta_1,\ldots,\delta_m\in \Gamma,$ the following conditions are equivalent. 
\begin{enumerate}
\item $(\gamma_1)R\oplus\ldots\oplus (\gamma_n)R \cong_{\gr} (\delta_1)R\oplus\ldots\oplus (\delta_m)R$ as graded right $R$-modules. 
 
\item $R(\gamma_1^{-1})\oplus\ldots\oplus R(\gamma_n^{-1}) \cong_{\gr} R(\delta_1^{-1})\oplus\ldots\oplus R(\delta_m^{-1})$ as graded left $R$-modules. 

\item There are $a\in\M_{n\times m}(R)[\ol\gamma][\ol\delta]$ and $b\in\M_{m\times n}(R)[\ol\delta][\ol\gamma]$ such that $ab=1_{\M_n(R)(\ol\gamma)}$ and $ba=1_{\M_m(R)(\ol\delta)}.$
\end{enumerate}

The proof of \cite[Proposition 1.3.16]{Roozbeh_book} carries over to the case when $\Gamma$ is not necessarily abelian. This proposition implies that  
$(\gamma)R$ is graded isomorphic to $(\delta)R$ if and only if there is an invertible element $a\in R_{\gamma\delta^{-1}}$ (whose inverse is then necessarily in $R_{\delta\gamma^{-1}}).$ 

If $K$ is a $\Gamma$-graded division ring, $\Gamma_K=\{\gamma\in\Gamma \,|\, K_\gamma\neq 0\}$ is a subgroup of $\Gamma.$ So, for any $\gamma,\delta\in\Gamma,$
\begin{center}
$(\gamma)K\cong_{\gr}(\delta)K$ iff $\gamma\delta^{-1}\in \Gamma_K$  iff $\gamma^{-1}\Gamma_K=\delta^{-1}\Gamma_K$  iff $\Gamma_K\gamma=\Gamma_K\delta.$ 
\end{center}

Let us consider $R=\M_n(K)(\gamma_1,\ldots, \gamma_n)$ for any $\gamma_1,\ldots, \gamma_n\in \Gamma$ now. If $e_{ij}$ denotes the $(i,j)$-th standard graded matrix unit of $R$, we have that $[e_{ii}R]\in \V^\Gamma(R)$ corresponds to $[e_{ii}(\bigoplus_{i=1}^n(\gamma_i)K)]=[(\gamma_i)K]\in \V^\Gamma(K)$ under the isomorphism induced by the equivalence $\Phi.$ By the discussion above and Lemma \ref{iso_carries_from_K0}, for any $i,j=1,\ldots, n,$ the following conditions are equivalent. 

\begin{center}
$[e_{ii}R]=[e_{jj}R]\;$ iff $[(\gamma_i)K]=[(\gamma_j)K]$ iff $(\gamma_i)K\cong_{\gr}(\gamma_j)K$ iff\\ 
$\gamma_i\gamma_j^{-1}\in \Gamma_K\;$  iff $\gamma_i^{-1}\Gamma_K=\gamma_j^{-1}\Gamma_K\;$ iff  $\gamma_i^{-1}[K]=\gamma_j^{-1}[K].\;\;\;$
\end{center} 

The relation $(\gamma_i)K=(\gamma_j\gamma_j^{-1}\gamma_i)K=(\gamma_j^{-1}\gamma_i)(\gamma_j)K$ implies that 
\begin{center}
$[(\gamma_i)K]=\gamma_i^{-1}\gamma_j[(\gamma_j)K]$ and $[e_{ii}R]=\gamma_i^{-1}\gamma_j[e_{jj}R]$
\end{center} 
for any $i,j=1,\ldots, n.$ The relation $R=\bigoplus_{i=1}^n e_{ii}R$ and the relations above, imply that
\[[R]=\sum_{i=1}^n[e_{ii}R] = \sum_{i=1}^n\gamma_{i}^{-1}\gamma_{1}[e_{11}R].\]

\subsection{Pre-ordered \texorpdfstring{$\Gamma$}{TEXT}-groups}\label{subsection_ordered_groups}

If $\Gamma$ is a group and $G$ an abelian group, a group homomorphism $\Gamma\to \Hom_{\Zset}(G, G)$ produces a left action of $\Gamma$ on $G.$ This action uniquely determines a left $\Zset[\Gamma]$-module structure on $G.$ In this case, $G$ is a $\Gamma$-group or a left $\Zset[\Gamma]$-module. 

Let $\geq$ be a reflexive and transitive relation (a pre-order) on a $\Gamma$-group $G$ such that $g_1\geq g_2$ implies $g_1 + h\geq g_2 + h$ and $\gamma g_1 \geq \gamma g_2$ for all $g_1, g_2, h\in G$ and $\gamma\in \Gamma.$ We say that such group $G$ is a {\em pre-ordered $\Gamma$-group} or a {\em pre-ordered $\Zset[\Gamma]$-module}. Note that \cite{Roozbeh_book} uses the terminology $\Gamma$-pre-ordered module for this concept. We shall not be using this terminology in order to avoid the implication that $G$ is pre-ordered by $\Gamma$ instead of by a relation $\geq.$
For example, the group ring $\Zset[\Gamma]$ is a pre-ordered $\Gamma$-group with the usual order given by $\sum k_\gamma\gamma\geq \sum m_\gamma\gamma$ if $k_\gamma\geq m_\gamma$ for all $\gamma.$  

If $G$ is a pre-ordered $\Gamma$-group, the set $G^+=\{ x\in G\,|\, x\geq 0\},$ called the cone of $G,$ is a $\Gamma$-monoid. We refer to its elements as the {\em positive elements} of $G$. Any additively closed subset of $G$ which contains 0 and is closed under the action of $\Gamma,$ defines a pre-order $\Gamma$-group structure on $G$. Such set $G^+$ is {\em strict} if $G^+\cap (-G^+)=\{0\}$ and this condition is equivalent with the pre-order being a partial order. In this case, we say that $G$ is an {\em ordered $\Gamma$-group}. If $\Zset^+$ is the set of nonnegative integers, $\Delta$ a subgroup of $\Gamma,$ and $\Gamma/\Delta$ the set of left cosets, the set $\Zset^+[\Gamma/\Delta]$ is a strict cone of the $\Gamma$-group $\Zset[\Gamma/\Delta],$ making $\Zset[\Gamma/\Delta]$ into an ordered $\Gamma$-group.  
 
An element $u$ of a pre-ordered $\Gamma$-group $G$ is an \emph{order-unit} if $u\in G^+$ and for any $x\in G$, there are a positive integer $n$ and $\gamma_1,\dots, \gamma_n \in \Gamma$ (not necessarily distinct) such that $x\leq \sum_{i=1}^n \gamma_i u.$ This condition holds if and only if $u\in G^+$ and for any $x\in G$ there is a nonzero $a\in \Zset^+[\Gamma]$ such that $x\leq au.$ Indeed, one direction follows by taking $a=\sum_{i=1}^n \gamma_i$ and the other by writing $a=\sum k_\gamma\gamma$ as $ \sum_{i=1}^n \gamma_i$ where any term of the form $k_\gamma \gamma$ with $k_\gamma$ positive is written as the sum $\gamma+\gamma+\ldots+\gamma$ of $k_\gamma$ terms. For example, if $\Delta$ is a subgroup of $\Gamma,$ $\Delta$ is an order-unit of the ordered $\Gamma$-group  $\Zset[\Gamma/\Delta]$. Indeed, for $x=\sum k_\gamma(\gamma\Delta)$ one can take $a=\sum |k_\gamma|\gamma$ and have $au-x\geq 0.$ 

If $H$ is a subgroup of a pre-ordered $\Gamma$-group $G$ which is closed under the $\Gamma$-action, then $H$ is also a pre-ordered $\Gamma$-group.  
Such subgroup $H$ is upwards directed if for $x,y\in H,$ there is $z\in H$ such that $x\leq z$ and $y\leq z.$ The property of being downwards directed is defined analogously and, by  \cite[Proposition 15.16]{Goodearl_book}, the following conditions are equivalent: (1) $H$ is upwards directed, (2) $H$ is  downwards directed, (3) $H=(H\cap G^+)+(H\cap -G^+).$ We say that a subgroup which satisfies any of these conditions is {\em directed}.  
The existence of an order-unit implies that $G=G^++(-G^+)$ and so $G$ is directed. 

If $G$ and $H$ are pre-ordered $\Gamma$-groups, a $\Zset[\Gamma]$-module homomorphism $f: G\to H$ is {\em order-preserving} or {\em positive} if $f(G^+)\subseteq H^+.$ If $G$ and $H$ are pre-ordered $\Gamma$-groups with order-units $u$ and $v$ respectively, an order-preserving $\Zset[\Gamma]$-module homomorphism $f: G\to H$ is
{\em order-unit-preserving} (or normalized using the terminology from \cite{Goodearl_interpolation_groups_book}) if  $f(u)=v.$

We let $\POG_\Gamma$ denote the category whose objects are pre-ordered $\Gamma$-groups and whose morphisms are order-preserving $\Zset[\Gamma]$-homomorphisms, we let $\OG_\Gamma$ denote the category whose objects are ordered $\Gamma$-groups and whose morphisms are morphisms of $\POG_\Gamma,$ and we let $\POG^u_\Gamma$ denote the category whose objects are pairs $(G, u)$ where $G$ is an object of $\POG_\Gamma$ and $u$ is an order-unit of $G$, and whose morphisms are morphisms of $\POG_\Gamma$ which are order-unit-preserving. Finally, we let $\OG^u_\Gamma$ denote the category whose objects are pairs $(G, u)$ of $\POG^u_\Gamma$ such that $G$ is also an object of $\OG_\Gamma$ and whose morphisms are morphisms of $\POG^u_\Gamma.$  

The image of the $\Gamma$-monoid $\V^{\Gamma}(R)$ of a $\Gamma$-graded unital ring $R$ under the natural map $\V^{\Gamma}(R)\to K_0^{\Gamma}(R)$ is a cone making $K_0^{\Gamma}(R)$ into a pre-ordered $\Gamma$-group. The element $[R]$ is an order-unit and so $K_0^{\Gamma}(R)$ is directed. Any graded ring homomorphism $\phi$ of graded rings gives rise to $K_0^{\Gamma}(\phi)$ which is a morphism of $\POG_\Gamma$ by section \ref{subsection_Grothendieck}. If $\phi$ is also unit-preserving, then $K_0^{\Gamma}(\phi)$ is a morphism of $\POG_\Gamma^u.$  

\subsection{The Grothendieck \texorpdfstring{$\Gamma$}{TEXT}-group of a graded division ring}\label{subsection_division_ring}
If $K$ is a $\Gamma$-graded division ring, recall that the support $\Gamma_K$ is a subgroup of $\Gamma$. If $\Gamma$ is abelian,
the map
\[\V^{\Gamma}(K)\to \Zset^+[\Gamma/\Gamma_K]\mbox{ given by }\left[(\gamma_1^{-1})K^{p_1}\oplus\ldots\oplus (\gamma_n^{-1})K^{p_n}\right]\mapsto \sum_{i=1}^n p_i(\gamma_i\Gamma_K)\]
is a canonical isomorphism of $\Gamma$-monoids by \cite[Proposition 3.7.1]{Roozbeh_book}. Note that the version of this map representing $\V^{\Gamma}(K)$ using graded left projective modules is given by $\left[K^{p_1}(\gamma_1)\oplus\ldots\oplus K^{p_n}(\gamma_n)\right]\mapsto \sum_{i=1}^n p_i(\gamma_i\Gamma_K).$ 
The proof of \cite[Proposition 3.7.1]{Roozbeh_book} uses \cite[Proposition 1.3.16]{Roozbeh_book}, which we stated for arbitrary $\Gamma$ in section \ref{subsection_Morita}, and the fact that every graded module over a graded division ring is graded free which is proven for arbitrary $\Gamma$ in \cite[Proposition 4.6.1]{NvO_book}. The rest of the proof of \cite[Proposition 3.7.1]{Roozbeh_book} does not use the fact that $\Gamma$ is abelian so it carries to the general case. This $\Gamma$-monoid isomorphism induces an order-preserving $\Gamma$-group isomorphism
\[K_0^{\Gamma}(K)\cong \Zset[\Gamma/\Gamma_K]\]  with an order-preserving inverse. 

\subsection{Pre-ordered bigroups}\label{subsection_ordered_bigroups}
A $\Gamma$-graded ring can have another group act on its Grothendieck $\Gamma$-group besides $\Gamma$. For example, if $R$ is a $\Gamma$-graded $*$-ring, both $\Gamma$ and $\Zset_2$ act on $K_0^{\Gamma}(R).$ 

Let $\Gamma_1$ and $\Gamma_2$ be two groups and let $G$ be an abelian group which is a $\Zset[\Gamma_1]$-$\Zset[\Gamma_2]$-bimodule. We refer to such $G$ as an $\Gamma_1$-$\Gamma_2$-bigroup. Let $\geq$ be a reflexive and transitive relation on a bigroup $G$ such that $g_1\geq g_2$ implies $g_1 + h\geq g_2 + h,$ $\gamma_1 g_1 \geq \gamma_1 g_2,$ and $ g_1\gamma_2 \geq  g_2\gamma_2$ for any $g_1, g_2, h\in G$ and $\gamma_i\in \Gamma_i$ for $i=1,2.$ In this case, we say that $G$ is a {\em pre-ordered $\Gamma_1$-$\Gamma_2$-bigroup}. In this paper, we consider just the case when $\Gamma_2$ is $\Zset_2.$ 
Let us denote the action of the nonzero element of $\Zset_2$ on a group $G$ by $x\mapsto x^*.$ An element $u$ of a pre-ordered $\Gamma$-$\Zset_2$-bigroup $G$ is an \emph{order-unit} if $u$ is an order-unit of $G$ as a $\Gamma$-group and $u^*=u.$
For example, consider the trivial action of $\Zset_2$ and the usual $\Gamma$-action on $\Zset[\Gamma/\Delta]$ for some subgroup $\Delta$ of $\Gamma.$ Then  $\Zset[\Gamma/\Delta]$ is an ordered $\Gamma$-$\Zset_2$-bigroup and $\Delta$ is an order-unit.

Let $G$ be a pre-ordered $\Gamma$-$\Zset_2$-bigroup and let $x\mapsto x^*$ denote the action of the nontrivial element of $\Zset_2.$ Let $\POG^*_\Gamma,$ $\OG^*_\Gamma,$ $\POG^{u*}_\Gamma,$ and $\OG^{u*}_\Gamma$ denote the categories $\POG_\Gamma,$ $\OG_\Gamma,$ $\POG^u_\Gamma,$ and $\OG^u_\Gamma$ with the following additional requirements: the objects are also objects in the category of $\Zset[\Gamma]$-$\Zset[\Zset_2]$-bimodules, the morphisms are also $\Zset[\Gamma]$-$\Zset[\Zset_2]$-bimodule homomorphisms and, for $\POG^{u*}_\Gamma,$ and $\OG^{u*}_\Gamma,$ the action of $\Zset_2$ fixes the order-units.

\subsection{The Grothendieck \texorpdfstring{$\Gamma$-$\Zset_2$}{TEXT}-bigroup}

If $R$ is a $*$-ring, then $*$ is a ring isomorphism mapping $R$ onto the opposite ring $R^{op}$ where $(R^{op}, +)$ is the same as $(R, +)$ and the multiplication is defined by $a\circ b=ba$. If $R$ is a $\Gamma$-graded $*$-ring, then $R^{op}$ is $\Gamma$-graded by $R^{op}_\gamma=R_{\gamma^{-1}}$ and $*$ is a graded ring isomorphism $R\to R^{op}$. Every graded right $R$-module $M$ has the graded left module structure given by $rm=mr^*.$ We let $M^{op}$ denote this graded left module and, analogously, we use $N^{op}$ to denote the graded right module obtained in this way from a graded left module $N$. 
If $M^*=\Hom_R(M,R)$ denotes the dual of a finitely generated graded left or right module $M,$ then 
\[
\Hom_R(M, R)^{op} \cong_{\gr} \Hom_R(M^{op}, R)
\]
which can be shown directly following the definitions, analogously to \cite[Lemma 1.2]{Roozbeh_Lia}. Thus, the map 
$[P]\mapsto [(P^{op})^*]$ defines a $\Zset_2$-action on $\V^{\Gamma}(R).$ The formulas $\Hom_R(P, R)(\gamma)=\Hom_R((\gamma^{-1})P, R)$ and $P^{op}(\gamma)=((\gamma^{-1})P)^{op}$ imply that the actions of $\Gamma$ and $\Zset_2$ commute and so $\V^{\Gamma}(R)$ is a $\Gamma$-$\Zset_2$-bimonoid. 
The $\Gamma$-group $K_0^{\Gamma}(R)$ inherits the $\Gamma$-$\Zset_2$-action from $\V^{\Gamma}(R)$ and becomes a pre-ordered $\Gamma$-$\Zset_2$-bigroup. The $\Zset_2$-action on $\V^{\Gamma}(R),$ represented via the homogeneous idempotents, is given by $[p]^*=[p^*]$ where $p^*$ is the $*$-transpose of the matrix $p\in \M_n(R)(\gamma_1, \ldots, \gamma_n).$ 

If $R$ is a unital $\Gamma$-graded $*$-ring and $F$ a finitely generated graded free $R$-module, then $[F]=[F]^*$ in $K_0^{\Gamma}(R).$ This claim can be shown by noting that the map $R\to (R^{op})^*$ mapping $a$ to the left multiplication by $a^*,$ is a graded module isomorphism which induces a graded isomorphism between $F$ and $(F^{op})^*$ (more details can be found in \cite[Lemma 1.8]{Roozbeh_Lia} which holds when $\Gamma$ is not abelian also). As a corollary of this claim, $[R]$ is an order-unit of $K_0^{\Gamma}(R)$ as a pre-ordered bigroup.
  
In the remark below, we relate the functor ${}^{op}$ to the functor ${}^{(-1)}$ used in \cite{Roozbeh_book} and \cite{Roozbeh_Lia}. Readers not interested in this connection can skip this remark.

\begin{remark}\label{remark_inversely_graded}
If $\Gamma$ is an abelian group and $R$ a $\Gamma$-graded ring, one can define the {\em inversely graded ring} $R^{(-1)}$ as $R$ with a $\Gamma$-grading given by $R^{(-1)}_\gamma=R_{\gamma^{-1}}.$ This corresponds to the $(-1)$-st Veronese ring as considered in \cite[Example 1.1.19]{Roozbeh_book}. If $\Gamma$ is not abelian, this definition does not necessarily define a $\Gamma$-grading of $R^{(-1)}$ since $r\in R_\gamma^{(-1)}=R_{\gamma^{-1}}$ and $s\in R_\delta^{(-1)}=R_{\delta^{-1}}$ implies that $rs\in R_{\gamma^{-1}\delta^{-1}}=R_{\delta\gamma}^{(-1)}$ which can be different from $R_{\gamma\delta}^{(-1)}.$ However, if $\Gamma^{op}$ is the opposite group, then $R^{(-1)}$ is a $\Gamma^{op}$-graded ring. The map 
${}^{-1}:\Gamma\to \Gamma^{op}$ is a group isomorphism which can be used to grade $R^{(-1)}$ by $\Gamma$ so that $R^{(-1)}_\gamma,$ defined as $R^{(-1)}_{\gamma^{-1}}$ via ${}^{-1}:\Gamma\to \Gamma^{op},$ is equal to $R_{\gamma}$ and so the identity is a $\Gamma$-graded ring isomorphism $R\to R^{(-1)}.$    

If $R$ is a $\Gamma$-graded $*$-ring, composing the inverse of the graded ring isomorphism $R\to R^{(-1)}$ with the graded isomorphism $*: R\to R^{op},$ produces a graded ring isomorphism $R^{(-1)}\cong_{\gr} R^{op}.$

The claims of  \cite{Roozbeh_Lia} which involve the functor  ${}^{(-1)}$ for a $*$-ring graded by an abelian group, continue to hold for the functor ${}^{op}$ in the case of a $*$-ring graded by a group which is not necessarily abelian.  
\end{remark}

\section{Simplicial \texorpdfstring{$\Gamma$}{TEXT}-groups}\label{section_simplicial_group}

In this section, we define simplicial $\Gamma$-groups and show that they can be realized by graded matricial $*$-rings over $\Gamma$-graded division $*$-rings.  

In the case when $\Gamma$ is trivial, a simplicial group is an abelian group isomorphic to a direct sum of finitely many copies of $\Zset$ equipped with the usual order. If there are $n$ copies of $\Zset,$ the positive cone is exactly the sum of $n$ copies of $\Zset^+.$ This definition is equivalent to the existence of elements $x_1,\ldots, x_n,$ which constitute a $\Zset$-basis called a simplicial basis. The positive cone is $\Zset^+x_1+\ldots+\Zset^+x_n$ in this case. The Grothendieck group of a matricial ring over a division ring has exactly this type of structure. 

By section \ref{subsection_division_ring}, if $R$ is a graded matricial ring over a graded division ring $K$, then $K_0^{\Gamma}(R)$ is isomorphic to a finite direct sum of $\Zset[\Gamma]$-modules of the form $\Zset[\Gamma/\Gamma_K].$ With this example in mind, we define a simplicial $\Gamma$-group as follows. 

Let $\Delta$ be a subgroup of $\Gamma$, let $\Gamma/\Delta$ denote the set of left cosets, and let $\Zset[\Gamma/\Delta]$ denote the permutation module on which $\Gamma$ acts by $(\gamma, \delta\Delta)\mapsto \gamma\delta\Delta.$ Both $\Zset[\Gamma]$ and $\Zset[\Gamma/\Delta]$ are directed ordered $\Gamma$-groups with the cones  $\Zset^+[\Gamma]$ and $\Zset^+[\Gamma/\Delta]$ respectively. Let also 
\[\pi:\Zset[\Gamma]\to \Zset[\Gamma/\Delta]\]
be the natural left $\Zset[\Gamma]$-module map induced by $1\mapsto \Delta$. Note that $\pi$ is order-preserving. Let $J$ be the set indexing the coset representatives $\gamma_j.$ The kernel of $\pi$ is a left $\Zset[\Gamma]$-ideal 
consisting of $x\in\Zset[\Gamma]$ such that there is a finite subset $J_0$ of $J$ and a finite subset $\Gamma_j$ of $\gamma_j\Delta$ such that 
$$x=\sum_{j\in J_0}\sum_{\gamma\in\Gamma_j} k_\gamma\gamma\mbox{ with }\sum_{\gamma\in\Gamma_j} k_\gamma=0\mbox{ for all }j\in J_0.$$

If $G$ is a $\Gamma$-group and $X\subseteq G$, we let \[\Stab(X) = \{\gamma\in\Gamma\,|\, \gamma x=x\;\;\mbox{ for all }\;\;x\in X\}.\]

\begin{definition}
If $G$ is a $\Gamma$-group, the set $X=\{x_1,\ldots,x_n\}\subseteq G$ is a {\em simplicial $\Gamma$-basis} of $G$ if 
the conditions (Stab) and (Ind) below hold. 
\begin{enumerate}
\item[(Stab)] $\Stab(x_i)=\Stab(x_j)$ for every $i,j=1,\ldots,n.$ 
\end{enumerate}
If $\Delta$ denotes $\Stab(x_i)=\Stab(X)$ for any $i,$ we say that $X$ is {\em stabilized} by $\Delta.$
\begin{enumerate}
\item[(Ind)] For every $a_i, b_i\in \Zset[\Gamma], i=1,\ldots, n$ with $\pi(a_i), \pi(b_i)\in \Zset^+[\Gamma/\Delta],$
\[\sum_{i=1}^n a_ix_i=\sum_{i=1}^n b_ix_i\;\;\mbox{ if and only if }\;\;a_i-b_i\in \ker (\pi:\Zset[\Gamma]\to \Zset[\Gamma/\Delta])\mbox{ for all }i=1,\ldots, n.\]
\end{enumerate}
In this case, we let 
\begin{enumerate}
\item[] $G^+=\{\sum_{i=1}^n a_ix_i\,|\, a_i\in \Zset[\Gamma]$ such that $\pi(a_i)\in \Zset^+[\Gamma/\Delta]$ for all $i=1,\ldots, n\}$ 
\end{enumerate}
and call $G^+$ the {\em simplicial $\Gamma$-cone} of the group $G$. We also let $\varnothing$ be the simplicial $\Gamma$-basis of the group $G=\{0\}$ and we let $G^+=\{0\}$ be the simplicial $\Gamma$-cone in this case.  
\label{simplicial_cone_definition}
\end{definition}

It is direct to check that the set $G^+$ as above contains 0, that it is additively closed and closed under the action of $\Gamma.$ So, the use of the terminology ``cone'' is justified. We also note that a simplicial $\Gamma$-cone is strict. Indeed, if we  assume that $\sum_{i=1}^n a_ix_i$ and $\sum_{i=1}^n -a_ix_i$ are both in $G^+,$ then $\pi(a_i), \pi(-a_i)\in \Zset^+[\Gamma/\Delta]$ for all $i=1,\ldots, n.$ Since the positive cone of $\Zset[\Gamma/\Delta]$ is strict, we have that $\pi(a_i)=0$ for all $i$ which, in turn, implies that $\sum_{i=1}^n a_ix_i=0$ by (Ind). Since a simplicial $\Gamma$-cone is strict, it determines a partial order on the $\Gamma$-group $G$ as in Definition \ref{simplicial_cone_definition}.  

\begin{definition} 
A $\Gamma$-group $G$ is a {\em simplicial $\Gamma$-group (or a simplicial $\Zset[\Gamma]$-module)} if either $G=\{0\}$ or if there is a finite simplicial $\Gamma$-basis  $\varnothing \neq X\subseteq G$ such that the simplicial $\Gamma$-cone $G^+,$ defined as in Definition \ref{simplicial_cone_definition}, orders $G$ so that $G$ is directed under the order induced by $G^+.$ 
\label{simplicial_definition}
\end{definition}

We misuse the term basis here since we neither claim the number of basis elements to be uniquely determined nor that every basis is stabilized by the same subgroup. We use this term, however, to match the existing terminology in the case when $\Gamma$ is trivial. In the case when $\Delta$ is a normal subgroup of $\Gamma$, the use of the term basis is justified since a simplicial $\Gamma$-basis corresponds to a basis of a free $\Zset[\Gamma/\Delta]$-module as we show in Proposition \ref{Delta_normal_case}. If $\Gamma$ is abelian, both the number of basis elements as well as the subgroup stabilizing a basis are unique as Proposition \ref{Delta_normal_case} also shows. 

\begin{example}\label{example_free_permutation_module}
(1) For any subgroup $\Delta$ of $\Gamma,$
it is direct to check that a direct sum $F$ of $n$ copies of the permutation module $\Zset[\Gamma/\Delta]$ is a simplicial $\Gamma$-group with a simplicial $\Gamma$-basis $\{e_1,\ldots, e_n\},$ where $e_i=(0, \ldots, 0, \Delta, 0,\ldots, 0)$ with the nonzero element $\Delta$ at the $i$-th place. Indeed, every element of the form $\sum_{i=1}^n \ol a_i e_i$ where $\ol a_i\in \Zset^+[\Gamma/\Delta]$ is equal to $\sum_{i=1}^n a_i e_i$ where $a_i$ is any element of $\pi^{-1}(\ol a_i).$ Conversely, every linear combination of the form 
$\sum_{i=1}^n a_i e_i$ with $a_i\in\Zset[\Gamma]$ and $\pi(a_i)\in \Zset^+[\Gamma/\Delta]$ is equal to $\sum_{i=1}^n \pi(a_i) e_i\in \bigoplus_{i=1}^n \Zset^+[\Gamma/\Delta].$ The condition (Ind) holds since the sum is direct. 

If $\Delta$ is normal, then $\Zset[\Gamma/\Delta]$ is a ring, the module $F$ is a free $\Zset[\Gamma/\Delta]$-module, and $\{e_1,\ldots, e_n\}$ is a $\Zset[\Gamma/\Delta]$-module basis of $F.$  

(2) If $G$ is a simplicial $\Gamma$-group with a simplicial $\Gamma$-basis $X,$ then $X\subseteq G$ implies $\Stab(G)\subseteq \Stab(X).$ If 
$\Gamma=D_3=\langle a,b|a^3=b^2=1, ba=a^2b\rangle,$ $\Delta=\{1,b\}$ and $G=\Zset[\Gamma/\Delta],$ then $G$ is a simplicial $\Gamma$-group with a simplicial $\Gamma$-basis $X=\{\Delta\}$ and $\Stab(G)=\{1\}\subsetneq\Stab(X)=\Delta.$ Proposition \ref{Delta_normal_case} shows that if $\Stab(X)$ is normal, then $\Stab(G)=\Stab(X).$
\end{example}

We prove several properties of a simplicial $\Gamma$-group.

\begin{proposition}\label{properties_of_simplicial}
Let $G$ be a simplicial $\Gamma$-group  with a simplicial $\Gamma$-basis $X=\{x_1, \ldots, x_n\}$ and let $\Delta=\Stab(X).$    
\begin{enumerate}
\item Every element of $G$ can be represented as $\sum_{i=1}^n a_i x_i$ for some $a_i\in \Zset[\Gamma]$ so we say that $X$ {\em generates} $G.$

\item If $a_i, b_i\in \Zset[\Gamma],$  $i=1,\ldots,n,$ then 
\[\sum_{i=1}^na_ix_i=\sum_{i=1}^nb_ix_i\;\;\mbox{ if and only if }\;\;\pi(a_i)=\pi(b_i)\mbox{ for all }i=1,\ldots, n.\]

\item $$G^+=\{\sum_{i=1}^n a_ix_i\; |\; a_i\in \Zset^+[\Gamma]\;\mbox{ for all }\;i=1,\ldots, n\}$$

\item The element $u=\sum_{i=1}^n x_i$ is an order-unit of $G.$ 

\item The group $G$ is isomorphic to a direct sum $F$ of $n$ copies of $\Zset[\Gamma/\Delta]$ as an ordered $\Gamma$-group. 
\end{enumerate}
\end{proposition}
\begin{proof}
The first claim follows from the assumption that $G$ is directed. Indeed, if $x\in G,$ then $x=x^+-x^-$ for some $x^+, x^-\in G^+.$ By the definition of $G^+,$ one can find $a_i^+$ and $a^-_i\in \Zset[\Gamma]$ such that $\pi(a_i^+), \pi(a_i^-)\in \Zset^+[\Gamma/\Delta]$ for $i=1,\ldots, n$ and such that $x^+=\sum_{i=1}^n a_i^+x_i$ and $x^-=\sum_{i=1}^n a_i^-x_i.$ Letting $a_i=a_i^+-a^-_i,$ we have that $x=x^+-x^-=\sum_{i=1}^n a_ix_i.$ 

To show the second claim, note first that every element of $a\in\Zset[\Gamma]$ can be represented as $a=a^+-a^-$ where $a^+, a^-\in \Zset^+[\Gamma]$ and, hence 
$\pi(a^+), \pi(a^-)\in \Zset^+[\Gamma/\Delta].$ So, for $a_i, b_i\in \Zset[\Gamma],$ $i=1, \ldots, n,$ let $a_i=a_i^+-a_i^-$ and $b_i=b_i^+-b_i^-$ where $a_i^+,a_i^-,b_i^+,b_i^-\in \Zset^+[\Gamma]$ for all $i.$ Then we have the following by (Ind).
\[
\begin{array}{lcl}
\sum_{i=1}^na_ix_i=\sum_{i=1}^nb_ix_i &\Leftrightarrow &\sum_{i=1}^n(a_i^+-a_i^-)x_i=\sum_{i=1}^n(b_i^+-b_i^-)x_i\\
&\Leftrightarrow & \sum_{i=1}^n(a_i^++b_i^-)x_i=\sum_{i=1}^n(b_i^++a_i^-)x_i\\
&\Leftrightarrow & \pi(a_i^++b_i^-)=\pi(b_i^++a_i^-)\mbox{ for all }i=1,\ldots, n \\
&\Leftrightarrow & \pi(a_i^+)-\pi(a_i^-) =\pi(b_i^+)-\pi(b_i^-)\mbox{ for all }i=1,\ldots, n\\
&\Leftrightarrow & \pi(a_i)=\pi(b_i)\mbox{ for all }i=1,\ldots, n.
\end{array}
\]

To show the third claim, let $C$ denote the set on the right-hand side of the equality in condition (3). The inclusion $C\subseteq G^+$ follows from Definition \ref{simplicial_cone_definition}. For the converse, let $x\in G^+$ so that $x=\sum_{i=1}^n a_ix_i$ where $\pi(a_i)\in \Zset^+[\Gamma/\Delta].$ Let $\{\gamma_j\, |\, j\in J\}$ be the set of coset representatives, let $J_i$ be finite subsets of $J,$ and let $\Gamma_{ji}$ be finite subsets of $\gamma_j\Delta$ for $i=1,\ldots,n$ such that $a_i=\sum_{j\in J_i}\sum_{\gamma\in\Gamma_{ij}}k_\gamma\gamma.$ The condition $\pi(a_i)\in \Zset^+[\Gamma/\Delta]$ implies that $k_j=\sum_{\gamma\in\Gamma_{ij}}k_\gamma$ is a nonnegative integer for every $j\in J_i.$ Let $b_i=\sum_{j\in J_i}k_j\gamma_j.$ Then $b_i\in \Zset^+[\Gamma]$ by definition and $\pi(b_i)=\pi(a_i)=\sum_{j\in J_i}k_j\gamma_j\Delta.$ The condition (Ind) then implies that $x=\sum_{i=1}^n a_ix_i=\sum_{i=1}^n b_ix_i\in C.$ 

To show the fourth claim, let $x\in G$ and let $x=\sum_{i=1}^n a_ix_i$ for some $a_i\in \Zset[\Gamma]$ which exist by part (1). For each $a_i=\sum_{j=1}^{l_i} k_{ij}\gamma_{ij},$ let $k$ be the maximum of the absolute values of $k_{ij}$ for all $i=1,\ldots, n$ and all $j=1,\ldots, l_i$ and let $a=\sum_{i=1}^n\sum_{j=1}^{l_i}k\gamma_{ij}.$ Then $a_i\leq a$ for all $i=1,\ldots, n$ so $a-a_i\in \Zset^+[\Gamma]$ which implies that $ax_i-a_ix_i\in G^+$ for all $i=1,\ldots,n.$ Thus, $au-x=\sum_{i=1}^n(ax_i-a_ix_i)\in G^+.$  

To show the last claim, let $f: F\to G$ be defined by $\sum_{i=1}^n a_ie_i\mapsto \sum_{i=1}^n a_ix_i$ where $\{e_1,\ldots, e_n\}$ is the simplicial $\Gamma$-basis of $F$ given in Example \ref{example_free_permutation_module}. Then $f$ is well-defined and injective by part (2), surjective by part (1), a $\Zset[\Gamma]$-module map by definition, and order-preserving with the order-preserving inverse by the definition of $F^+$ and $G^+.$  
\end{proof}

The next proposition deals with two special cases.  

\begin{proposition} 
\begin{enumerate}
\item If $\Delta$ is a normal subgroup of $\Gamma$ and  $G$ is a simplicial $\Gamma$-group with a simplicial $\Gamma$-basis $X=\{x_1,\ldots, x_n\}$ and $\Stab(X)=\Delta,$ then $\Stab(G)=\Delta,$ $\pi$ is both a left and a right $\Zset[\Gamma]$-module homomorphism, $G$ is a $\Zset[\Gamma/\Delta]$-module by $\pi(a)x=ax,$ and $G$ is isomorphic to a direct sum $F$ of $n$ copies of $\Zset[\Gamma/\Delta]$ as a $\Zset[\Gamma/\Delta]$-module.
 
\item If $\Gamma$ is abelian then every simplicial $\Gamma$-basis of a simplicial $\Gamma$-group $G$ is stabilized by the same subgroup of $\Gamma$ and has the same number of elements.  
\end{enumerate}
\label{Delta_normal_case} 
\end{proposition}
\begin{proof}
Since $X\subseteq G$ implies $\Stab(G)\subseteq \Stab(X)=\Delta,$ let us show the inverse inclusion. Let $\gamma\in\Delta$ and let $x=\sum_{i=1}^na_ix_i$ for some $a_i\in \Zset[\Gamma].$ Then $\gamma x=\sum_{i=1}^n\gamma a_ix_i=\sum_{i=1}^n a_i\gamma_ix_i$ for some $\gamma_i\in \Delta$ for $i=1,\ldots, n$ since $\Delta$ is normal. Thus, $\gamma_i\in \Delta$ implies that $\gamma_ix_i=x_i$ and so $\gamma x=\sum_{i=1}^n a_i\gamma_ix_i=\sum_{i=1}^n a_ix_i=x.$  

The map $\pi$ is also a right $\Zset[\Gamma]$-module map since for every $a\in \Zset[\Gamma]$ and $\gamma\in\Gamma,$ $\pi(a\gamma)=a\gamma\Delta=a\Delta\gamma=\pi(a)\gamma.$ This implies that the map $\Zset[\Gamma/\Delta]\times G\to G$ given by $\pi(a)x=ax$ for any $x\in G$ and $a\in \Zset[\Gamma]$ is well-defined since if $x=\sum_{i=1}^n a_ix_i$ and  $\pi(a)=\pi(b)$ for some $b\in \Zset[\Gamma],$ then $\pi(aa_i)=\pi(a)a_i=\pi(b)a_i=\pi(ba_i)$ so that  $\pi(a)x=\sum_{i=1}^n aa_ix_i=\sum_{i=1}^n ba_ix_i=\pi(b)x$ by part (2) of Proposition \ref{properties_of_simplicial}. This map induces a $\Zset[\Gamma/\Delta]$-module structure on $G.$ It is direct to check that the isomorphism $f$ from the proof of part (5) of Proposition \ref{properties_of_simplicial} is a $\Zset[\Gamma/\Delta]$-module isomorphism. 

To prove part (2), let us assume that $G$ has simplicial $\Gamma$-bases $X$ and $Y$ stabilized by $\Delta_1$ and $\Delta_2$ respectively. By part (1), $\Delta_2=\Stab(G)=\Delta_1.$  The commutative and unital ring $\Zset[\Gamma/\Stab(G)]$ has the invariant basis property and so $X$ and $Y$ have the same number of elements.   
\end{proof}

\subsection{The Grothendieck \texorpdfstring{$\Gamma$}{TEXT}-group of a graded matricial ring over a graded division ring}\label{Grothendieck_of_matricial}

Let $K$ be a $\Gamma$-graded division ring, let $n$ and $p(i)$ be positive integers, and let $\gamma_{i1}, \ldots,\gamma_{ip(i)}\in \Gamma,$ for $i=1,\ldots,n.$ Let $$R=\bigoplus_{i=1}^n\M_{p(i)}(K)(\gamma_{i1},\dots,\gamma_{ip(i)})$$ be a graded matricial ring, let $\pi_i$ be the projection of $R$ onto the $i$-th component, and let $e_{kl}^i\in R$ be such that $\{\pi_i(e_{kl}^i)\,|\, k,l=1,\ldots, p(i)\}$ is the set of the standard graded matrix units of $\pi_i(R).$ 
By sections \ref{subsection_Morita} and \ref{subsection_division_ring}, 
\[\V^{\Gamma}(R)\cong \bigoplus_{i=1}^n\V^{\Gamma}(K)\cong \bigoplus_{i=1}^n \Zset^+[\Gamma/\Gamma_K]\] as $\Gamma$-monoids and, 
under these isomorphisms,
$$[e_{kk}^iR] \mapsto (0, \ldots, 0, [(\gamma_{ik})K], 0,\ldots 0)\mapsto (0, \ldots, 0, \gamma_{ik}^{-1}\Gamma_K, 0,\ldots 0)$$ for all $i=1,\ldots, n$ and $k=1,\ldots, p(i).$ Thus, $\gamma_{i1}[e_{11}^iR]=\gamma_{i1}\gamma_{i1}^{-1}\gamma_{ik}[e_{kk}^iR]=\gamma_{ik}[e_{kk}^iR]$ (see section \ref{subsection_Morita}) corresponds to $(0, \ldots, 0, \Gamma_K, 0,\ldots 0).$  
Extend these $\Gamma$-monoid isomorphisms to order-preserving $\Gamma$-group isomorphisms $$K_0^{\Gamma}(R)\cong \bigoplus_{i=1}^n K_0^{\Gamma}(K)\cong \bigoplus_{i=1}^n \Zset[\Gamma/\Gamma_K].$$ 

The $\Gamma$-group $\bigoplus_{i=1}^n \Zset[\Gamma/\Gamma_K]$ is a simplicial $\Gamma$-group and the set from Example \ref{example_free_permutation_module} is a simplicial $\Gamma$-basis stabilized by $\Gamma_K$. Thus, the $\Gamma$-group $K_0^{\Gamma}(R)$ is also a simplicial $\Gamma$-group and the set 
\[\{\gamma_{11}[e_{11}^1R],\ldots,\gamma_{n1}[e_{11}^nR]\}\]
is a simplicial $\Gamma$-basis stabilized by $\Gamma_K.$ 

Let us partition the ordered list $\gamma_{i1}, \ldots, \gamma_{ip(i)}$  such that the inverses of the elements of the same partition part belong to the same left coset of $\Gamma/\Gamma_K$ and the inverses of the elements from different partition parts belong to different cosets. Let $\gamma_{il1},\dots,\gamma_{ilr_{il}}$ be the $l$-th part for $l=1, \ldots, k_i$ where $k_i$ is the number of partition parts and $r_{il}$ is the number of elements in the $l$-th part of the $i$-th ordered list  $\gamma_{i1}, \ldots, \gamma_{ip(i)}$. Note that we can have $\gamma_{i1}=\gamma_{i11}.$ Thus, we partition $\gamma_{i1},\dots,\gamma_{ip(i)}$ as 
\[\gamma_{i11},\dots,\gamma_{i1r_{i1}},\; \gamma_{i21},\dots,\gamma_{i2r_{i2}},\; \;\;\;\ldots,\;\;\;\; \gamma_{ik_i1},\dots,\gamma_{ik_ir_{ik_i}}\]
and we have that $\sum_{l=1}^{k_i} r_{il}=p(i).$ 

By \cite[Remark 2.10.6]{NvO_book}, the graded rings $\M_{p(i)}(K)(\gamma_{i1}, \ldots, \gamma_{ip(i)})$ and $\M_{p(i)}(K)(\gamma_{i\sigma(1)}, \ldots, \gamma_{i\sigma(p(i))})$ are graded isomorphic for any permutation $\sigma$ of the set $\{1,\ldots, p(i)\}.$ By \cite[Proposition 1.3 (1)]{Roozbeh_Lia},
these rings are also graded $*$-isomorphic. Hence, reordering the elements of $\Gamma$ as above does not change the graded $*$-isomorphism class of the ring. This produces the following relation. 
\[[R]=\sum_{i=1}^n\sum_{k=1}^{p(i)}[e^i_{kk}R] = 
\sum_{i=1}^n\sum_{k=1}^{p(i)}\gamma_{ik}^{-1}\gamma_{i1}[e^i_{11}R]
=\sum_{i=1}^n\sum_{l=1}^{k_i} r_{il} \gamma_{il1}^{-1}\gamma_{i1}[e^i_{11}R]\]

Let $K$ be any division ring now and let $\Delta$ be a subgroup of $\Gamma$. Consider $K$ to be trivially graded by $\Gamma$ and consider the $\Gamma$-grading of the group ring $K[\Delta]$ given by $K[\Delta]_\gamma=K\{\gamma\}$ if $\gamma\in\Delta$ and $K[\Delta]_\gamma=0$ otherwise. This makes $K[\Delta]$ into a $\Gamma$-graded division ring with $\Gamma_{K[\Delta]}=\Delta$. The proof of the proposition below follows from the rest of this section. 

\begin{proposition} Let $\Gamma$ be a group with a subgroup $\Delta,$ let $K$ be a division ring, let $n$ and $p(i)$ be positive integers, and let $\gamma_{i1},\ldots,\gamma_{ip(i)}\in\Gamma$ for $i=1,\ldots, n.$ If $$R=\bigoplus_{i=1}^n \M_{p(i)}(K[\Delta])(\gamma_{i1},\ldots,\gamma_{ip(i)}),$$ $K_0^{\Gamma}(R)$ is a simplicial $\Gamma$-group with a simplicial $\Gamma$-basis $\{\gamma_{11}[e^1_{11}R], \ldots, \gamma_{n1}[e^n_{11}R]\}$ stabilized by $\Delta.$    
\label{simplicial_exists}
\end{proposition}

\subsection{Realization of a simplicial  \texorpdfstring{$\Gamma$}{TEXT}-group}

We show the converse of Proposition \ref{simplicial_exists}: every simplicial  $\Gamma$-group can be realized as the Grothendieck $\Gamma$-group of a graded matricial ring over a graded division ring. 

\begin{theorem}
If $\Gamma$ is a group and $G$ is a simplicial $\Gamma$-group with an order-unit $u$ and a simplicial $\Gamma$-basis $\{x_1,\ldots, x_n\}$ stabilized by $\Delta,$ then there is a $\Gamma$-graded division ring $K,$ a $\Gamma$-graded matricial ring $R$ over $K,$ and an isomorphism $f$ of $\OG^u_\Gamma$ such that $$f: (K_0^{\Gamma}(R), [R])\cong (G, u).$$ 
If we consider $G$ as an ordered $\Gamma$-$\Zset_2$-bigroup with the trivial $\Zset_2$-action, then $K$ (thus also $R$) can be given a structure of a graded $*$-ring and $f$ can be found in $\OG^{u*}_\Gamma.$   
\label{realization_simplicial}
\end{theorem}
\begin{proof}
By part (3) of Proposition \ref{properties_of_simplicial}, we can write $u$ as $\sum_{i=1}^n a_ix_i$ for some $a_i\in\Zset^+[\Gamma]$. 
Let us write $a_i$ as $\sum_{l=1}^{k_i}\sum_{t=1}^{m_i} k_{ilt}\gamma^{-1}_{ilt}$ where $\gamma^{-1}_{ilt}\Delta=\gamma^{-1}_{ilt'}\Delta$ for all $t,t'=1,\ldots m_i,$ $\gamma^{-1}_{ilt}\Delta\neq \gamma^{-1}_{il't'}\Delta$ for all $l\neq l'$ when $l,l'\in \{1,\ldots, k_i\},$ and $k_{ilt}$ are nonnegative integers.  
Since $\gamma^{-1}_{ilt}x_i=\gamma^{-1}_{ilt'}x_i$ for any  $t,t'=1,\ldots m_i,$ we can replace $a_i$ by $\sum_{l=1}^{k_i}\sum_{t=1}^{m_i} k_{ilt}\gamma^{-1}_{il1}.$ 
If we let $r_{il}$ denote $\sum_{t=1}^{m_i} k_{ilt},$ we have that $a_i=\sum_{l=1}^{k_i}r_{il}\gamma^{-1}_{il1}.$

Let $p(i)=\sum_{l=1}^{k_i} r_{il}.$ Since $r_{il}\geq 0$ for all $i$ and all $l,$ $p(i)\geq 0.$ We claim that for all $i,$ there is $l\in \{1, \ldots, k_i\}$ such that $r_{il}>0$ which implies that $p(i)>0.$ Assume the opposite, that for some $i\in\{1,\ldots, n\}$ $r_{il}=0$ for all $l=1,\ldots, k_i.$ This implies that $a_i=0$ and hence  we can write $u=\sum_{j=1, j\neq i}^n a_jx_j.$ Let $b\in\Zset^+[\Gamma]$ be such that $x_i\leq bu=\sum_{j=1, j\neq i}^n ba_jx_j$  so that $\sum_{j=1, j\neq i}^n ba_jx_j-1x_i\geq 0$. Definition \ref{simplicial_cone_definition} implies that $\pi(-1)=\pi(-1(1_\Gamma))=-1\Delta\geq 0$ which is a contradiction. Hence $p(i)>0.$

Let $F$ be any field, let $K=F[\Delta],$ and let $\ol{\gamma_{il1}}$ denote the list $\gamma_{il1}, \ldots, \gamma_{il1}$ of $\gamma_{il1}$ repeated $r_{il}$ times for $i=1,\ldots, n$ and $l=1,\ldots, k_i.$ Let
\[R= \bigoplus_{i=1}^n \M_{p(i)}(K)(\ol{\gamma_{i11}}, \ol{\gamma_{i21}}, \ldots, \ol{\gamma_{ik_i1}}).\] 

The correspondence $f: \gamma_{i11}[e^i_{11}R]\mapsto x_i$ extends to an isomorphism $K_0^{\Gamma}(R)\to G$ of $\OG_\Gamma$ by Definition \ref{simplicial_cone_definition} and Proposition \ref{properties_of_simplicial}. This correspondence is order-unit-preserving since 
$$f([R])=f\left(\sum_{i=1}^n\sum_{l=1}^{k_i} r_{il} \gamma_{il1}^{-1}\gamma_{i11}[e^i_{11}R]\right)=\sum_{i=1}^n\sum_{l=1}^{k_i} r_{il} \gamma_{il1}^{-1}x_i=\sum_{i=1}^n a_i x_i=u.$$ 

To prove the last sentence of the theorem, consider any involution $^*$ on $F$ (one such involution, the identity, always exists) and the standard involution on $K=F[\Delta]$ i.e. $(a\delta)^*=a^*\delta^{-1}$ for $a\in F$ and $\delta\in\Delta.$ This involution makes $K$ into a graded division $*$-ring and any matricial ring over $K$ into a graded $*$-ring. The action of $\Zset_2$ induced by the involution is trivial on the $K_0^{\Gamma}$-group of any matricial ring over a graded division $*$-ring by \cite[Proposition 1.9]{Roozbeh_Lia} and Remark \ref{remark_inversely_graded}, so the last sentence follows from the rest.
\end{proof}

For the proof of Theorem \ref{countable_dim_gr}, we need to adapt the concept of a convex subset and an ideal of a pre-ordered group to our setting. 
If $G$ is a pre-ordered $\Gamma$-group, a subset $H$ of $G$ is {\em convex} if  $x\leq z\leq y$ for $x,y \in H$ and $z\in G,$ implies $z\in H$. It is direct to check that the intersection of convex $\Gamma$-subgroups of $G$ is again a convex $\Gamma$-subgroup. So, for any subset $X$ of $G,$ one can define the convex $\Gamma$-subgroup generated by $X.$ For any $u\in G^+,$ one checks directly that 
\[H=\{x\in G \,|\, -au\leq x\leq au\mbox{ for }a\in \Zset^+[\Gamma]\}\]
is a convex $\Gamma$-subgroup and that $u$ is its order-unit. Moreover, $H$ is the smallest convex $\Gamma$-subgroup which contains $u$ and we say that $H$ is the convex $\Gamma$-subgroup generated by $\{u\}.$

A {\em $\Gamma$-ideal} of a pre-ordered $\Gamma$-group $G$ is any directed and convex $\Gamma$-subgroup of $G$. 

\begin{proposition}
If $G$ is a simplicial $\Gamma$-group with a simplicial $\Gamma$-basis $X$ stabilized by $\Delta,$ then $H$ is a $\Gamma$-ideal of $G$ if and only if $H$ is a $\Gamma$-subgroup of $G$ generated by a subset of $X.$

As a corollary, if $H$ is a $\Gamma$-ideal of $G,$ then there are subsets $Y$ and $Z$ of $X$ such that $Y$ is a simplicial $\Gamma$-basis of $H$ and $Z$ a simplicial $\Gamma$-basis of $G/H.$  
\label{ideals_in_simplicial}
\end{proposition}
\begin{proof}
The proof parallels \cite[Proposition 3.8 and Corollary 3.9]{Goodearl_interpolation_groups_book} with some additional details. 

To prove the first sentence, let $X=\{x_1, \ldots, x_n\}.$ For any subset $Y$ of $X$, let $$H=\{x\in G\; |\; x=\sum_{i=1}^n a_ix_i,\; a_i\in \Zset[\Gamma] \mbox{ and }\pi(a_i)=0\mbox{ for all } i\mbox{ such that }x_i\notin Y\}.$$ 
The description of $H$ is independent from the representation of an element of $G$ as a linear combination of the basis elements since if $x=\sum_{i=1}^n a_ix_i=\sum_{i=1}^n b_ix_i$ and $\pi(a_i)=0$ for each $i$ such that $x_i\notin Y,$ then $\pi(b_i)=\pi(a_i)=0$ for such $i$ also by part (2) of Proposition \ref{properties_of_simplicial}. It is direct to check that $H$ is convex and directed. By the definition of a simplicial $\Gamma$-basis, $H$ can also be described as $$\{x\in G\; |\; x=\sum_{x_i\in Y} a_ix_i,\; a_i\in \Zset[\Gamma] \}$$
and from this description it is clear that $H$ is a $\Gamma$-subgroup of $G$ generated by $Y.$

Conversely, if $H$ is a $\Gamma$-ideal of $G$, let $Y=\{x_i\, |\, x_i\in H\}.$ It is clear that the $\Gamma$-subgroup $K$ of $G$ generated by $Y$ is contained in $H.$ Suppose that there is $x\in H$ which is not in $K.$ Let $x=\sum_{i=1}^n a_ix_i$ and since $x\notin K,$ there is $j$ such that $x_j\notin Y$ and $\pi(a_j)\neq 0.$ 
Since $H$ is directed, $x=x^+-x^-$ with $x^+, x^-\in H^+.$ Represent $x^+, x^-$ using the basis elements so that $x^+=\sum_{i=1}^n a_i^+x_i$ and  $x^-=\sum_{i=1}^n a_i^-x_i$ and, since $H^+=G^+\cap H,$ with $\pi(a_i^+)\geq 0$ and  $\pi(a_i^-)\geq 0$ for all $i.$ The condition $\pi(a_j^+)-\pi(a_j^-)=\pi(a_j)\neq 0$ implies that at least one of $\pi(a_j^+), \pi(a_j^-)$ is nonzero. If $\pi(a_j^+)\neq 0,$ let $b_j\in\pi^{-1}(\pi(a_j^+))$ be such that $b_j\in \Zset^+[\Gamma]$ and let us represent $b_j$ as a finite sum $\sum k_\gamma\gamma$ with $k_\gamma\in \Zset$ and $k_\gamma>0.$ Then we have that $x^+=\sum_{i=1}^n a_i^+x_i\geq a_j^+x_j=b_jx_j\geq k_\gamma \gamma x_j\geq \gamma x_j \geq 0$ for any $\gamma$ appearing in the sum  $b_j=\sum k_\gamma\gamma.$ The convexity of $H$ then implies that $\gamma x_j\in H.$ Since $H$ is a $\Gamma$-subgroup of $G$, we have that $x_j=\gamma^{-1}\gamma x_j\in H.$ This implies that $x_j\in Y$ which is a contradiction. The case $\pi(a_j^-)\neq 0$ is handled similarly. Thus, we have that $H=K.$

The second sentence of the proposition follows from the first. Let $X$ and $Y$ be as in the first part of the proof,  let $Z=X-Y,$ and let $K$ be the $\Gamma$-subgroup of $G$ generated by $Z.$ Then $H$ and $K$ are simplicial $\Gamma$-groups with simplicial $\Gamma$-bases $Y$ and $Z$ respectively. 
\end{proof}

\section{Direct limits and ultrasimplicial \texorpdfstring{$\Gamma$}{TEXT}-groups}\label{section_limits}

In this section, we introduce ultrasimplicial $\Gamma$-groups (Definition \ref{definition_ultrasimplicial}) and show that every ultrasimplicial $\Gamma$-group is realizable by a $\Gamma$-graded ultramatricial ring over a $\Gamma$-graded division ring (Theorem \ref{realization_ultrasimplicial}).

\subsection{Generating intervals and direct limits} 
The non-unital ring case can be handled by considering generating intervals. 
If $G$ is a pre-ordered $\Gamma$-group, a subset $D$  of $G^+$ is a {\em generating interval} if $D$ is upwards directed, 
convex, and such that every element of $G^+$ is a sum of elements from $\Zset^+[\Gamma]D$. Let $\POG^D_\Gamma$ denote the category whose objects are pairs $(G, D)$ where $G$ is an object of $\POG_\Gamma$ and $D$ is a generating interval of $G$, and whose morphisms $f:(G, D)\to (H, E)$ are morphisms of $\POG_\Gamma$ such that $f(D)\subseteq E.$  
The category $\OG^D_\Gamma$ is defined analogously to $\POG^D_\Gamma$ with the additional requirement that the pre-order of an object of $\OG^D_\Gamma$ is an order. 

\begin{proposition}
The categories $\POG_\Gamma,$ $\OG_\Gamma,$ $\POG^u_\Gamma,$ $\OG^u_\Gamma,$ $\POG^D_\Gamma,$ and $\OG^D_\Gamma$ are closed under the formation of direct limits.  
\label{direct_limits}
\end{proposition}
\begin{proof}
The claim for the first four categories follows directly from the proof of the analogous claims in the case when $\Gamma$ is trivial (\cite[Propositions 1.15 and 1.16]{Goodearl_interpolation_groups_book}). We sketch the proofs for completeness. 

If $I$ is a directed set, $(G_i, f_{ij}), i,j\in I, i\leq j,$ is a directed system in $\POG_\Gamma,$ and $G$ is a direct limit of $(G_i, f_{ij})$ in the category of $\Zset[\Gamma]$-modules with the translational maps $f_i: G_i\to G,$ let $G^+=\bigcup_i f_i(G_i^+).$ It is direct to check that $G^+$ is a cone in $G$ making $G$ a pre-ordered $\Gamma$-group and the translational maps order-preserving homomorphisms.

If $G_i$ are objects in $\OG_\Gamma,$ then one can check that $G^+$ is a strict cone by definition, using that the cone $G_i^+$ is strict for every $i\in I$. 

If $(G_i, u_i)$ are objects in $\POG^u_\Gamma$ and $f_{ij}$ morphisms in $\POG^u_\Gamma,$ let $u=f_i(u_i)$ for some $i.$ Since $f_{ij}(u_i)=u_j,$ $u=f_i(u_i)$ for any $i\in I.$ It is direct to check that $u$ is an order-unit of $G$ and so $f_i, i\in I$ are morphisms in $\POG^u_\Gamma.$   

If $((G_i, D_i), f_{ij})$ is a directed system in $\POG^D_\Gamma,$ and $G$ a direct limit in $\POG_\Gamma$ with the translational maps $f_i,$ the set $D=\bigcup_i f_i(D_i)$ is in $G^+$ and $f_i(D_i)\subseteq D$ by definition. The definitions of $D$ and $G^+$ also imply that every element of $G^+$ is a sum of elements from $\Zset^+[\Gamma]D$. One checks that $D$ is convex and upwards directed using definitions and the properties of direct limit. So, the category $\POG^D_\Gamma$ is closed under the formation of direct limits. The claim for the category $\OG^D_\Gamma$ follows from the statements for $\OG_\Gamma$ and $\POG^D_\Gamma$.  
\end{proof}

Let $\POG^{D*}_\Gamma$ (resp. $\OG^{D*}_\Gamma$) denote the category $\POG^D_\Gamma$ (resp. $\OG^D_\Gamma$) such that an object $(G,D)$ has an additional requirement that $G$ is also an object of $\POG^*_\Gamma$ (resp. $\OG^*_\Gamma$), that a morphism is also a morphism of $\POG^*_\Gamma$ (resp. $\OG^*_\Gamma$) and that $D^*=D$ where $x\mapsto x^*$ denotes the action of the nontrivial element of $\Zset_2.$ The following can be shown analogously to Proposition \ref{direct_limits}.

\begin{corollary}
The categories $\POG^*_\Gamma,$ $\OG^*_\Gamma,$ $\POG^{u*}_\Gamma,$ $\OG^{u*}_\Gamma,$ $\POG^{D*}_\Gamma,$ and $\OG^{D*}_\Gamma$  are closed under the formation of direct limits.  
\label{direct_limits_star}
\end{corollary}

\subsection{Interpolation property and ultrasimplicial \texorpdfstring{$\Gamma$}{TEXT}-groups}

Recall that a pre-ordered group $G$ satisfies the  {\em interpolation property} (also known as the Riesz interpolation property) if the following holds. 
\begin{itemize}
\item For any finite sets $X, Y\subseteq G,$ if $x\leq y$ for every $x\in X, y\in Y,$ then there is $z\in G$ such that $x\leq z\leq y$ for every $x\in X, y\in Y.$
\end{itemize}
By induction, it is sufficient to require this condition to hold for two-element sets $X$ and $Y.$ For partially ordered groups, this property is equivalent to the Riesz decomposition and Riesz refinement properties listed below and the proof of these equivalences can be found in \cite[Theorem IV.6.2]{Davidson} or \cite[Proposition 2.1]{Goodearl_interpolation_groups_book}.

\begin{itemize}
\item For every $x,y_1, y_2\in G^+$ such that $x\leq y_1+y_2,$ there are $x_1,x_2\in G^+$ such that $x=x_1+x_2$ and $x_i\leq y_i$ for $i=1,2.$
 
\item For every $x_1,x_2,y_1, y_2\in G^+$ such that $x_1+x_2=y_1+y_2,$ there are $z_{ij}\in G^+, i,j=1,2$ such that $x_i=z_{i1}+z_{i2}$ and $y_j=z_{1j}+z_{2j}$ for $i,j=1,2.$ 
\end{itemize}

\begin{lemma}
Every simplicial $\Gamma$-group has the interpolation property. As a corollary, every direct limit of a directed system of simplicial $\Gamma$-groups has the interpolation property. 
\label{lemma_simplicial_interpolation}
\end{lemma}
\begin{proof} 
Let $G$ be a simplicial $\Gamma$-group with a simplicial $\Gamma$-basis $X=\{e_1,\ldots, e_n\}$ and let $\Delta=\Stab(X).$  Let $x_1+x_2=y_1+y_2$ for some $x_j, y_k\in G^+, j,k=1,2.$ By part (3) of Proposition \ref{properties_of_simplicial}, we have that  $x_j=\sum_{i=1}^n b_{ji}e_i$ and $y_k=\sum_{i=1}^n c_{ki}e_i$ for some $b_{ji}, c_{ki}\in \Zset^+[\Gamma]$, $j,k=1,2, i=1,\ldots,n,$ and, by part (2) of Proposition \ref{properties_of_simplicial}, $\pi(b_{1i})+\pi(b_{2i})=\pi(c_{1i})+\pi(c_{2i})$ for all $i=1,\ldots,n$ where $\pi$ is the natural left $\Zset[\Gamma]$-module map $\Zset[\Gamma]\to \Zset[\Gamma/\Delta].$ Since $\Zset[\Gamma/\Delta]$ satisfies the interpolation property (which is direct to check using the fact that $\Zset$ satisfies the interpolation property), there are $d_{jk}^i\in\Zset[\Gamma]$ such that $\pi(d_{jk}^i)\geq 0,$ for all $j,k=1,2,$ $i=1,\ldots, n,$ $\pi(d_{j1}^i)+ \pi(d_{j2}^i)=\pi(b_{ji})$ for $j=1,2,$ and  $\pi(d_{1k}^i)+\pi(d_{2k}^i)=\pi(c_{ki})$ for $k=1,2.$

Letting $z_{jk}=\sum_{i=1}^n d_{jk}^ie_i\in G^+, j,k=1,2$ produces the required equations since $$\pi(b_{ji})=\pi(d_{j1}^i)+\pi(d_{j2}^i)\mbox{ implies that }x_j=\sum_{i=1}^n b_{ji}e_i=\sum_{i=1}^n (d_{j1}^i+d_{j2}^i)e_i=z_{j1}+z_{j2}$$ for $j=1,2$ by (Ind) and, similarly, $y_k=z_{1k}+z_{2k},$ for $k=1,2.$

It is direct to check that the interpolation property is preserved under direct limits so the second sentence of the lemma holds. 
\end{proof}

If $(G, u)$ is an ordered group with the interpolation property, then the set $\{x\in G\, |\, 0\leq x\leq u\}$ is a generating interval by the paragraph preceding Lemma 17.8 of \cite{Goodearl_interpolation_groups_book}. We write this set shorter as $[0, u],$ and we formulate and prove the $\Gamma$-group version of this statement below.   

\begin{proposition}
If $(G, u)$ is an ordered $\Gamma$-group with the interpolation property, then $[0, u]$ is a generating interval.  
Thus, if $(G, u)$ is a simplicial $\Gamma$-group, then  $[0, u]$ is a generating interval. 
\label{generating_interval_of_simplicial} 
\end{proposition}
\begin{proof}
It is direct to show that $[0, u]$ is convex. To show that $[0,u]$ is upwards directed, let $x,y\in [0, u]$ and let $z$ be an interpolant for $\{x,y\}\leq \{u\}.$ By $0\leq x\leq z\leq u,$ $z\in [0,u]$ and so $[0, u]$ is upwards directed. Lastly, if $x\in G^+,$ we want to represent $x$ as a finite sum of elements from $\Zset^+[\Gamma][0, u].$ Since $u$ is an order-unit, there are $n$ and $\gamma_i\in\Gamma$ for $i=1,\ldots,n$ such that $x\leq \sum_{i=1}^n\gamma_iu.$ By interpolation, there are $x_i\in G^+$ such that $x=\sum_{i=1}^n x_i$ and $x_i\leq \gamma_iu$ for all $i=1,\ldots,n.$ The relations $0\leq x_i\leq \gamma_i u$ imply that $0\leq \gamma_i^{-1}x_i\leq u.$ For $y_i=\gamma_i^{-1}x_i$ we have that $y_i\in [0, u]$ and $x=\sum_{i=1}^n\gamma_iy_i.$   
 
The last sentence follows directly from the first and Lemma \ref{lemma_simplicial_interpolation}.  
\end{proof}

We use Proposition \ref{direct_limits}, Corollary \ref{direct_limits_star}, and Proposition \ref{generating_interval_of_simplicial} to define an ultrasimplicial $\Gamma$-group. 

\begin{definition} Let $\Gamma$ be a group and let $G$ be an ordered and directed $\Gamma$-group. If $u$ is an order-unit of $G,$ $(G, u)$ is an {\em ultrasimplicial $\Gamma$-group} if $(G, u)$ is a direct limit of a sequence of simplicial $\Gamma$-groups $(G_n, u_n)$ with simplicial $\Gamma$-bases stabilized by the same subgroup $\Delta$ of $\Gamma$ in the category $\OG^u_\Gamma.$ If $D$ is a generating interval of $G$, $(G, D)$ is an {\em ultrasimplicial $\Gamma$-group} if $(G, D)$ is a direct limit of a sequence of simplicial $\Gamma$-groups $(G_n, [0, u_n])$ with simplicial $\Gamma$-bases stabilized by the same subgroup $\Delta$ of $\Gamma$ in the category $\OG^D_\Gamma.$ In either one of these two cases, we say that $G$ is an {\em ultrasimplicial $\Gamma$-group.}

If all simplicial $\Gamma$-groups $G_n$ above are considered as $\Gamma$-$\Zset_2$-bigroups with the trivial $\Zset_2$-action, direct limits can be considered in the categories $\OG^{u*}_\Gamma$ or $\OG^{D*}_\Gamma$ and the group $G$ can also be considered as a $\Gamma$-$\Zset_2$-bigroup with the trivial action of $\Zset_2.$  
\label{definition_ultrasimplicial}
\end{definition}

The proposition below is the $\Gamma$-group version of \cite[Proposition 12.3]{Goodearl_Handelman}. We use this property to describe the Grothendieck $\Gamma$-group of a non-unital graded regular ring and to show Corollary \ref{dir_lim_corollary}. 
 
\begin{proposition} Let $\Gamma$ be a group with a subgroup $\Delta$ and let $(G, u)$ be an object of $\OG^u_\Gamma$ such that $\Delta\subseteq \Stab(\{u\})$ and such that $G$ is directed and has the interpolation property. If $f:(G, u)\to (\Zset[\Gamma/\Delta], \Delta)$ is a morphism of $\OG^u_\Gamma$ and $D=\{x\in \ker f\;|\; 0\leq x\leq u\},$ then the following holds. 
\begin{enumerate}
\item $\ker f$ is a directed and convex $\Zset[\Gamma]$-submodule of $G.$
 
\item $D$ is a generating interval in $(\ker f)^+.$

\item $G^+=\{ x+au \;|\; a\in \Zset^+[\Gamma], x\in \ker f, x+ad\geq 0$ for some $d\in D\}.$
\end{enumerate}
\label{description_of_kernel}
\end{proposition}
\begin{proof}
It is direct to show that $\ker f$ is a convex $\Zset[\Gamma]$-submodule of $G.$ To show that $\ker f$ is directed, it is sufficient to show that $\ker f=(\ker f)^+-(\ker f)^+$. To show this, we follow the ideas of the proof of \cite[Proposition 12.3]{Goodearl_Handelman}, adapting some arguments as needed. Let $x\in \ker f.$ Since $G$ is directed, 
we can write $x=y-z$ for $y,z\in G^+.$ For $y,$  $y\leq \sum_{i=1}^n \gamma_iu$ for some $n$ and $\gamma_i\in \Gamma, i=1,\ldots,n.$ By the interpolation property, there
are $y_i\in G^+, i=1,\ldots,n$ such that $y=\sum_{i=1}^n y_i$ and $0\leq y_i\leq \gamma_iu.$ Applying $f$ to this last inequality implies that $0\leq f(y_i)\leq \gamma_i\Delta.$ The last relation implies that $f(y_i)=0$ or $f(y_i)=\gamma_i\Delta$ for every $i=1,\ldots, n.$ 

If $f(y_i)=0,$ then $y_i\in (\ker f)^+$ and we can write $y_i=0+y_i-0$ which is in $\{0,1\}\gamma_iu+(\ker f)^+-(\ker f)^+.$ If $f(y_i)=\gamma_i\Delta,$ then $f(\gamma_iu-y_i)=0$ and $y_i\leq \gamma_iu$ so $y_i=\gamma_i u+0 -(\gamma_iu-y_i)$ which is also in $\{0,1\}\gamma_iu+(\ker f)^+-(\ker f)^+.$ In both cases, we have $y_i=p_i\gamma_iu+a_i-b_i$ for $p_i\in\{0, 1\}$ and $a_i, b_i\in (\ker f)^+.$ Let 
\[y=\sum_{i=1}^n p_i\gamma_iu +a-b\mbox{  for }a=\sum_{i=1}^na_i\in(\ker f)^+\mbox{ and }b=\sum_{i=1}^n b_i\in (\ker f)^+.\] Analogously, for $z\in G^+,$ we obtain that there are $m$, $p_j'\in \{0, 1\},$ $\gamma'_j\in \Gamma,$ for $j=1, \ldots, m,$ and $a', b'\in(\ker f)^+$ such that $z=\sum_{j=1}^m p_j'\gamma'_ju+a'-b'.$ 

Since $f(x)=0,$ we have that $\sum_{i=1}^n p_i\,\gamma_i\Delta=f(y)=f(z)=\sum_{j=1}^m p_j'\, \gamma'_j\Delta.$ Thus, there is a bijection $\sigma: A=\{i\in \{1, \ldots, n\}\, |\,p_i=1\}\to B=\{j\in \{1,\ldots, m\}\, |\, p_j'=1\}$ such that $p_i=1$ if and only if $p_{\sigma(i)}'=1$ and $\gamma_i\Delta=\gamma'_{\sigma(i)}\Delta$ for all $i\in A.$ So, $(\gamma_i-\gamma'_{\sigma(i)})u=0$ for all $i\in A$ by the assumption that $\Delta\subseteq \Stab(\{u\}).$ Since  either $p_i=0$ or $i\in A$
for every $i=1,\ldots, n,$ we have that $\sum_{i=1}^n p_i\gamma_iu-\sum_{j=1}^m p_j'\gamma'_ju=\sum_{i\in A} \gamma_iu-\sum_{i\in A} \gamma'_{\sigma(i)}u=0$ and so $x=y-z=(a+b')-(a'+b)$ is in $(\ker f)^+-(\ker f)^+$ which finishes the proof of (1). 

To show (2), it is direct to check that $D$ is convex. If $x,y\in D$ and $z$ is an interpolant for $\{x,y\}\leq \{x+y, u\},$ then $0\leq x\leq z\leq u.$ By $x\leq z\leq x+y,$ $f(z)=0$ and so $z\in D.$ If $x\in (\ker f)^+,$ let $n$ and $\gamma_1,\ldots, \gamma_n$ be such that $x\leq \sum_{i=1}^n\gamma_iu.$ Let elements $x_i\in G^+,
i=1,\ldots,n$ be such that $x=\sum_{i=1}^n x_i$ and $0\leq x_i\leq \gamma_iu$ and let $y_i=\gamma_i^{-1}x_i\in [0, u]$ for $i=1,\ldots, n$. The relations $0\leq x_i\leq x$ imply that $f(x_i)=0$ for all $i$ and so $f(y_i)=\gamma_i^{-1}f(x_i)=0$ for all $i$ as well. Thus $x=\sum_{i=1}^n\gamma_iy_i$ for $y_i\in D, i=1,\ldots, n.$

To show (3), let $C$ denote the set on the right-hand side of the equation in (3). Then $C\subseteq G^+$ since for $d\in D,$ $d\leq u$ and so $ad\leq au$ which implies $x+ad\leq x+au.$ Hence if $x+ad\geq 0,$ we have that $x+au\geq 0$. For the converse, let $y\in G^+$ and $f(y)=a\Delta$ for some $a=\sum_{i=1}^n\gamma_i\in \Zset^+[\Gamma].$ The element $x=y-au$ is in $\ker f$ by definition. Since $\ker f$ is directed, $x=v-w$ for some $v,w\in (\ker f)^+$ and, since $y\in G^+,$ $w\leq y+w=au+x+w=\sum_{i=1}^n\gamma_iu+v.$ By interpolation, there are $w_0\in G^+$ and $w_i\in G^+, i=1,\ldots,n$ such that $w=w_0+\sum_{i=1}^n w_i,$ $w_0\leq v$ and $w_i\leq \gamma_iu$ so that $0\leq \gamma_i^{-1}w_i\leq u.$ Let $d$ be an interpolant for $\{\gamma_1^{-1}w_1,\ldots, \gamma_n^{-1}w_n\}\leq \{u,\sum_{i=1}^n\gamma_i^{-1} w\}.$ The relations $0\leq \gamma_i^{-1}w_i\leq d\leq \sum_{i=1}^n\gamma_i^{-1} w$ for $i=1,\ldots, n,$ and the convexity of $\ker f$ imply that $d$ is in $\ker f.$ Since $0\leq \gamma_i^{-1}w_i\leq d\leq u,$ $d$ is in $D.$ Hence, we have that $w=w_0+\sum_{i=1}^n \gamma_i\gamma_i^{-1}w_i\leq v+\sum_{i=1}^n \gamma_i d= v+ad$ and so $x+ad=v-w+ad\geq 0.$ This implies that $y=x+au$ is in $C.$ 
\end{proof}

\subsection{The Grothendieck \texorpdfstring{$\Gamma$}{TEXT}-group of a non-unital graded ring and the functor \texorpdfstring{$K_0^{\Gamma}$}{TEXT}}

After a short review of Grothendieck $\Gamma$-groups of non-unital graded rings, we show that the functor $K_0^{\Gamma}$ preserves direct limits on appropriate categories. This implies Corollary \ref{dir_lim_corollary} stating that $K_0^{\Gamma}(R)$ is ultrasimplicial if $R$ is a graded ultramatricial ring over a graded division ring. 

Let $R$ be a ring, possibly non-unital. The Grothendieck group $K_0(R)$ can be defined using unitization of $R$ (see \cite{Goodearl_Handelman}, for example). If $R$ is also $\Gamma$-graded, the Grothendieck $\Gamma$-group of $R$ is defined analogously, see \cite[Section 3.5]{Roozbeh_book}. If $R$ is a graded $*$-ring, see \cite[Section 4]{Roozbeh_Lia}. We briefly review some relevant facts of these constructions. 

If $R$ is a $\Gamma$-graded ring, possibly non-unital, a unitization $R^u$ of $R$ can be defined as $R\oplus \Zset$ 
with the addition given component-wise, the multiplication given by
$(r,m)(s,n)=(rs+nr+ms, mn),$
for  $r,s \in R,$ $m,n\in \Zset,$ and the $\Gamma$-grading given by 
$R^u_0= R_0\times\Zset,$ $R^u_\gamma= R_\gamma\times \{0\},$ $\gamma\neq 0.$ 

The map $K_0^{\Gamma}(R^u) \rightarrow K_0^{\Gamma}(R^u/R)$ induced by the graded epimorphism $p_R:R^u\rightarrow R^u/R$ is a natural $\Gamma$-group homomorphism. The group $K_0^{\Gamma}(R)$ is defined as the kernel of this homomorphism. It inherits the  action of $\Gamma$ and the pre-order structure from $K_0^{\Gamma}(R^u).$ 
Let  
\[D_R=
\big \{x\in \ker\left( K_0^{\Gamma}(R^u)\to K_0^{\Gamma}(R^u/R)\right)\; | \;0\leq x\leq [R^u]\big\}.
\] 
If $R$ is unital, then the definition of its $K_0^{\Gamma}$-group coincides with the construction using the projective modules and the set $D_R$ is naturally isomorphic to $\{x\in K_0^{\Gamma}(R)\, |\, 0\leq x\leq [R]\}$ (the version of this statement for trivial $\Gamma$ is shown in \cite[Proposition 12.1]{Goodearl_Handelman}).

If $R$ and $S$ are two $\Gamma$-graded rings  (possibly non-unital) and $\phi: R\to S$ a graded ring homomorphism (possibly non-unital), then $K_0^{\Gamma}(\phi)$ is a morphism of $\POG_\Gamma$ such that $K_0^{\Gamma}(\phi)(D_R)\subseteq D_S.$  

Since we consider $\Zset$ trivially graded in the above construction, $K_0^\Gamma(\Zset)=\Zset[\Gamma]$ (see \cite[Example 3.1.5]{Roozbeh_book}) and $K_0^\Gamma(p_R): (K_0^\Gamma(R^u), [R^u])\to (\Zset[\Gamma], 1)$ is a morphism of $\POG^u_\Gamma.$ Clearly, $\{1\}\subseteq \Stab([R^u]).$ So, the assumptions of Proposition \ref{description_of_kernel} are satisfied provided that $K_0^\Gamma(R^u)$ is ordered and with interpolation. We consider the conditions on the graded ring $R$ which guarantee that. In that case, we can conclude that $D_R$ is a generating interval. 

If $R$ is a unital (von Neumann) regular ring (i.e. $x\in xRx$ for every $x$), then $K_0(R)$ has interpolation by \cite[Theorem 2.8]{Goodearl_book}. In this case, when $K_0(R)$ is also ordered (for example when $\M_n(R)$ is directly finite for every $n$), then $D_R$ is a generating interval. Recall that a $\Gamma$-graded ring $R$ is {\em graded regular} if $x\in xRx$ for every homogeneous element $x$ (see \cite[Section 1.1.9]{Roozbeh_book}). We note the graded version of \cite[Theorem 2.8]{Goodearl_book}: if $R$ is a unital $\Gamma$-graded regular ring, then $K_0^\Gamma(R)$ is a pre-ordered $\Gamma$-group with interpolation. This statement can be shown directly following the proof of 
\cite[Theorem 2.8]{Goodearl_book} and translating it to the graded ring setting. Further, \cite[Theorems 1.7 and 1.11]{Goodearl_book}, which the proof also uses, can be formulated and shown for graded regular rings as well. As a result, if $R$ is a graded matricial ring over a graded division ring $K,$ then $(K_0^\Gamma(R), [R])$ is an ordered and directed $\Gamma$-group with interpolation. The assumptions of Proposition \ref{generating_interval_of_simplicial} are satisfied and so $D_R,$  naturally isomorphic to $[0, [R]],$ is a generating interval. If $R$ and $S$ are two matricial $\Gamma$-graded rings and $\phi: R\to S$ a graded ring homomorphism (possibly non-unital), $K_0^{\Gamma}(\phi)$ is a morphism of $\OG_\Gamma^D.$  
   
If $R$ is a $\Gamma$-graded $*$-ring, then $R^u$ is also a graded $*$-ring with $(r,n)^*=(r^*, n)$ for $r\in R$ and $n\in \Zset,$ the map $p_R:R^u\rightarrow R^u/R$ is a $*$-homomorphism, and $K_0^\Gamma(p_R)$ is a morphism of $\POG^{u*}_\Gamma.$ The kernel of this map inherits the action of $\Zset_2$ from $K_0^{\Gamma}(R^u).$ 

Let $\GrRings$ denote the category of $\Gamma$-graded rings, let $\GrRings^*$ denote the category of $\Gamma$-graded $*$-rings, let $\GrURings$ denote the category of unital, $\Gamma$-graded rings, and let $\GrURings^*$ denote the category of unital, $\Gamma$-graded $*$-rings. 

\begin{proposition}
For any group $\Gamma,$ the functor  $K_0^{\Gamma}$ preserves direct limits on the following categories.
\begin{enumerate}
\item $K_0^{\Gamma}: \GrRings\to \POG_\Gamma,$ and $K_0^{\Gamma}: \GrURings\to \POG^u_\Gamma.$ 
\item $K_0^{\Gamma}: \GrRings^*\to \POG_\Gamma^*$ and $K_0^{\Gamma}: \GrURings^*\to \POG^{u*}_\Gamma.$
\end{enumerate}
\label{direct_limits_K0}  
\end{proposition}
\begin{proof}
To show (1), let $I$ be a directed set, let $(R_i, \phi_{ij})$ be a directed system in $\GrRings,$ let $R$ be its direct limit, and let $\phi_i$ be the translational maps. Let $G$ be a direct limit of $(K_0^{\Gamma}(R_i), K_0^{\Gamma}(\phi_{ij}))$ in the category $\POG_\Gamma$ which exists by Proposition \ref{direct_limits} and let $f_i$ denote the translational maps. If $f: G\to K_0^{\Gamma}(R)$ is a unique map in $\POG_\Gamma$ such that $ff_i=K_0^{\Gamma}(\phi_i),$ \cite[Theorem 3.2.4]{Roozbeh_book} shows that $f$ is injective and that 
$(K_0^{\Gamma}(R))^+\subseteq f(G^+).$ Since $K_0^{\Gamma}(R)$ is directed, this last relation implies that $f$ is surjective and that $f^{-1}$ is order-preserving. So, $f$ is an isomorphism of $\POG_\Gamma.$ 

If $R_i, i\in I,$ from above are unital, the maps $\phi_{ij}, i\leq j,$ are morphisms of $\GrURings$, and $R$ is a direct limit in $\GrURings$, let $G,$ $f_i,$ and $f$ be as in the previous paragraph. Since $K_0^{\Gamma}(\phi_{ij})([R_i])=[R_j],$ $f_i([R_i])=f_j([R_j])$ for any $i,j\in I.$ 
Let $u\in G^+$ be the element $u=f_i([R_i])$ for any $i.$ For $x\in G$ such that $x=f_i(x_i),$ let $a\in \Zset^+[\Gamma]$ be such that $x_i\leq a[R_i].$ Then $x\leq au$ so $u$ is an order-unit of $G.$ Since $f(u)=f(f_i([R_i]))=K_0^{\Gamma}(\phi_{i})([R_i])=[R],$ $f$ is an isomorphism of $\POG^u_\Gamma.$  

To show (2), let $(R_i, \phi_{ij})$ be a directed system in $\GrRings^*$ and let $G, f_i,$ and $f$ be as in the first paragraph of the proof. For $g\in G$ such that $g=f_i(x_i)$ we can define $g^*=f_i(x_i^*).$ The map $g\mapsto g^*$ is well-defined and it makes $f_i, i
\in I,$ into morphisms of $\POG_\Gamma^*.$ Then $f$ is also a morphism of $\POG^*_\Gamma$ since $f(g^*)=ff_i(x_i^*)=K_0^{\Gamma}(\phi_{i})(x_i^*)=K_0^{\Gamma}(\phi_{i})(x_i)^*=(ff_i(x_i))^*=f(g)^*.$ If, in addition,  $(R_i, \phi_{ij})$ is a directed system in $\GrURings^*,$ then $K_0^{\Gamma}(\phi_{ij})$ and $f_i$ are morphisms in $\POG^{u*}_\Gamma.$ The order-unit $u=f_i([R_i])$ in $G$ is such that $u^*=f_i([R_i])^*=f_i([R_i]^*)=f_i([R_i])=u$ so $(G, u)$ is an object in $\POG^{u*}_\Gamma.$ 
\end{proof}

We now establish an expected result: every $\Gamma$-graded ultramatricial ring over a graded division ring gives rise to an ultrasimplicial $\Gamma$-group. The main result of section \ref{section_limits}, Theorem \ref{realization_ultrasimplicial}, states the converse of this statement: every ultrasimplicial $\Gamma$-group can be realized by a $\Gamma$-graded ultramatricial ring over a graded division ring. 

\begin{corollary}
Let $K$ be a $\Gamma$-graded division ring. 
\begin{enumerate}
\item If $R$ is a unital graded ultramatricial ring over $K,$ then $(K_0^{\Gamma}(R), [R])$ is an ultrasimplicial $\Gamma$-group. 

\item If $R$ is a graded ultramatricial ring over $K,$ then $(K_0^{\Gamma}(R), D_R)$ is an ultrasimplicial $\Gamma$-group.   
\end{enumerate}
If $K$ is a $\Gamma$-graded division $*$-ring and $R$ is a direct limit in $\GrURings^*$ (respectively $\GrRings^*$), then the ultrasimplicial $\Gamma$-group from (1) (respectively (2)) is an object in $\OG^{u*}_\Gamma$ (respectively $\OG^{D*}_\Gamma$).
\label{dir_lim_corollary}
\end{corollary}
\begin{proof}
(1) Let $(R_n, \phi_{nm})$ be a directed system of graded matricial rings over $K$ with $R$ as their direct limit in $\GrURings.$ By sections \ref{subsection_Morita},  \ref{subsection_division_ring}, and \ref{Grothendieck_of_matricial}, 
$K_0^{\Gamma}(R_n)$ is a simplicial $\Gamma$-group for every $n$ with simplicial $\Gamma$-bases stabilized by $\Gamma_K$
and, by Proposition \ref{direct_limits_K0}, $(K_0^{\Gamma}(R), [R])$ is a direct limit of $\left((K_0^{\Gamma}(R_n), [R_n]), K_0^\Gamma(\phi_{nm})\right)$ in the category $\POG^u_\Gamma.$ Moreover, $((K_0^{\Gamma}(R_n), [R_n]), K_0^{\Gamma}(\phi_{nm}))$ is a directed system in $\OG^u_\Gamma$ so $(K_0^{\Gamma}(R), [R])$ is an object of $\OG^u_\Gamma.$ 

(2) Let $(R_n, \phi_{nm})$ be a directed system of graded matricial rings over $K$ with $R$ as their direct limit in $\GrRings,$ so that the maps $\phi_{nm}$ may possibly be non-unital. If $R^u_n, n\in \omega,$ and $R^u$ denote the unitizations of $R_n$ and $R$ respectively, then $K_0^\Gamma(R_n^u)$ and $K_0^{\Gamma}(R^u)$ have the interpolation property by Lemma \ref{lemma_simplicial_interpolation} and by part (1). Thus, $D_{R_n}$ and $D_R$ are generating intervals of $K_0^\Gamma(R_n)$ and $K_0^\Gamma(R)$ respectively by Proposition \ref{description_of_kernel}. 

To shorten the notation, we use $\ol \phi$ for $K_0^\Gamma(\phi)$ for any graded ring homomorphism $\phi.$ Since $\ol{\phi_n}(D_{R_n})\subseteq D_R$ we have that $\bigcup_{n\in\omega}\ol{\phi_n}(D_{R_n})\subseteq D_R.$ Let $G$ be a direct limit of $(K_0^{\Gamma}(R_n), \ol{\phi_{nm}})$ in the category $\POG_\Gamma,$ let $f_n:K_0^{\Gamma}(R_n)\to G$ be the translational maps and let $f: G\to K_0^{\Gamma}(R)$ be a unique morphism of $\POG_\Gamma$ such that $ff_n=\ol{\phi_n}.$ Then $f$ is an isomorphism by Proposition \ref{direct_limits_K0}. It is sufficient to show that $D=\bigcup_{n\in\omega} f_n(D_{R_n})$ is a generating interval of $G$ such that $f(D)=D_R.$ 

If $x\in G^+,$ then $x=f_n(x_n)$ for some $x_n\in K_0^{\Gamma}(R_n)^+.$ Since $D_{R_n}$ is a generating interval, $x_n$ is a sum of elements from $\Zset^+[\Gamma]D_{R_n}$ and so $x$ is a sum of elements from $\Zset^+[\Gamma]D$ by the definition of $D$. If $x,y\in D,$ we can find $n$ such that $x=f_n(x_n)$ and $y=f_n(y_n)$ for some $x_n, y_n\in D_n.$ The set $D_{R_n}$ is upwards directed so there is $z_n\in D_{R_n}$ $z_n\geq x_n, z_n\geq y_n,$ and so $f_n(z_n)\geq x$ and $f_n(z_n)\geq y,$ showing that $D$ is upwards directed. If $x,y\in D$ and $x\leq z\leq y$ for some $z\in G,$ we can find $n$ such that $x=f_n(x_n), y=f_n(y_n),$ and $z=f_n(z_n)$ for some $x_n, y_n\in D_{R_n}$ and $z_n\in K_0^\Gamma(R_n).$ The order $f_n(x_n)\leq f_n(z_n)\leq f_n(y_n)$ implies that there is $m\geq n$ such that $\ol{\phi_{nm}}(x_n)\leq \ol{\phi_{nm}}(z_n)\leq \ol{\phi_{nm}}(y_n).$ Since $D_{R_m}$ is convex, $\ol{\phi_{nm}}(z_n)\in D_{R_m}$ and so $z=f_n(z_n)=f_m(\ol{\phi_{nm}}(z_n))\in f_m(D_{R_m})\subseteq D,$ showing that $D$ is convex. 

As $D=\bigcup_{n\in\omega} f_n(D_{R_n}),$ $f(D)=\bigcup_{n\in\omega} ff_n(D_{R_n})=\bigcup_{n\in\omega}\ol{\phi_n}(D_{R_n})\subseteq D_R.$
It remains to show that $D_R\subseteq f(D).$ Let $\phi_{nm}^u: R_n^u\to R^u_m, n\leq m$ and $\phi_{n}^u: R_n^u\to R^u$ denote the extensions of $\phi_{nm}$ and  $\phi_{n}$ respectively and let $\iota_n: R_n\to R^u_n$ and $\iota: R\to R^u$ denote the natural injections. 
Let $x\in D_R.$ Since $f$ is onto, $x=ff_n(x_n)=\ol{\phi_n}(x_n)$ for some $n$ and $x_n\in K_0^\Gamma(R_n)^+.$ By the definition of $D_R,$ we have that $0\leq \ol\iota(x)\leq [R^u]$ and that $\ol\iota(x)=\ol\iota\ol{\phi_n}(x_n)=\ol{\phi_n^u}\ol{\iota_n}(x_n).$ The relation $0\leq \ol{\phi_n^u}\ol{\iota_n}(x_n)\leq \ol{\phi^u_n}([R^u_n])$ and part (1) imply that there is $m\geq n$ such that $0\leq \ol{\phi_{nm}^u}\ol{\iota_n}(x_n)=\ol{\iota_m}\ol{\phi_{nm}}(x_n)\leq \ol{\phi^u_{nm}}([R^u_n])=[R^u_m].$ Thus, $x_m=\ol{\phi_{nm}}(x_n)$ is in $D_{R_m}$ and $x=ff_n(x_n)=ff_m\ol{\phi_{nm}}(x_n)=ff_m(x_m)\in ff_m(D_{R_m})\subseteq f(D).$

If $K$ is also a $\Gamma$-graded division $*$-ring, $R_n$ and $R$ are graded $*$-rings and their Grothendieck $\Gamma$-groups are $\Gamma$-$\Zset_2$ bigroups with the trivial $\Zset_2$-action. Thus, the connecting and translational maps $\ol{\phi_{nm}}$ and $\ol{\phi_n}$ are maps in the categories $\POG^{u*}_\Gamma$ or $\POG^{D*}_\Gamma.$
\end{proof}

\subsection{Fullness of \texorpdfstring{$K_0^{\Gamma}$}{TEXT} and realization of ultrasimplicial \texorpdfstring{$\Gamma$}{TEXT}-groups}

Let $K$ be a $\Gamma$-graded division ring and let $\GrMat_K$ denote the category whose objects are graded matricial rings over $K$ and whose morphisms are graded ring homomorphisms. We show that the functor $K_0^{\Gamma}: \GrMat_K\to \OG^D_\Gamma$ is full (Proposition \ref{fullness}). Using this result, we show the main result of this section, Theorem \ref{realization_ultrasimplicial}. 

\begin{proposition} (Fullness)
Let $K$ be a $\Gamma$-graded division ring, let $R$ and $S$ be graded matricial rings over $K$ and let $f:(K_0^{\Gamma}(R), D_R)\rightarrow (K_0^{\Gamma}(S), D_S)$ a morphism of $\OG^D_\Gamma.$ There is a graded ring homomorphism $\phi:R\rightarrow S$ such that $K_0^{\Gamma}(\phi)=f$. Furthermore, if $f$ is order-unit-preserving (i.e. a morphism of $\OG^u_\Gamma$), then $\phi$ can be chosen to be unit-preserving. 
\label{fullness}
\end{proposition}

Proposition \ref{fullness} is an analogue of \cite[Theorem 5.1.3 (2)]{Roozbeh_book} (also \cite[Theorem 2.13 (2)]{Roozbeh_Annalen}) 
stating that if $\Gamma$ is abelian, if $K$ is commutative, and if $f$ is order-unit-preserving, then the graded ring homomorphism $\phi$ can also be required to be a graded $K$-algebra homomorphism. In the case when $f$ is order-unit-preserving, Proposition \ref{fullness} can be shown analogously to \cite[Theorem 5.1.3 (2)]{Roozbeh_book}, which is shown following the steps of the classic proof \cite[Lemma 15.23 (a)]{Goodearl_book}. The proof of \cite[Theorem 12.5]{Goodearl_Handelman} extends the classic proof to the case when $f$ is not necessarily order-unit-preserving and we adapt these arguments in our proof.   

\begin{proof}
Let us first assume that $f$ is order-unit-preserving $f:(K_0^{\Gamma}(R), [R])\rightarrow (K_0^{\Gamma}(S), [S])$.  Let 
\[R=\bigoplus_{i=1}^n R_i, \mbox{ for }R_i=\M_{p(i)}(K)(\gamma_{i1},\dots,\gamma_{ip(i)})\]
and let $\pi_i: R\to R_i$ and $\iota_i: R_i\to R$ denote the canonical projections and injections for $i=1,\ldots, n.$ We also use $e^i_{kl}$ to denote the element of $R$ such that $\pi_i(e^i_{kl})$ is the $(k,l)$-th standard graded matrix unit of $R_i$ as before. 

Since  $f$ is order-preserving, there is a finitely generated graded right projective $S$-module $P_i$ such that $[P_i]=f([\iota_i(R_i)]).$ Then  $\bigoplus_{i=1}^n P_i\cong_{\gr} S$ by Lemma \ref{iso_carries_from_K0} since
\[[\bigoplus_{i=1}^n P_i]=\sum_{i=1}^n[P_i]=\sum_{i=1}^n f\left([\iota_i(R_i)]\right)=f([\bigoplus_{i=1}^n \iota_i(R_i)])=f\left([R]\right)=[S].\] Let $g_i$ be homogeneous orthogonal idempotents in $S$ with $\sum_{i=1}^n g_i=1$ such that $g_iS\cong_{\gr} P_i.$  

Consider the modules $(\gamma_{i1}^{-1})e^i_{11}R$ and the finitely generated graded projective $S$-modules $Q_i$ such that $[Q_i]=f([(\gamma_{i1}^{-1})e^i_{11}R]).$ Since we have that $[\iota_i(R_i)]=\sum_{k=1}^{p(i)}\gamma_{ik}^{-1}\gamma_{i1}[e^i_{11}R]$ 
(see section \ref{subsection_Morita}), we have that 
\[[g_iS]=[P_i]=f([\iota_i(R_i)])=\sum_{k=1}^{p(i)}\gamma_{ik}^{-1}\gamma_{i1}f([e^i_{11}R])=\]
\[=\sum_{k=1}^{p(i)}\gamma_{ik}^{-1}f([(\gamma_{i1}^{-1})e^i_{11}R])=\sum_{k=1}^{p(i)}\gamma_{ik}^{-1}[Q_i]=[\bigoplus_{k=1}^{p(i)}(\gamma_{ik})Q_i]\]
which implies that $g_iS\cong_{\gr}\bigoplus_{k=1}^{p(i)}(\gamma_{ik})Q_i$ again by Lemma \ref{iso_carries_from_K0}. 

Let $g^i_{kl}$ be the elements of the first ring below corresponding to the standard graded matrix units of the last ring in the chain of the isomorphisms below.
\[g_iSg_i\cong_{\gr}\End_S(g_iS)\cong_{\gr}\End_S(\bigoplus_{k=1}^{p(i)}(\gamma_{ik})Q_i)\cong_{\gr} \M_{p(i)}(\End_{S}(Q_i))(\gamma_{i1}, \ldots,\gamma_{ip(i)})\] We have that $g_{kk}^i$ are homogeneous orthogonal idempotents with $\sum_{k=1}^{p(i)}g_{kk}^i=g_i$ and $g^i_{11}S\cong_{\gr}(\gamma_{i1})Q_i$ so that 
\[[g_{11}^iS]=[(\gamma_{i1})Q_i]=\gamma_{i1}^{-1}[Q_i]=\gamma_{i1}^{-1}f([(\gamma_{i1}^{-1})e^i_{11}R])=\gamma_{i1}^{-1}\gamma_{i1}f([e^i_{11}R])=f([e^i_{11}R]).\] Thus, the correspondence $e_{kl}^i\mapsto g_{kl}^i$ extends to a graded ring homomorphism $\phi: R\to S.$ We have that $K_0^\Gamma(\phi)=f$
since 
\[K_0^\Gamma(\phi)([e_{11}^iR])=[\phi(e_{11}^i)S]=[g_{11}^iS]=f([e^i_{11}R])\]
and $[e_{11}^iR], i=1,\ldots,n,$ generate $K_0^\Gamma(R).$

If $f$ is a morphism of $\OG^{D}_\Gamma$ not necessarily of $\OG^u_\Gamma,$ it extends to a morphism $f^u:K^\Gamma_0(R^u)\to K_0^\Gamma(S^u)$ of  $\OG^u_\Gamma$ by $[R^u]\mapsto [S^u]$. By the previous case, there is a graded ring homomorphism $\phi^u: R^u\to S^u$ such that $K_0^\Gamma(\phi^u)=f^u.$ 

If $p_R: R^u\to R^u/R$ and $p_S: S^u\to S^u/S$ denote the natural projections, we claim that $\phi^u$ maps $\ker p_R$ into $\ker p_S.$ If $r\in \ker p_R$ is a homogeneous element of $R^u,$ then $K_0^\Gamma(p_R)([rR^u])=0.$ By the definition of $f^u,$ $f^u(\ker K_0^\Gamma(p_R))\subseteq \ker K_0^\Gamma(p_S),$ and so $0=K_0^\Gamma(p_S)f^u([rR^u])=K_0^\Gamma(p_S)([\phi^u(r)S^u])=[p_S\phi^u(r)S^u],$ which implies that $p_S\phi^u(r)S^u=0$ by Lemma \ref{iso_carries_from_K0}. 
Every finitely generated graded right ideal of a graded regular ring is generated by a homogeneous idempotent (see \cite[Proposition 1.1.32]{Roozbeh_book}). So, since $S^u$ is graded regular, $xS^u=0$ implies that $x=0$ for every homogeneous element $x\in S^u.$ Hence, the condition $p_S\phi^u(r)S^u=0$ implies that $p_S\phi^u(r)=0.$ 

If $r$ is any element of $\ker p_R,$ then any homogeneous component $r_\gamma$ of $r$ is also in $\ker p_R$ since $p_R$ is a graded map. By the case when $r$ is homogeneous, $\phi^u(r_\gamma)\in \ker p_S$ for every component $r_\gamma$ and so  $\phi^u(r)\in \ker p_S.$ This proves the claim. By the claim, we can define $\phi:R\to S$ as the restriction of $\phi^u$ on the kernel of $p_R.$      
\end{proof}

We now use Proposition \ref{fullness} to prove the main result of this section. 

\begin{theorem} (Ultrasimplicial Realization)
Let $\Gamma$ be a group, let $\Delta$ be a subgroup of $\Gamma,$ and let $(G_n, u_n)$ be a sequence of simplicial $\Gamma$-groups with simplicial $\Gamma$-bases stabilized by $\Delta.$ 
\begin{enumerate}
\item If $(G,u)$ is a direct limit of $(G_n, u_n)$ in $\OG^u_\Gamma$, there is a graded division ring $K$ with $\Gamma_K=\Delta,$ a unital $\Gamma$-graded ultramatricial ring $R$ over $K,$ and an isomorphism $f$ of $\OG^u_\Gamma$ such that 
\[f:(G, u)\cong (K_0^{\Gamma}(R), [R]).\]
 
\item If $(G,D)$ is a direct limit of $(G_n, [0, u_n])$ in $\OG^D_\Gamma$, there is a graded division ring $K$ with $\Gamma_K=\Delta,$ a $\Gamma$-graded ultramatricial ring $R$ over $K,$ and an isomorphism $f$ of $\OG^D_\Gamma$ such that 
\[f:(G, D)\cong (K_0^{\Gamma}(R), D_R).\]
\end{enumerate}
\label{realization_ultrasimplicial}
\end{theorem}
\begin{proof}
The proof of (1) follows the steps of \cite[Theorem 15.24 (b)
]{Goodearl_book}. Let $g_{nm}$ be the connecting maps and let $g_n$ be the translational maps. We use induction to produce graded matricial rings $R_n,$ unit-preserving graded ring homomorphisms $\phi_{nm}: R_n\to R_m$ for every $n<m,$ and isomorphisms $f_n: (G_n, u_n)\to (K_0^{\Gamma}(R_n), [R_n])$ of $\OG^u_\Gamma$ such that $\phi_{mk}\phi_{nm}=\phi_{nk}$ for every $n<m<k,$ and $f_mg_{nm}=K_0^{\Gamma}(\phi_{nm})f_n$ for every $n<m.$

If $F$ is any field, consider it trivially graded by $\Gamma$ and let $K=F[\Delta]$ with the grading as in section \ref{Grothendieck_of_matricial} so that $\Gamma_K=\Delta.$
Use Theorem  \ref{realization_simplicial} to produce a graded  matricial ring $R_1$ over $K$ and an isomorphism  $f_1:(G_1, u_1)\to(K_0^{\Gamma}(R_1), [R_1])$ of  $\OG_\Gamma^u.$  

Assuming the induction hypothesis for $n,$ use Theorem  \ref{realization_simplicial} again to produce a graded matricial ring $R_{n+1}$ and an isomorphism $f_{n+1}: (G_{n+1}, u_{n+1})\to (K_0^{\Gamma}(R_{n+1}), [R_{n+1}])$ of $\OG_\Gamma^u.$ Then use Proposition \ref{fullness} to produce a unit-preserving graded ring map $\phi_{n(n+1)}: R_n\to R_{n+1}$ such that $K_0^{\Gamma}(\phi_{n(n+1)})=f_{n+1}g_{n(n+1)}f_{n}^{-1}.$ For any $k<n,$ define $\phi_{k(n+1)}$ as $\phi_{n(n+1)}\phi_{kn}.$
The construction and the induction hypothesis imply the conditions  $\phi_{m(n+1)}\phi_{km}=\phi_{k(n+1)}$ for every $k<m<n+1,$ and $f_{n+1}g_{m(n+1)}=K_0^{\Gamma}(\phi_{m(n+1)})f_m$ for every $m<n+1.$

Let $R$ be a direct limit of $(R_n, \phi_{nm})$ and let $\phi_n$ denote the translational maps. By construction, $R$ is a unital graded ultramatricial ring with unit $1_R=\phi_n(1_{R_n})$ for any $n.$ The map $f: (G, u)\to (K_0^{\Gamma}(R), [R])$ given by $f(g_n(x_n))=K_0^{\Gamma}(\phi_n)(f_n(x_n))$
is an isomorphism of $\OG_\Gamma^u.$

To prove (2), if $g_{nm}$ are the connecting maps, we have that $g_{nm}(u_n)\leq u_m$. Produce $R_n$ and $\phi_{n(n+1)}$ in the same way as in the proof of part (1) using Proposition \ref{fullness}. The maps $f_n$ are isomorphisms in $\OG^u_\Gamma$ but $g_{nm}$ may not be order-unit-preserving, and so $\phi_{nm}$ may fail to be unit-preserving. Let $R$ be a direct limit of $(R_n,\phi_{nm}).$ Then $f(D)=\bigcup_{n\in \omega} fg_n([0, u_n])=\bigcup_{n\in \omega}K_0^\Gamma(\phi_n)f_n([0, u_n])=\bigcup_{n\in \omega}K_0^\Gamma(\phi_n)(D_{R_n})=D_R$ so 
$f:(G, D)\to (K_0^\Gamma(R), D_R)$ is an isomorphism of $\OG_\Gamma^D.$
\end{proof}

\subsection{The bigroup case.}
If $\Gamma$ is abelian, \cite[Theorem 2.7]{Roozbeh_Lia}, stated below, can be used for the involutive version of Theorem \ref{realization_ultrasimplicial}. 

\begin{proposition} \cite[Theorem 2.7]{Roozbeh_Lia} ($*$-Fullness)
Let $\Gamma$ be an abelian group, let $K$ be a $\Gamma$-graded $*$-field such that every nonzero graded component contains a unitary element (invertible element $x$ with $x^{-1}=x^*$), let $R$ and $S$ be graded matricial $*$-algebras over $K$ and let $f:(K_0^{\Gamma}(R), D_R)\rightarrow (K_0^{\Gamma}(S), D_S)$ a morphism of $\OG^D_\Gamma.$ There is a graded $*$-algebra homomorphism $\phi:R\rightarrow S$ such that $K_0^{\Gamma}(\phi)=f$. If $f$ is order-unit-preserving, then $\phi$ can be chosen to be unit-preserving.  
\label{fullness_involution}
\end{proposition}

The following theorem is a corollary of Theorem \ref{realization_ultrasimplicial} and Proposition \ref{fullness_involution}. 

\begin{theorem} (Ultrasimplicial Bigroup Realization) 
Let $\Gamma$ be an abelian group, let $\Delta$ be a subgroup of $\Gamma,$ and let $(G_n, u_n)$ be a sequence of simplicial $\Gamma$-groups with simplicial $\Gamma$-bases stabilized by $\Delta.$ 
\begin{enumerate}
\item If $(G,u)$ is a direct limit of $(G_n, u_n)$ in $\OG^{u*}_\Gamma$, there is a graded $*$-field $K$ such that every nonzero graded component contains a unitary element and $\Gamma_K=\Delta,$ there is a unital $\Gamma$-graded ultramatricial $*$-algebra $R$ over $K,$ and there is an isomorphism $f$ of $\OG^{u*}_\Gamma$ such that $f:(G, u)\cong (K_0^{\Gamma}(R), [R]).$
 
\item If $(G,D)$ is a direct limit of $(G_n, [0, u_n])$ in $\OG^{D*}_\Gamma$, there is a graded $*$-field $K$ as in (1), there is a $\Gamma$-graded ultramatricial $*$-algebra $R$ over $K,$ and there is an isomorphism $f$ of $\OG^{D*}_\Gamma$ such that 
$f:(G, D)\cong (K_0^{\Gamma}(R), D_R).$
\end{enumerate}
\label{realization_ultrasimplicial_star}
\end{theorem}
\begin{proof}
Note that if $\Gamma$ is abelian, then $K=F[\Delta]$ from the proof of Theorem \ref{realization_ultrasimplicial} is a graded $*$-field such that  $\delta$ is a unitary element in $K_\delta=F\{\delta\}$ for every $\delta\in \Gamma_K=\Delta.$ The proof of Theorem \ref{realization_ultrasimplicial} transfers to the involutive case step by step  except that we are working in the category $\OG_\Gamma^{u*}$ instead of $\OG_\Gamma^u$ in part (1) and in $\OG_\Gamma^{D*}$ instead of $\OG_\Gamma^D$ in part (2). This is possible because we can use Proposition \ref{fullness_involution} instead of Proposition \ref{fullness}. Theorem \ref{realization_simplicial} was formulated for $\Gamma$-$\Zset_2$-bigroups also. 
\end{proof}

\section{Dimension \texorpdfstring{$\Gamma$}{TEXT}-groups}\label{section_dimension_group}

\subsection{Review of the trivial case.}
To motivate our definition of a dimension $\Gamma$-group, we review the relevant definitions in the case when $\Gamma$ is trivial. Recall that 
a pre-ordered abelian group $G$ is {\em unperforated} if the following condition holds. 
\begin{itemize}
\item For every positive integer $n$ and every $x\in G,$ $nx\geq 0$ implies that $x\geq 0.$
\end{itemize}
A pre-ordered abelian group $G$ is a {\em dimension group} if it is directed, partially ordered, unperforated and has the interpolation property. 

By \cite[Proposition 3.15]{Goodearl_interpolation_groups_book} and the paragraph following its proof, a directed and ordered group $G$ is a dimension group if and only if the following condition holds.
\begin{itemize}
\item[(SDP)] If $\sum_{i=1}^n a_ix_i=0$ for some $a_i\in \Zset$ and $x_i\in G^+,$ then there are a positive integer $m$, $b_{ij}\in\Zset^+,$ and $y_j\in G^+$ for $i=1,\ldots n$ and $j=1\ldots,m$ such that 
\begin{center}
$x_i=\sum_{j=1}^m b_{ij}y_j$ for all $i=1,\ldots, n\;\;$ and $\;\;\sum_{i=1}^n a_ib_{ij}=0$ for all $j=1,\ldots, m.$ 
\end{center}
\end{itemize}

This property, referred to as the {\em strong decomposition property} in \cite{Goodearl_interpolation_groups_book}, is the key to the proof of the main structural theorem of dimension groups: every dimension group is isomorphic to a direct limit of a directed system of simplicial groups. 

\subsection{Generalizing the properties of being unperforated and satisfying (SDP)}
First, we look for a meaningful way of defining the property of being an unperforated $\Gamma$-group. Given the trend of earlier definitions, one would expect that a pre-ordered $\Gamma$-group $G$ is unperforated if for every $x\in G$ and every $a\in \Zset^+[\Gamma],$ $ax\in G^+$ implies that $x$ is in $G^+.$ However, this property fails already in some rather basic cases which we would prefer not to rule out. For example, if $\Gamma=\langle x |x^2=1\rangle\cong \Zset_2$ and $G=\Zset[x]/\langle x^2=1\rangle\cong \Zset[\Zset_2],$ $G$ is a simplicial $\Gamma$-group and we would like it to be unperforated. However, $1+x\in\Zset^+[\Gamma],$ $(1+x)(1-x)=1-x^2=0\in G^+$ and $1-x\notin G^+.$

To obtain a more meaningful definition, let us consider the case when $G$ is a simplicial $\Gamma$-group with a simplicial $\Gamma$-basis $\{x_1,\ldots, x_n\}$ stabilized by $\Delta$ and let $ax\in G^+$ for $a\in\Zset^+[\Gamma]$ and $x=\sum_{j=1}^n b_jx_j\in G$ for some $b_j\in \Zset[\Gamma], j=1,\ldots, n.$ Then we have that $ax=\sum_{j=1}^n ab_jx_j\geq 0$ which implies that $\pi(ab_j)\geq 0$ for every $j=1,\ldots,n$ by Definition \ref{simplicial_cone_definition}. This motivates the following.

\begin{definition} Let $\Delta$ be a subgroup of $\Gamma$ and let $\pi:\Zset[\Gamma]\to\Zset[\Gamma/\Delta]$ be the natural $\Zset[\Gamma]$-module map.
A pre-ordered $\Gamma$-group $G$  is {\em unperforated} with respect to a subgroup $\Delta$ of $\Gamma$ if $ax\in G^+$ for some $x\in G$ and $a\in\Zset^+[\Gamma]$ implies that there are a positive integer $m,$ elements $y_j\in G^+$ and $b_j\in \Zset[\Gamma],$ $j=1,\ldots, m$ such that 
\[x=\sum_{j=1}^m b_jy_j\;\;\mbox{ and }\;\;\pi(ab_j)\geq 0\;\;\mbox{ for all }\;\;j=1,\ldots,m.\]
\end{definition}

Note that if $\Gamma$ is trivial, then this definition is equivalent to the condition that $nx\in G^+$ for $n\in \Zset^+$ implies $x\in G^+.$ Indeed, if $nx\in G^+$ for $n\in \Zset^+$ implies $x\in G^+,$ then one can choose $m=1,$ $b_1=1$ and $y_1=x.$ Conversely, assume that $nx\in G^+$ and   
there are a positive integer $m,$ $y_j\in G^+$ and $b_j\in \Zset,$ $j=1,\ldots, m,$ such that $x=\sum_{j=1}^m b_jy_j$ and $nb_j\geq 0$ for all $j=1,\ldots,m.$ The conditions $nb_j\geq 0$ and $n\geq 0$ imply that $b_j\geq 0$ and so $b_jy_j\in G^+$ for every $j.$ This implies that $x=\sum_{j=1}^m b_jy_j$ is in $G^+.$

It may also appear that it is sufficient to require that $m=1$ in the definition of an unperforated $\Gamma$-group. The following example illustrates that this is not sufficient.

\begin{example}
Let $\Gamma=\langle x | x^2=1\rangle\cong \Zset_2$ and $G=\Zset[\Gamma]\oplus\Zset[\Gamma].$ Then $G$ is a simplicial $\Gamma$-group with a simplicial $\Gamma$-basis $\{(1,0), (0,1)\}$ stabilized by $\{1\}.$ As such, it is unperforated with respect to $\{1\}$ by Proposition \ref{unperforated} (see below). Consider the elements $a=1+x\in \Zset^+[\Gamma]$ and $u=(1-x, 2-x)\in G$ such that $au=(0, 1+x)\in G^+.$ Assuming that there are $b=b_1+b_2x\in \Zset[\Gamma]$ and $y=(y_1,y_2)=(m_1+m_2x, n_1+n_2x)\in G^+$ such that $u=by$ and $ab\geq 0,$ we claim that we arrive at a contradiction. The relation $ab\geq0$ implies that $b_1+b_2\geq 0.$ The relation $by=u$ implies that $(b_1+b_2x)(m_1+m_2x)=1-x$ and $(b_1+b_2x)(n_1+n_2x)=2-x.$ The equation $(b_1+b_2x)(m_1+m_2x)=1-x$ implies $b_1m_1+b_2m_2=1$ and $b_2m_1+b_1m_2=-1.$ From this system of equations we obtain that $(b_2^2-b_1^2)m_2=b_2+b_1$ and $(b_1^2-b_2^2)m_1=b_2+b_1$ which implies that $b_1+b_2=0$ since $m_1\geq 0$ and $m_2\geq 0.$ Hence $b_2=-b_1$ and so  $b_1(m_1-m_2)=1$ which implies that $b_1=\pm 1.$  
With these values, the equation  $(b_1+b_2x)(n_1+n_2x)=2-x$ becomes  $\pm 1(1-x)(n_1+n_2x)=2-x$ which implies that  $\pm 1(n_1-n_2)=2$ and $\pm 1(-n_1+n_2)=-1.$ So, 
$2=\pm 1(n_1-n_2)=\mp1(-n_1+n_2)=1$ which is a contradiction. 
\end{example}
 
Now, we generalize the strong decomposition property as follows.

\begin{definition}
If $\Delta$ is a subgroup of $\Gamma,$ $\pi:\Zset[\Gamma]\to\Zset[\Gamma/\Delta]$ is the natural $\Zset[\Gamma]$-module map, and $G$ is a directed and ordered $\Gamma$-group, then $G$ has the {\em strong decomposition property} with respect to $\Delta$ if the following condition holds. 
\begin{itemize}
\item[(SDP$_\Delta$)] If $\sum_{i=1}^n a_ix_i=0$ for some $a_i\in \Zset[\Gamma]$ and $x_i\in G^+,$ then there are a positive integer $m$, $b_{ij}\in\Zset^+[\Gamma]$ and $y_j\in G^+$ for $i=1,\ldots n, j=1\ldots,m,$ such that \[x_i=\sum_{j=1}^m b_{ij}y_j\;\;\mbox{ for all }\;\;i=1,\ldots, n\;\;\mbox{ and }\;\;\sum_{i=1}^n \pi(a_ib_{ij})=0\;\;\mbox{ for all }\;\;j=1,\ldots, m.\] 
\end{itemize}
\label{SDP_Delta_definition}
\end{definition}

If $\Gamma$ is trivial, it is direct to check that (SDP$_{\{1\}}$) is equivalent to (SDP). Also, if $G$ is a directed and ordered $\Gamma$-group which satisfies (SDP$_\Delta$) for some subgroup $\Delta$ of $\Gamma$ then $G$ satisfies (SDP$_{\Delta'}$) for every subgroup $\Delta'$ of $\Gamma$ which contains $\Delta.$ Indeed, if $\pi'$ and $\pi_{\Delta\Delta'}$ denote the natural $\Zset[\Gamma]$-module homomorphisms $\Zset[\Gamma]\to\Zset[\Gamma/\Delta']$ and $\Zset[\Gamma/\Delta]\to\Zset[\Gamma/\Delta']$ respectively, then $\pi'=\pi_{\Delta\Delta'}\pi$ so the claim follows. Similarly, if $G$ is unperforated with respect to $\Delta$ then $G$ is unperforated with respect to $\Delta'.$ Thus, if $G$ is a directed and ordered $\Gamma$-group which satisfies (SDP$_\Delta$) for some subgroup $\Delta$ or is unperforated with respect to $\Delta,$ then the {\em smallest} such $\Delta$ is the most relevant. 

\begin{lemma}
A simplicial $\Gamma$-group with a simplicial $\Gamma$-basis stabilized by $\Delta$ satisfies (SDP$_{\Delta}$). As a corollary, an ultrasimplicial $\Gamma$-group which is a direct limit of simplicial $\Gamma$-groups with simplicial $\Gamma$-bases stabilized by $\Delta$ satisfies (SDP$_\Delta$).  
\label{lemma_simplicial_SDP}
\end{lemma}
\begin{proof}
Let $G$ be a simplicial $\Gamma$-group with a simplicial $\Gamma$-basis $X=\{e_1, \ldots, e_m\}$ and $\Delta=\Stab(X).$ Assume that $\sum_{i=1}^n a_ix_i=0$ for some $a_i\in\Zset[\Gamma]$ and $x_i\in G^+$ for $i=1,\ldots, n.$ By part (3) of Proposition \ref{properties_of_simplicial}, we can write $x_i$ as $\sum_{j=1}^m b_{ij}e_j$ for some $b_{ij}\in\Zset^+[\Gamma]$. Then $$0=\sum_{i=1}^n a_ix_i=\sum_{i=1}^n a_i\left(\sum_{j=1}^m b_{ij}e_j\right)=\sum_{j=1}^m\left(\sum_{i=1}^n a_i b_{ij}\right)e_j$$
which implies that $\sum_{i=1}^n \pi(a_i b_{ij})=0$ so $G$ satisfies (SDP$_{\Delta}$).  

Having (SDP$_\Delta$) is preserved under taking direct limits, so the second sentence of the lemma follows. 
\end{proof}
 
\begin{proposition}
If $G$ is an ordered and directed $\Gamma$-group which satisfies (SDP$_\Delta$), then $G$ is unperforated with respect to $\Delta.$ 
\label{unperforated} 
\end{proposition}
\begin{proof}
Let $ax\in G^+$ for some $x\in G$ and $a\in \Zset^+[\Gamma].$ Since $G$ is directed, we can represent $x$ as $x^+-x^-$ for some $x^+, x^-\in G^+.$ Let $x_1=x^+, x_2=x^-,$ and $x_3=ax,$ so that $x_i\in G^+$ for $i=1,2,3.$ The relation $ax_1-ax_2-x_3=0$ implies that there are a positive integer $m$, elements $b_{ij}\in\Zset^+[\Gamma]$ and $y_j\in G^+$ for $i=1,2,3, j=1\ldots,m$ such that $x_i=\sum_{j=1}^m b_{ij}y_j$ for $i=1,2,3$ and $\pi(ab_{1j}-ab_{2j}-b_{3j})=0$ for all $j=1,\ldots, m.$ Hence $\pi(b_{3j})=\pi(a(b_{1j}-b_{2j}))$ and $\pi(b_{3j})\geq 0$ since $b_{3j}\geq 0.$ For $b_j=b_{1j}-b_{2j}$ we have that $\pi(ab_j)=\pi(b_{3j})\geq 0$ for all $j=1,\ldots,m,$ and $x=x^+-x^-=\sum_{j=1}^m (b_{1j}-b_{2j})y_j=\sum_{j=1}^m b_jy_j.$
\end{proof}

\subsection{Defining a dimension \texorpdfstring{$\Gamma$}{TEXT}-group}

If $G$ is a $\Gamma$-group as in Proposition \ref{unperforated}, we would like to have that $G$ has interpolation, just like in the case when $\Gamma$ is trivial. This can be shown under an additional assumption: that $\Delta$ is contained in the stabilizer of $G$ for some $\Delta$ for which (SDP$_\Delta$) holds. Since $\Stab(G)$ is a normal subgroup of $\Gamma,$ if $\Delta\subseteq \Stab(G),$ then the {\em normal closure} $\ol \Delta$ of $\Delta$ (i.e. the normal subgroup of $\Gamma$ generated by $\Delta$) is in $\Stab(G)$ as well. Before proving Proposition \ref{interpolation}, we show a short technical lemma. 

\begin{lemma}
If $G$ is a $\Gamma$-group, $\Delta$ is a subgroup of $\Gamma$ such that $\Delta\subseteq \Stab(G),$ and $\pi:\Zset[\Gamma]\to \Zset[\Gamma/\Delta]$ is the natural $\Zset[\Gamma]$-module map, then $\pi(a)=0$ implies $ax=0$ for all $a\in\Zset[\Gamma]$ and $x\in G.$
\label{stabilizer_lemma}
\end{lemma}
\begin{proof}
Let $\{\gamma_j\,|\,j\in J\}$ be a set of left coset representatives for $\Gamma/\Delta.$ For $a\in \Zset[\Gamma],$ let $J_0$ be a finite subset of $J$ and let $\Gamma_j$ be a finite subset of $\gamma_j\Delta$ such that $a=\sum_{j\in J_0}\sum_{\gamma\in\Gamma_j} k_\gamma\gamma.$ The condition $\pi(a)=0$ implies that $\sum_{\gamma\in\Gamma_j} k_\gamma=0$ for all $j\in J_0.$ The condition  $\Delta\subseteq \Stab(G)$ implies that $\gamma x=\gamma_jx$ for all $\gamma\in\Gamma_j$ and any $x\in G$ so that $ax=\sum_{j\in J_0}\sum_{\gamma\in\Gamma_j} k_\gamma\gamma x=\sum_{j\in J_0}\sum_{\gamma\in\Gamma_j} k_\gamma\gamma_j x=\sum_{j\in J_0}\left(\sum_{\gamma\in\Gamma_j} k_\gamma\right)\gamma_j x=0.$
\end{proof}

\begin{proposition}
If $G$ is an ordered and directed $\Gamma$-group which satisfies (SDP$_\Delta$) for some $\Delta\subseteq \Stab(G),$ then $G$ has the interpolation property. 
\label{interpolation} 
\end{proposition}
\begin{proof}
The proof is the same as the proof of Lemma \ref{lemma_simplicial_interpolation} except that  
elements of $G^+$ which exist by (SDP$_\Delta$) replace the simplicial $\Gamma$-basis elements of the proof of Lemma \ref{lemma_simplicial_interpolation} and the condition $\Delta\subseteq \Stab(G)$ ensures the implication $\pi(a)=0 \Rightarrow ax=0$ for $x\in G$ by Lemma \ref{stabilizer_lemma}. This implication is needed to replace the condition (Ind) in the proof of Lemma \ref{lemma_simplicial_interpolation}. 
\end{proof}

Propositions \ref{unperforated} and \ref{interpolation} motivate the definition of a dimension $\Gamma$-group below. 

\begin{definition}
If $G$ is a directed and ordered $\Gamma$-group then $G$ is a {\em dimension $\Gamma$-group (or a dimension $\Zset[\Gamma]$-module)} if $G$ has the strong decomposition property with respect to some subgroup $\Delta$ of $\Gamma$ such that $\Delta\subseteq \Stab(G).$ 
\label{dimension_definition}
\end{definition}

If $\Gamma$ is trivial, Definition \ref{dimension_definition} coincides with the usual definition because any ordered and directed group has (SDP) if and only if it is unperforated and  has the interpolation property. By Lemma \ref{lemma_simplicial_SDP}, all ultrasimplicial $\Gamma$-groups satisfy (SDP$_\Delta$) for some subgroup $\Delta$ of $\Gamma.$ However, such $\Delta$ cannot always be found inside the stabilizer (see part (2) of Example \ref{example_free_permutation_module}). Thus, a simplicial $\Gamma$-group $G$ may fail to be a dimension $\Gamma$-group. If $G$ has a simplicial $\Gamma$-basis with a normal stabilizer, then $G$ is a dimension $\Gamma$-group by Proposition \ref{Delta_normal_case} so the following proposition holds. 

\begin{proposition}
If $G$ is an ultrasimplicial $\Gamma$-group which is a direct limit of a sequence of simplicial $\Gamma$-groups with simplicial $\Gamma$-bases stabilized by a normal subgroup of $\Gamma$, then $G$ is a countable dimension $\Gamma$-group. In particular, if $\Gamma$ is abelian, every  ultrasimplicial $\Gamma$-group is a countable dimension $\Gamma$-group.
\label{ultrasimplicial_is_dimension}
\end{proposition}

The statement ``simplicial groups are building blocks of every dimension group'' remains to hold in the case when $\Zset[\Gamma]$ is Noetherian. Under this assumption, we show Theorem \ref{any_dim_gr} stating that every dimension $\Gamma$-group is a direct limit of a directed system of simplicial $\Gamma$-groups which can be constructed so that their simplicial $\Gamma$-bases are stabilized by a normal subgroup of $\Gamma.$

\subsection{Dimension \texorpdfstring{$\Gamma$}{TEXT}-groups as direct limits of ultrasimplicial \texorpdfstring{$\Gamma$}{TEXT}-groups}\label{subsection_dim_as_lim_simplicial}

The main objective of this section is to show Theorem \ref{any_dim_gr}. The assumption that $\Zset[\Gamma]$ is a Noetherian ring enables us to adapt the steps of the proof of the classic result (see \cite[Section 3.4]{Goodearl_interpolation_groups_book}) to the case when $\Gamma$ is nontrivial. First, we show that the appropriately modified Shen criterion, Proposition \ref{telescoping}, holds if $\Zset[\Gamma]$ is Noetherian. Then we show that Proposition \ref{telescoping} implies Theorem \ref{countable_dim_gr} which then implies Theorem \ref{any_dim_gr}. We also note that the proof of Proposition \ref{telescoping} requires only that $\Zset[\Gamma]$ is {\em left} Noetherian. However, $\Zset[\Gamma]$ is left Noetherian if and only if it is right Noetherian (the natural involution is an isomorphism between $\Zset[\Gamma]$ and the opposite ring $\Zset[\Gamma]^{op}$), so we are not assuming more than we need in the proof. 

If $\Gamma$ is  polycyclic-by-finite (has a normal subgroup of finite index which admits a subnormal series with cyclic factors) then $\Zset[\Gamma]$ is a Noetherian ring by Hall (\cite{Hall}). Thus, if $\Gamma$ is finite or finitely generated abelian, then  $\Zset[\Gamma]$ is Noetherian. If $\Gamma$ is abelian, then $\Zset[\Gamma]$ is Noetherian exactly when $\Gamma$ is finitely generated. Indeed, if $\Zset[\Gamma]$ is Noetherian and $\Gamma$ is abelian, then $\Gamma$ is a Noetherian group (each subgroup is finitely generated) which implies that $\Gamma$ is finitely generated itself. 

\begin{proposition} (Generalized Shen criterion) If $\Gamma$ is such that $\Zset[\Gamma]$ is Noetherian, $G_1$ is a simplicial $\Gamma$-group with a simplicial $\Gamma$-basis stabilized by a normal subgroup $\Delta$ of $\Gamma$, $G$ is a dimension $\Gamma$-group which satisfies (SDP$_\Delta$) and $\Delta\subseteq \Stab(G)$, and if $g_1:G_1\to G$ is a morphism of $\OG_\Gamma,$ then there is a simplicial $\Gamma$-group $G_2$ with a simplicial $\Gamma$-basis stabilized by $\Delta$ and morphisms $g_{12}: G_1\to G_2$ and $g_2: G_2\to G$ in $\OG_\Gamma$ such that $g_1=g_2g_{12}$ and $\ker g_{12}=\ker g_1.$ 
\label{telescoping}
\end{proposition}
This proposition is the key for the inductive argument of creating a ``telescope'' of simplicial $\Gamma$-groups $G_1\to G_2\to\ldots$ with $G$ as their limit. 
Given $g_1: G_1\to G$ as in the proposition, one would like to take $G_2=G_1/\ker g_1$ but this may fail to be simplicial. The strong decomposition property guarantees that one can find an appropriate simplicial $\Gamma$-group $G_2$.  
\begin{proof}
The proof parallels the case when $\Gamma$ is trivial given in \cite[Proposition 3.16]{Goodearl_interpolation_groups_book} (also \cite[Lemma IV.7.1]{Davidson}). We modify some arguments using definitions and results of this paper when necessary. First we show the following claim.  

{\bf Claim.} If $z_1, \ldots, z_k$ are in $\ker g_1,$ then there is a simplicial $\Gamma$-group $G_2$ with a simplicial $\Gamma$-basis stabilized by $\Delta$ and there are morphisms $g_{12}: G_1\to G_2,$ and $g_2: G_2\to G$ in $\OG_\Gamma$ such that $g_1=g_2g_{12}$ and $z_1,\ldots, z_k\in\ker g_{12}.$  

Clearly if $G_1$ is trivial, one can take $G_2=G_1,$ and $g_{12}, g_2$ to be the trivial maps. Hence assume that $G_1$ is nontrivial and let $\{e_1, \ldots, e_n\}$ be a simplicial $\Gamma$-basis stabilized by $\Delta$. We prove the claim by induction on $k$. If $k=1,$ let $z=z_1\in\ker g_1.$ Then $z=\sum_{i=1}^n a_ie_i$ for some $a_i\in \Zset[\Gamma], i=1,\ldots,n.$
Let $x_i=g_1(e_i)\in G^+.$ Since $\sum_{i=1}^n a_ix_i=g_1(\sum_{i=1}^n a_ie_i)=g_1(z)=0,$ there are a positive integer $m$, $b_{ij}\in\Zset^+[\Gamma],$ and $y_j\in G^+$ for $i=1,\ldots n, j=1\ldots,m$ such that $x_i=\sum_{j=1}^m b_{ij}y_j$ for all $i=1,\ldots, n$ and $\sum_{i=1}^n \pi(a_ib_{ij})=0$ for all $j=1,\ldots, m$ by (SDP$_\Delta$).  

Let $G_2$ be a simplicial $\Gamma$-group with a simplicial $\Gamma$-basis $\{f_1,\ldots, f_m\}$ stabilized by $\Delta$ (for example, one such group is given in Example  \ref{example_free_permutation_module}). Let $g_{12}(e_i)=\sum_{j=1}^m b_{ij}f_j,$ and $g_2(f_j)=y_j$ for $j=1,\ldots, m$ and extend these maps to morphisms in $\OG_\Gamma.$ By the assumption that $\Delta$ is normal, $\pi$ is both a left and a right $\Zset[\Gamma]$-module homomorphism. Thus, the extension $g_{12}$ is well-defined since if $\sum_{i=1}^n c_ie_i=0,$ then $\pi(c_i)=0$ for all $i=1,\ldots, n$ and then $\pi(c_ib_{ij})=0$ for all $i=1,\ldots,n$ and $j=1,\ldots, m$ so that $g_{12}(\sum_{i=1}^n c_ie_i)=\sum_{j=1}^m\left(\sum_{i=1}^n c_ib_{ij}\right)f_j=0.$ The extension $g_2$ is well-defined since if $\sum_{j=1}^m c_jf_j=0,$ then $\pi(c_j)=0$ for all $j=1,\ldots, m,$ and so $c_jy_j=0$ for all $j=1,\ldots, m$ by Lemma \ref{stabilizer_lemma}. This implies $g_2(\sum_{j=1}^m c_jf_j)=\sum_{j=1}^m c_jy_j=0.$ The maps $g_{12}$ and $g_2$ are order-preserving since $b_{ij}\in\Zset^+[\Gamma]$ and $y_j\in G^+.$ By the definitions of these maps, we have that 
$$g_2g_{12}(e_i)=g_2\left(\sum_{j=1}^m b_{ij}f_j\right)=\sum_{j=1}^m b_{ij}y_j=x_i=g_1(e_i)$$ 
and so $g_2g_{12}=g_1.$ We have that $z\in \ker g_{12}$ since 
$$g_{12}(z)=g_{12}(\sum_{i=1}^n a_ie_i)=\sum_{i=1}^n a_i\left(\sum_{j=1}^m b_{ij}f_j\right)=\sum_{j=1}^m\left(\sum_{i=1}^n a_i b_{ij}\right)f_j=0$$
where the last equality holds since $\pi\left(\sum_{i=1}^n a_i b_{ij}\right)=0.$ 

Assuming that the assumption holds when considering $k$ elements in $\ker g_1,$ let $z_1,\ldots, z_k, z_{k+1}$ be in $\ker g_1.$ Let $H$ be a simplicial $\Gamma$-group with a simplicial $\Gamma$-basis stabilized by $\Delta$ which exists by the induction hypothesis for $z_1,\ldots, z_k,$ and let $g_{1H}: G_1\to H$ and $ g_H: H\to G$ be 
morphisms in $\OG_\Gamma$ such that $g_Hg_{1H}=g_1$ and $z_1,\ldots, z_k\in \ker g_{1H}.$ Then $0=g_1(z_{k+1})=g_Hg_{1H}(z_{k+1})$ and so $g_{1H}(z_{k+1})\in \ker g_H.$ By the case $k=1,$ there are a simplicial $\Gamma$-group $G_2$ with a simplicial $\Gamma$-basis stabilized by $\Delta$ and $g_2:G_2\to G, g_{H2}: H\to G_2$ such that $g_H=g_2g_{H2}$ and $g_{1H}(z_{k+1})\in \ker g_{H2}.$ Let $g_{12}=g_{H2}g_{1H}.$ Then $g_2g_{12}=g_2g_{H2}g_{1H}=g_Hg_{1H}=g_1$ and $z_1,\ldots, z_{k+1}\in \ker g_{12}.$ This finishes the proof of the claim.   

To prove the proposition now, we use the assumption that $\Zset[\Gamma]$ is a Noetherian ring so that every finitely generated $\Zset[\Gamma]$-module has its submodules finitely generated. In particular, any $\Zset[\Gamma]$-submodule of any simplicial $\Gamma$-group is finitely generated. As a result, $\ker g_1$ is finitely generated as a $\Zset[\Gamma]$-module. Let $\{z_1, \ldots, z_k\}$ be generators of $\ker g_1.$ If $G_2, g_{12}$ and $g_2$ are as in the claim, then $\ker g_1\subseteq \ker g_{12}$ follows from  $z_1,\ldots, z_k\in\ker g_{12}$ and $\ker g_{12}\subseteq\ker g_1$ follows from $g_1=g_2g_{12}.$ 
\end{proof}

\begin{theorem}
Let $\Gamma$ be a group such that $\Zset[\Gamma]$ is Noetherian and let $G$ be a countable dimension $\Gamma$-group which satisfies (SDP$_\Delta$) for some $\Delta$ with $\Delta\subseteq \Stab(G)$. By replacing $\Delta$ with its normal closure in $\Gamma$, we can assume that $\Delta$ is normal.  
\begin{enumerate}
\item There is a countable sequence of simplicial $\Gamma$-groups $G_n$ with simplicial $\Gamma$-bases stabilized by $\Delta$ and an isomorphism between $G$ and a direct limit of the $\Gamma$-groups $G_n$ in $\OG_\Gamma.$
 
\item If $(G, u)$ is an object of $\OG_\Gamma^u,$ the simplicial $\Gamma$-groups and the isomorphism from (1) can be found in $\OG_\Gamma^u.$

\item If $(G, D)$ is an object of $\OG^D_\Gamma$, the simplicial $\Gamma$-groups and the isomorphism from (1) can be found in $\OG_\Gamma^D.$
\end{enumerate}
If $G$ is considered to be a $\Gamma$-$\Zset_2$-bigroup with the trivial $\Zset_2$-action, the simplicial $\Gamma$-groups and isomorphisms in the three parts can be found in the categories $\OG_\Gamma^*, \OG_\Gamma^{u*},$ and $\OG_\Gamma^{D*}$ respectively.  
\label{countable_dim_gr}
\end{theorem}
\begin{proof}
If $\Gamma$ is trivial, the statement becomes \cite[Theorem 3.17 and Corollary 3.18]{Goodearl_interpolation_groups_book} (parts (1) and (2))  and \cite[Theorem IV.7.3]{Davidson} (parts (2) and (3)). Our proof follows the ideas of the proofs of these results, modifying the arguments using the results of this paper.  

To show (1), we use the assumption that $G$ is countable to index the elements of $G^+$ by the infinite countable cardinal $\omega$ so that $G^+=\{x_n\, |\, n\in \omega\}$ with repetitions of elements allowed. We use induction to construct a sequence of simplicial $\Gamma$-groups $G_n, n\in \omega,$ with simplicial $\Gamma$-bases stabilized by $\Delta$ and morphisms $g_{n(n+1)}: G_n\to G_{n+1}$ and $g_n: G_n\to G$ of $\OG_\Gamma$ such that the following three conditions hold $g_{n+1}g_{n(n+1)}=g_n,$ $\ker g_n=\ker g_{n(n+1)}$ and $x_n\in g_n(G_n^+)$. Let $G_0=\Zset[\Gamma/\Delta]$ and let $g_0$ be the $\Gamma$-group homomorphism extension of the map $\Delta\mapsto x_0.$ This extension is well-defined since if $a\Delta=0$ for some $a\in \Zset[\Gamma],$ then $\pi(a)=0.$ The assumption $\Delta\subseteq\Stab(G)$ ensures that $ax_0=0$ by Lemma \ref{stabilizer_lemma}. The map $g_0$ is order-preserving since $x_0$ is in $G^+.$  

Now let us assume that $G_i, g_i, i=0,\ldots, n$ and $g_{i(i+1)}, i=0, \ldots, n-1$ have been defined and that they have the three required properties. Let $H=G_n\oplus \Zset[\Gamma/\Delta]$ so that $H$ is a simplicial $\Gamma$-group with a simplicial $\Gamma$-basis stabilized by $\Delta$ and let $h: H\to G$ be the $\Gamma$-group homomorphism extension of the map $(x,0)\mapsto g_n(x)$ and $(0, \Delta)\mapsto x_{n+1}.$ This extension is well-defined by the assumption that $\Delta\subseteq\Stab(G)$ and $h$ is order-preserving since $g_n$ is order-preserving and $x_{n+1}\in G^+.$ By Proposition \ref{telescoping}, there is a simplicial $\Gamma$-group $G_{n+1}$ with a simplicial $\Gamma$-basis stabilized by $\Delta$ and morphisms $g_{H}: H\to G_{n+1}$ and $g_{n+1}: G_{n+1}\to G$ of $\OG_\Gamma$ such that $g_{n+1}g_{H}=h$ and $\ker g_H=\ker h.$ Let $i: G_n\to H$ be the natural injection and let $g_{n(n+1)}=g_Hi.$ The three desired properties are satisfied by definitions. 

Let $g_{nm}=g_{n(n+1)}\ldots g_{(m-1)m}$ for $n<m$ and let $F$ be a direct limit of $(G_n, g_{nm})$ in $\OG_\Gamma$ which exists by Proposition \ref{direct_limits}. Let $f_n: G_n\to F$ be the translational maps and let $f: F\to G$ be a unique morphism of $\OG_\Gamma$ such that $ff_n=g_n.$ Thus we have that $f(F^+)\subseteq G^+.$ The converse holds since  $x_n\in g_n(G_n^+)=ff_n(G_n^+)\subseteq f(F^+)$ for every $n,$ and so $G^+\subseteq f(F^+).$ This also implies that $f$ is onto since $G$ is directed. To show that $f$ is injective, let $x\in F$ be such that $f(x)=0.$ Write $x$ as $f_n(y_n)$ for some $y_n\in G_n.$ Since $g_n(y_n)=ff_n(y_n)=f(x)=0,$ $y_n\in \ker g_n=\ker g_{n(n+1)}.$ Thus $x=f_n(y_n)=f_{n+1}g_{n(n+1)}(y_n)=f_{n+1}(0)=0.$ This shows that $f$ is an isomorphism of $\OG_{\Gamma}.$

To show (2), let $u$ be an order-unit of $G,$ let $H_n, n\in \omega,$ be simplicial $\Gamma$-groups, and let $h_{nm},$ $n<m,$ and $h_n$ be morphisms which exist by part (1) in the category $\OG_\Gamma$. Since $u\in G^+,$ there is $n_0$ and $u_{n_0}\in H_{n_0}$ such that $u=h_{n_0}(u_{n_0}).$ Since $G$ is a direct limit of $(H_n, h_{nm}),$ with $n_0\leq n<m,$ without loss of generality we can assume that $n_0=1.$ Let $u_n=h_{1n}(u_1)$ so that $u=h_1(u_1)=h_nh_{1n}(u_1)=h_n(u_n)$ for every $n>1$ and $u_n=h_{1n}(u_1)=h_{kn}h_{1k}(u_1)=h_{kn}(u_k)$ for every $1<k<n.$ 

Let $G_n=\{x\in H_n \,|\, -au_n\leq x\leq au_n\mbox{ for }a\in \Zset^+[\Gamma]\}$ be the convex $\Gamma$-subgroup of $H_n$ generated by $\{u_n\}$ so that $(G_n, u_n)$ is an object of $\OG^u_\Gamma.$ The group $G_n$ is directed since for $x,y\in G_n$ with $-au_n\leq x\leq au_n$ and $-bu_n\leq y\leq bu_n$ for some $a,b\in \Zset^+[\Gamma],$ $x,y\leq(a+b)u_n\in G_n.$ So $G_n$ is a $\Gamma$-ideal of $H_n.$ By Proposition \ref{ideals_in_simplicial}, $G_n$ is a simplicial $\Gamma$-group with a simplicial $\Gamma$-basis stabilized by $\Delta$. Applying the order-preserving map $h_{n(n+1)}$ to the relation $-au_n\leq x\leq au_n,$ we obtain that $-au_{n+1}\leq h_{n(n+1)}(x)\leq au_{n+1}$ so we have that $h_{n(n+1)}(G_n)\subseteq G_{n+1}.$ Let $g_{n(n+1)}$ be the restriction of $h_{n(n+1)}$ on $G_n.$ This produces a directed system of simplicial $\Gamma$-groups $((G_n, u_n), g_{nm}).$  

Let $g_n$ be the restriction of $h_n$ to $G_n.$ Clearly $\bigcup_{n\in\omega} g_n(G_n)\subseteq G$ and  $\bigcup_{n\in\omega} g_n(G_n^+)\subseteq G^+.$ If $x\in G^+,$ then $x=h_n(y_n)$ for some $y_n\in H_n^+$ and there is $a\in \Zset^+[\Gamma]$ such that $0\leq x\leq au.$ Hence  $0\leq h_n(y_n)\leq ah_n(u_n).$
This implies the existence of $m>n$ such that $0\leq h_{nm}(y_n)\leq ah_{nm}(u_n)=au_m.$ Then, $h_{nm}(y_n)\in G_m^+$ and so $x=h_mh_{nm}(y_n)=g_m(h_{nm}(y_n))\in  g_m(G_m^+).$ This shows that $\bigcup_{n\in\omega} g_n(G_n^+)= G^+.$ As $G$ is directed, this implies that $\bigcup_{n\in\omega} g_n(G_n)= G.$ Hence, $(G, u)$ is a direct limit of $((G_n, u_n), g_{nm})$ in $\OG_{\Gamma}^u.$ Note that the map $g_{nm}$ maps $[0, u_n]$ into $[0, u_m],$ that $g_n$ maps $[0, u_n]$ into $[0, u],$ and that $[0, u]=\bigcup_{n\in \omega}g_n([0, u_n]).$

To show (3), let $D$ be a generating interval of $G,$ let us index the elements of $D$ so that $D=\{d_n\, |\, n\in \omega\}$ with possible repetitions, and let $H_n, h_{nm},$ and $h_n$ be groups and morphisms which exist by part (1).  

For $v_1=d_1\in D,$ there are a positive integer $k_1$ and an element $u_1\in H_{k_1}^+$ such that $h_{k_1}(u_1)=v_1.$ Let $G_1$ be the convex $\Gamma$-subgroup of $H_{k_1}$ generated by $\{u_1\}$ and let $D_1$ be $[0, u_1]$ which is a generating interval of $G_1$ by Proposition \ref{generating_interval_of_simplicial}. Let $g_1$ be the restriction of $h_{k_1}$ on $G_1$ and let $v_2\in D$ be such that $v_1, d_2\leq v_2.$ There are a positive integer $m$ and $u_2'\in H_m^+$ such that $h_m(u_2')=v_2.$ Since $0\leq v_1=h_{k_1}(u_1)\leq v_2=h_m(u_2'),$ there is an integer $k_2\geq \max\{m,k_1\}$ such that $0\leq h_{k_1k_2}(u_1)\leq h_{mk_2}(u_2').$ Let $u_2=h_{mk_2}(u_2'),$ let $G_2$ be the convex $\Gamma$-subgroup of $H_{k_2}$ generated by $\{u_2\},$ let $D_2$ be $[0, u_2]$ (thus a generating interval by Proposition \ref{generating_interval_of_simplicial}), let $g_2$ be the restriction of $h_{k_2}$ on $G_2,$ and let $g_{12}$ be the restriction of $h_{k_1k_2}$ on $G_1.$ Let $v_3$ be the element of $D$ such that $v_2, d_3\leq v_3$ and continue the construction inductively. By the construction, the translational maps $g_n$ are such that $g_n(u_n)=v_n$ and the connecting maps $g_{nm}$ such that $g_{nm}(D_n)\subseteq D_m$ for $n<m$ so that both are morphisms of $\OG_\Gamma^D.$ Analogously to  part (2), we have that $G$ is a direct limit of $(G_n, g_{nm})$ and that $\bigcup_{n\in\omega}g_{n}(D_n)=\bigcup_{n\in \omega}[0, v_n]=D.$ Thus, $(G, D)$ is a direct limit of $((G_n, D_n), g_{nm})$ in $\OG_{\Gamma}^D.$

The last sentence of the theorem follows directly since the $\Zset_2$-action is trivial. 
\end{proof}

\begin{theorem}
Let $\Gamma$ be a group such that $\Zset[\Gamma]$ is Noetherian and let $G$ be a dimension $\Gamma$-group which satisfies (SDP$_\Delta$) for some $\Delta$ with $\Delta\subseteq \Stab(G)$. By replacing $\Delta$ with its normal closure in $\Gamma$, we can assume that $\Delta$ is normal.   
\begin{enumerate}
\item There are a directed system $I,$ simplicial $\Gamma$-groups $G_i, i\in I,$ with simplicial $\Gamma$-bases stabilized by $\Delta,$ a directed system $(G_i, g_{ij})$ of $\OG_\Gamma,$ and an isomorphism between $G$ and a direct limit of $(G_i, g_{ij})$ in $\OG_\Gamma.$
 
\item If $(G, u)$ is an object of $\OG_\Gamma^u,$ the directed system and the isomorphism from (1) can be found in $\OG_\Gamma^u.$

\item If $(G, D)$ is an object of $\OG^D_\Gamma$, the directed system and the isomorphism from (1) can be found in $\OG_\Gamma^D.$
\end{enumerate}
If $G$ is considered to be a $\Gamma$-$\Zset_2$-bigroup with the trivial $\Zset_2$-action, the directed systems and isomorphisms in the three parts can be found in the categories $\OG_\Gamma^*, \OG_\Gamma^{u*},$ and $\OG_\Gamma^{D*}$ respectively.     
\label{any_dim_gr}
\end{theorem}
\begin{proof}
If $G$ is zero, there is nothing to prove. Let us assume that $G$ is nonzero so that there is a nonzero element in $G$ and, since $G$ is directed, there is $0\neq x\in G^+.$ Assuming that $nx=0$ for some positive integer $n$ implies that $0\leq x\leq 2x\leq 3x\leq \ldots\leq nx=0$ which implies that $x=0$ since $G^+$ is strict. Hence, the elements $x, 2x, 3x,\ldots$ are different and so $G^+$ is infinite. 

The rest of the proof of (1) follows the proof of \cite[Theorem 3.19]{Goodearl_interpolation_groups_book} using results of this paper instead of the analogous results for the case $\Gamma=\{1\}.$ 

Let $\A$ be the family of all nonempty finite subsets of $G^+.$ Then $\A$ is directed under inclusion. For every $A\in\A,$ we construct a simplicial $\Gamma$-group $G_A$ with a simplicial $\Gamma$-basis stabilized by $\Delta$ and morphisms $g_A: G_A\to G$ and $g_{BA}: G_B\to G_A, B\subset A$ of $\OG_\Gamma$ such that the following four conditions hold: $g_{A}g_{BA}=g_B,$ $\ker g_B=\ker g_{BA},$ $g_{BA}g_{CB}=g_{CA}$ and $A\subseteq g_A(G_A^+)$ for all $C\subset B\subset A$.

If $A$ is a singleton $\{a\},$ let $G_A=\Zset[\Gamma/\Delta]$ and let $g_A$ be the extension of the map $\Delta\mapsto a$ to a morphism of $\OG_\Gamma$ which exists by the assumptions that $\Delta\subseteq \Stab(G)$ and that $a$ is in $G^+.$ Thus, $a=g_A(\Delta)\in g_A(G_A^+)$ and the first three conditions trivially hold since $B\subset A$ implies that $B=\varnothing.$   

Assuming that $G_A,$ $g_A,$ and $g_{BA},B\subset A,$ have been constructed for all sets $A, B\in \A$ of cardinality less than $n,$ let us consider $A\in \A$ of cardinality $n.$ Let $\B$ be the set of all nonempty and proper subsets of $A$ and let $G_\B=\bigoplus_{B\in\B} G_B$ which is a simplicial $\Gamma$-group with a simplicial $\Gamma$-basis stabilized by $\Delta.$ For $B\in\B,$ let $q_B$ be the natural injection $G_B\to G_\B.$ Let $g:G_\B\to G$ be a unique morphism of $\OG_\Gamma$ such that $g_B=gq_B.$ By Proposition \ref{telescoping}, there is a simplicial $\Gamma$-group $G_A$ with a simplicial $\Gamma$-basis stabilized by $\Delta$ and morphisms $g_{\B A}: G_\B\to G_A$ and $g_A: G_A\to G$ of $\OG_\Gamma$ such that $g=g_Ag_{\B A}$ and $\ker g_{\B A}=\ker g.$ For every $B\subset A,$ let $g_{BA}=g_{\B A}q_B.$ One checks that the above four conditions hold so that the collection  $\{(G_A, g_{BA})\, |\, A\in \A, B\subset A\}$ is a directed system.   
We list some more details for checking the third condition. 
Let $C\subset B\subset A.$ By the induction hypothesis, we have that $g_Bg_{CB}=g_C$ so that $gq_Bg_{CB}=g_Bg_{CB}=g_C=gq_C.$ Hence, $(q_Bg_{CB}-q_C)(G_C)\subseteq \ker g=\ker g_{\B A}.$ Consequently, $g_{\B A}q_Bg_{CB}=g_{\B A}q_C$ and so $g_{BA}g_{CB}=g_{\B A}q_Bg_{CB}=g_{\B A}q_C=g_{CA}.$

Let $H$ be a direct limit of the collection $\{(G_A, g_{BA})\, |\, A\in \A, B\subset A\}$ in $\OG_\Gamma$ (which exists by Proposition \ref{direct_limits}).
Let $h_A,$ for $A\in \A,$ be the translational maps and let $h: H\to G$ be a unique morphism of $\OG_\Gamma$ such that $hh_A=g_A.$ For any $A\in\A,$ $A\subseteq g_A(G_A^+)=hh_A(G_A^+)\subseteq h(H^+)$ so $G^+\subseteq h(H^+)$ and the converse holds since $h$ is order-preserving. Hence $G^+=h(H^+).$ This implies that $h$ is onto as $G$ is directed. 
If $h(x)=0$ for some $x\in H,$ let $x=h_B(x_B)$ for some $B\in \A$ and $x_B\in G_B.$ Then $g_B(x_B)=hh_B(x_B)=h(x)=0.$ Since $G^+$ is infinite, there is $A\supset B$ such that $g_{BA}(x_B)=0$ by the second condition. Thus, $x=h_B(x_B)=h_Ag_{BA}(x_B)=h_A(0)=0$ showing that $h$ is injective. 

The proofs of parts (2) and (3) are analogous to the proofs of parts (2) and (3) of Theorem \ref{countable_dim_gr}. We list some more details for the proof of (3).  Let us index the elements of $D$ by the directed set $\A$ so that $D=\{d_A\, |\, A\in \A\}$ with possible repetitions and let $H_A, h_A,$ and $h_{AB}$ be groups and morphisms which exist by part (1).  

We use induction on the number of elements of the sets $A\in\A$ to construct simplicial $\Gamma$-groups $G_A,$ connecting maps $g_{BA}$ and translational maps $g_A$ as required. If $A=\{a\},$ let $v_A=d_A\in D.$ There are $B_A\in \A$ and $u_A\in H_{B_A}^+$ such that $h_{B_A}(u_A)=v_A.$ Let $G_A$ be the convex $\Gamma$-subgroup of $H_{B_A}$ generated by $\{u_A\}$, let $D_A$ be $[0, u_A]$ (which is a generating interval of $G_A$ by Proposition \ref{generating_interval_of_simplicial}), and let $g_A$ be the restriction of $h_{B_A}$ on $G_A.$ Assume that simplicial $\Gamma$-groups $G_A,$ elements $v_A\in D$ such that $v_B, d_A\leq v_A$ for $B\subset A,$ and order-units $u_A\in G_A^+$ are defined for every $A\in \A$ with less than $n$ elements and are such that $G_A$ is a convex subset of $H_{B_A}$ for some $B_A\in\A$ and that $g_A,$ and $g_{CA}$ are the restrictions of $h_{B_A}$ and $h_{B_C B_A}$ for $C\subset A,$  $g_A(u_A)=v_A,$ $g_{CA}(u_C)\leq u_A$ for $C\subset A.$ 

Let us consider a set $A\in \A$ with $n$ elements. Let $\B$ be the set of all proper and nonempty subsets of $A.$ Since $D$ is upwards directed, there is $v_A\in D$ such that $v_A\geq v_B$ for all $B\in\B$ and $v_A\geq d_A.$ For $v_A,$ there is $C\in \A$ and $u_C\in G_C^+$ such that $h_C(u_C)=v_A.$ For any $B'\in \B,$ $0\leq v_{B'}=g_{B'}(u_{B'})=h_{B_{B'}}(u_{B'})\leq v_A=h_C(u_C).$ Let $B_A\in \A$ be larger than all $B_{B'}$ for $B'\in \B$ and $C$ and such that 
$0\leq h_{B_{B'} B_A}(u_{B'})\leq h_{C B_A}(u_C).$ Define $u_A$ as $h_{C B_A}(u_C)$ and $G_A$ as the convex $\Gamma$-subgroup of $H_{B_A}$ generated by $\{u_A\}$. 
Let $D_A$ be $[0, u_A]$ (which is a generating interval again by Proposition \ref{generating_interval_of_simplicial}), let $g_A$ be the restriction of $h_{B_A}$ to $G_A$ and let $g_{CA}$ be the restriction of $h_{B_C B_A}$ to $G_A$ for any $C\in \B.$ 
By the construction, the translational maps $g_A$ are such that $g_A(u_A)=h_{B_A}(h_{C B_A}(u_C))=h_C(u_C)=v_A$ and the connecting maps $g_{CA}$ are such that $g_{CA}(D_C)\subseteq D_A$ for $C\subset A$ since $0\leq h_{B_{C} B_A}(u_{C})\leq u_A.$ Thus $\bigcup_{A\in\A}g_A(D_A)=\bigcup_{A\in \A}[0, v_A]=D$ and $\bigcup_{A\in \A}g_A(G_A^+)=G^+$ by construction and so $(G, D)$ is a direct limit of $((G_A, D_A), g_{BA})$ in $\OG_{\Gamma}^D.$

The last sentence of the theorem follows directly since the $\Zset_2$-action is trivial. 
\end{proof}

\subsection{Realization of dimension \texorpdfstring{$\Gamma$}{TEXT}-groups for \texorpdfstring{$\Gamma$}{TEXT} with  \texorpdfstring{$\Zset[\Gamma]$}{TEXT} Noetherian}\label{subsection_realization}

The statement that every countable dimension group can be realized as a $K_0$-group of an ultramatricial algebra over a field generalizes the Effros-Handelman-Shen Theorem for $C^*$-algebras. In this section, we show the generalization of this classic result for dimension $\Gamma$-groups if $\Gamma$ is such that $\Zset[\Gamma]$ is Noetherian.

\begin{theorem} (Dimension $\Gamma$-group Realization)
Let $\Gamma$ be a group such that $\Zset[\Gamma]$ is Noetherian. If $G$ is a countable dimension $\Gamma$-group satisfying (SDP$_\Delta$) for some $\Delta\subseteq\Stab(G)$, then $G$ is an ultrasimplicial $\Gamma$-group and, as such, realizable by a $\Gamma$-graded ultramatricial ring over a graded division ring $K$ with $\Gamma_K$ equal to the normalizer of $\Delta.$

In the case when $\Gamma$ is abelian (thus $\Gamma$ is a finitely generated abelian group), and $G$ is considered to be a $\Gamma$-$\Zset_2$-bigroup with the trivial $\Zset_2$-action, then $G$ is realizable by a $\Gamma$-graded ultramatricial $*$-algebra over a graded $*$-field $K$ with $\Gamma_K=\Delta.$
\label{realization_dim_gr}
\end{theorem}
\begin{proof}
The first part of the theorem follows directly from Theorems \ref{countable_dim_gr} and \ref{realization_ultrasimplicial}. The last sentence follows from Theorems \ref{countable_dim_gr} and \ref{realization_ultrasimplicial_star}.
\end{proof}

We obtain a direct corollary of Proposition \ref{ultrasimplicial_is_dimension}, Corollary \ref{dir_lim_corollary}, and Theorem \ref{realization_dim_gr} as follows.
\begin{corollary}
Let $\Gamma$ be a group such that $\Zset[\Gamma]$ is Noetherian and let $(G, D)$ be an object of $\POG_\Gamma^D.$ Then 
$(G, D)\cong (K_0^{\Gamma}(R), D_R)$ for some $\Gamma$-graded ultramatricial ring $R$ over a graded division ring $K$ with $\Gamma_K$ normal if and only if $G$ is a countable dimension $\Gamma$-group. 
\label{realization_corollary}
\end{corollary}
\begin{proof}
One direction is a direct corollary of Proposition \ref{ultrasimplicial_is_dimension} and Corollary \ref{dir_lim_corollary}. The converse follows directly from Theorem \ref{realization_dim_gr}. 
\end{proof}

Since every finitely generated abelian group $\Gamma$ is such that $\Zset[\Gamma]$ is Noetherian and each of its subgroups is normal, Proposition \ref{ultrasimplicial_is_dimension} and Theorem \ref{realization_dim_gr} imply the following corollary. 

\begin{corollary}
Let $\Gamma$ be a finitely generated abelian group. The class of ultrasimplicial $\Gamma$-groups and the class of countable  dimension $\Gamma$-groups coincide. 
\label{realization_corollary_abelian}
\end{corollary}

\section{Some remarks and open problems}\label{section_open_problems}

\subsection{The Realization Problem} The Realization (or representation) Problem for monoids is asking for a description of all monoids which are isomorphic to $\V(R)$ for some von Neumann regular ring $R$. The survey article \cite{Ara_realization} contains a comprehensive overview of this and several related open problems. The Realization Problem for groups is asking for a description of all abelian groups realizable by a von Neumann regular ring.   

The Realization Problem is known to hold for large classes of monoids. The classic results include realization of countable dimension monoids, i.e. the positive cones of dimension groups or, equivalently, cancellative, refinement (the Riesz refinement property holds), unperforated, conical ($x+y=0 \Rightarrow x=y=0$) monoids. Any such countable monoid can be realized by an ultramatricial algebra over a field. The assumption that the monoid is countable here is relevant by an example from \cite{Wehrung} of a dimension monoid of size $\aleph_2$  which cannot be realized by a regular ring. More recently, graph monoids, monoids with simple presentations in terms of vertices of directed graphs, have been shown to be realizable (see \cite{Ara_Brustenga} and \cite{Ara_Goodearl}).    
In \cite{Ara_Pardo}, Ara and Pardo show that graph monoids appear more generally than it may seem from their definition: by \cite[Theorem 3.6]{Ara_Pardo}, any finitely generated conical regular ($2x\leq x$ for all $x$) refinement monoid is isomorphic to a graph monoid of a countable row-finite graph. 
Subsequently, in \cite{Ara_Bosa_Pardo}, Ara, Bosa, and Pardo settle the Realization Problem for finitely generated conical refinement monoids by showing that every such monoid is realizable by a regular ring. 

In the context of this paper, we generalize the Realization Problem for groups to the case when a group $\Gamma$ acts on an abelian group $G$ so that the classic Realization Problem corresponds to the case when $\Gamma$ is trivial. Recall that a $\Gamma$-graded ring $R$ is {\em graded regular} if $x\in xRx$ for every homogeneous element $x$. We state the modified Realization Problem as follows. 

{\bf $\mathbf{\Gamma}$-Realization Problem.}  Describe all abelian $\Gamma$-groups which are isomorphic to $K_0^{\Gamma}(R)$ for some $\Gamma$-graded von Neumann regular ring $R.$

If $R$ is a graded ring which is regular it is automatically graded regular. Moreover, since every regular ring can be trivially graded by any group, the class of graded regular rings is larger than the class of regular rings. In addition, every abelian group can be made into a $\Gamma$-group by considering the trivial action of $\Gamma.$ Thus, a solution of the $\Gamma$-Realization Problem will also solve the Realization Problem. 

Theorem \ref{realization_dim_gr} shows that if $\Gamma$ is such that $\Zset[\Gamma]$ is Noetherian, then every countable dimension $\Gamma$-group can be realized by a $\Gamma$-graded ultramatricial ring $R$ over a graded division ring. Note that such ring $R$ is graded regular. This brings us to our first question. 

\begin{question} Can every countable dimension $\Gamma$-group be realized by a $\Gamma$-graded ultramatricial ring over a graded division ring even if $\Zset[\Gamma]$ is not necessarily Noetherian? 
\label{question_Noetherian}
\end{question}

Every dimension $\Gamma$-group can be considered as a $\Gamma$-$\Zset_2$-bigroup with the trivial $\Zset_2$-action. If such $G$ is countable and $\Gamma$ is finitely generated abelian, $G$ can be realized by a $\Gamma$-graded ultramatricial $*$-algebra over a graded $*$-field by Theorem \ref{realization_dim_gr}. However, a countable dimension $\Gamma$-group $G$ can be equipped with some nontrivial $\Zset_2$-action making it into a $\Gamma$-$\Zset_2$-bigroup. In this case, we ask whether it can also be realized by a graded $*$-ring. In particular, since the involutive version of a regular ring is a $*$-regular ring (every principal right ideal is generated by a self-adjoint idempotent) and a graded $*$-regular ring is defined analogously by replacing ``ideal'' by ``graded ideal'' and ``idempotent'' by ``homogeneous idempotent'', we pose the following question. 

\begin{question} Can every countable dimension $\Gamma$-$\Zset_2$-bigroup with a nontrivial $\Zset_2$-action be realized by a $\Gamma$-graded $*$-regular ring? 
\end{question}

\subsection{The Grothendieck \texorpdfstring{$\Gamma$}{TEXT}-group of a smash product}\label{subsection_smash}

Generalizing some results of \cite{Cohen_Montgomery},  Ara, Hazrat, Li, and Sims (\cite{Ara_et_al_Steinberg}) showed that the category of graded left modules over a $\Gamma$-graded ring $R$ is isomorphic to the category of left modules of the smash product $R\#\Gamma.$ The ring $R\#\Gamma$ is defined by introducing new symbols $p_\gamma$ for $\gamma\in \Gamma$ and letting the finite sums of the form $\sum r p_\gamma$ be added component-wise and multiplied by the rule $rp_\gamma sp_\delta = rs_{\gamma\delta^{-1}} p_\delta$ where $s_{\gamma\delta^{-1}}\in R_{\gamma\delta^{-1}}$ denotes the term in the representation of $s$ as a sum of homogeneous terms. In \cite{Cohen_Montgomery}, Cohen and Montgomery studied the case when $\Gamma$ is finite and $R$ is an algebra over a field $K.$ The consideration of the smash product in \cite{Cohen_Montgomery} was motivated by the fact that there is a bijective correspondence of $\Gamma$-gradings on $R$ and $R\# \Hom_K(K[\Gamma], K)$-module structures of $R$. This correspondence has been established by considering $p_\gamma, \gamma\in \Gamma,$ to be the elements of the basis of $\Hom_K(K[\Gamma], K)$ dual to the standard basis $\{\gamma\, |\,\gamma\in\Gamma\}$ of $K[\Gamma].$ The notation  $R\#\Gamma$ of \cite{Ara_et_al_Steinberg} corresponds to $R\#\Hom_K(K[\Gamma], K)$ of \cite{Cohen_Montgomery}.

Although the ring $R\#\Gamma$ can also be graded by $\Gamma$ (by $(R\#\Gamma)_\gamma=\sum_\delta R_\gamma p_\delta$), the isomorphism of the categories from \cite{Ara_et_al_Steinberg} is a way to consider the action of $\Gamma$ on the standard $K_0$-group of a ring instead of on the Grothendieck $\Gamma$-group of a graded ring. Thus, the results of this paper on graded rings can be formulated in terms of the module structure over the smash product. 

For example, in the case when $\Gamma$ is $\Zset$ and $R$ is a Leavitt path algebra $L_K(E)$ of a graph $E$ over a field $K,$ $L_K(E)$ is naturally graded by $\Zset$ and there is a $\Zset$-monoid isomorphism $\V^{\Zset}(L_K(E))\cong \V(L_K(E)\#\Zset)$ by \cite[Corollary 5.3 and Proposition 5.7]{Ara_et_al_Steinberg}. This isomorphism induces an order-preserving $\Zset$-group isomorphism $$K_0^{\Zset}(L_K(E))\cong K_0(L_K(E)\#\Zset).$$
The graph $C^*$-algebra of the skew-product graph $E\times_1\Zset$ (or the covering $\ol E$ of \cite{Ara_et_al_Steinberg}) has the Grothendieck group with a natural $\Zset$-group structure since the 
inclusion $L_\Cset(E\times_1\Zset)\to C^*(E\times_1\Zset)$ induces an isomorphism of $K_0$-groups and, by  \cite[Corollary 5.3 and Proposition 5.7]{Ara_et_al_Steinberg}, $$K_0(C^*(E\times_1\Zset))\cong K_0(L_\Cset(E\times_1\Zset))\cong K_0^\Zset(L_\Cset(E)).$$  

\subsection{The converse of Propositions \ref{unperforated} and \ref{interpolation}}

If $\Gamma$ is trivial, a directed and ordered group $G$ satisfies (SDP) if and only if it is a dimension group, i.e. it is unperforated and satisfies the interpolation property, by  \cite[Proposition 3.15]{Goodearl_interpolation_groups_book} or \cite[Lemma IV.7.1]{Davidson}. The proof fundamentally relies on the fact that every element of $\Zset=\Zset[\{1\}]$ is either positive or negative, a fact which is not necessarily true for the order on $\Zset[\Gamma]$ for nontrivial $\Gamma.$ By Propositions \ref{unperforated} and \ref{interpolation}, every dimension $\Gamma$-group is ordered, directed, unperforated with respect to some subgroup of $\Gamma,$ and has the interpolation property. We ask whether a $\Gamma$-group with these four properties is necessarily a dimension $\Gamma$-group. 

\begin{question}
If $\Gamma$ is a group with a subgroup $\Delta$ and $G$ is an ordered and directed $\Gamma$-group which satisfies the interpolation property and which is unperforated with respect to $\Delta,$ does $G$ satisfy (SDP$_{\Delta'}$) for some subgroup $\Delta'$? If so, can $\Delta'$ be found such that $\Delta'\subseteq \Stab(G)$? 
\label{question_converse}
\end{question}

\subsection{Weakening the assumptions}

We ask whether the assumption that $\Zset[\Gamma]$ is Noetherian can be weakened or deleted from the results of this paper which assume this condition. For example, we ask whether Theorem \ref{countable_dim_gr} is true (and, consequently, Theorems \ref{any_dim_gr} and \ref{realization_dim_gr}) if the assumption that $\Zset[\Gamma]$ is Noetherian is weakened. 

\appendix

\section{Faithfulness and the Classification Theorem}\label{appendix_ultramatricial}

If $\Gamma$ is an abelian group and $K$ a $\Gamma$-graded field, \cite[Theorem 5.2.4]{Roozbeh_book} states that the functor $K_0^\Gamma$ classifies $\Gamma$-graded ultramatricial $K$-algebras. We show that this result holds without assuming that $\Gamma$ is abelian (Theorem \ref{classification}). In addition, the connecting maps in \cite[Theorem 5.2.4]{Roozbeh_book} are assumed to be unit-preserving while we do not make this assumption.

Analogously to the classic proof (\cite[Theorem 15.26]{Goodearl_book}), the proof of \cite[Theorem 5.2.4]{Roozbeh_book} uses the intertwining method and \cite[Theorem 5.1.3]{Roozbeh_book} stating that $K_0^\Gamma$ is a full and faithful functor on an appropriate category. If $\Gamma$ is any group and $K$ a $\Gamma$-graded field, Proposition \ref{fullness} asserts the fullness of the functor $K_0^\Gamma$ on the category of graded matricial $K$-algebras (if $K$ is commutative, the map $e_{kl}^i\mapsto g_{kl}^i$ from the proof extends to a graded {\em algebra} homomorphism, not just to a graded {\em ring} homomorphism).

We turn to faithfulness next and show Proposition \ref{faithfulness} stating that the function $K_0^\Gamma$ is faithful on 
the category obtained from the category of graded matricial $K$-algebras by identifying the graded inner automorphism with the identity map. Thus, Propositions \ref{fullness} and \ref{faithfulness} parallel \cite[Theorem 5.1.3]{Roozbeh_book}. Further, \cite[Theorem 3.4]{Roozbeh_Lia}
is the involutive version of \cite[Theorem 5.1.3, part (1)]{Roozbeh_book}: it states that if $\Gamma$ is abelian and $K$ has an involution which satisfies some additional conditions, then the appropriately modified functor is faithful. 

\begin{proposition} (Faithfulness)
Let $\Gamma$ be any group, let $K$ be a $\Gamma$-graded field, let $R$ and $S$ be graded matricial $K$-algebras, and let $\phi, \psi: R\to S$ be graded algebra homomorphisms (not necessarily unit-preserving). The following are equivalent. 
\begin{enumerate}[\upshape(1)]
\item $K^{\Gamma}_0(\phi)=K^{\Gamma}_0(\psi)$

\item There exists an invertible element $u\in S_{1_\Gamma}$ such that $\phi(r)=u\psi(r)u^{-1},$ for all $r\in R.$  
\end{enumerate}  
\label{faithfulness}
\end{proposition}
\begin{proof}
If the arguments related to the presence of an involution are not considered, the proof of \cite[Theorem 3.4]{Roozbeh_Lia} can be used for the current proof: the relations on the standard graded matrix units hold without the requirements that $\Gamma$ is abelian and Lemma \ref{iso_carries_from_K0} can be used instead of \cite[Lemma 3.2]{Roozbeh_Lia}. We also emphasize that the maps in \cite[Theorem 3.4]{Roozbeh_Lia} are not assumed to be unit-preserving. 
\end{proof}

Propositions \ref{fullness} and \ref{faithfulness} (fullness and faithfulness) imply the classification result below.

\begin{theorem} (Classification)
If $\Gamma$ is any group, $K$ is a $\Gamma$-graded field, $R$ and $S$ are graded ultramatricial $K$-algebras, and $f: (K^{\Gamma}_0(R), D_R)\to (K^{\Gamma}_0(S), D_S)$ is an isomorphism of the category $\OG^D_\Gamma$, then there is a graded algebra isomorphism $\phi: R\to S$ such that $K^{\Gamma}_0(\phi)=f.$ If $f$ is order-unit-preserving, then $\phi$ can be chosen to be unit-preserving. 
\label{classification}
\end{theorem}

\begin{proof}
We recall \cite[Theorem 4.5]{Roozbeh_Lia} stating that if $\Gamma$ is abelian and if $K$ has an involution which satisfies some additional conditions, then $\phi$ can be obtained to be a graded $*$-algebra isomorphism.  
If the arguments related to the presence of an involution are not considered and Propositions \ref{fullness} and \ref{faithfulness} are used instead of \cite[Theorem 2.7 and Theorem 3.4]{Roozbeh_Lia}, the proof of \cite[Theorem 4.5]{Roozbeh_Lia} modifies to the proof of the theorem. 
\end{proof}

One may wonder whether the assumption that $K$ is commutative can be dropped from Theorem \ref{classification} and still have that the homomorphism $\phi$ is both a graded ring and a graded $K$-module homomorphism. The example below illustrates that this cannot be required. 

Let $K$ be a $\Gamma$-graded division ring such that $\Gamma_K=\Gamma.$ For example, let $K=F[\Gamma]$ graded by $K_\gamma=F\{\gamma\}$ for any group $\Gamma$ and any field $F$ trivially graded by $\Gamma.$ Let $\gamma,\delta\in\Gamma,$ $R=\M_1(K)(\gamma)$ and $S=\M_1(K)(\delta).$ By the assumptions that $\Gamma_K=\Gamma$ and that $K$ is a graded division ring, there is an invertible element $a\in K_{\delta\gamma^{-1}}.$ Thus, the map $x\mapsto axa^{-1}$ is a graded ring isomorphism $R\cong_{\gr}S$ and so $(K_0^\Gamma(R), [R])\cong (K_0^\Gamma(S), [S])$ in $\OG^u_\Gamma.$ If we assume that there is $\phi: R\to S$ which is both a graded ring and a graded $K$-module isomorphism, we have that $\phi(x)=x\phi(1)$ for any $x\in K.$ In particular, if $x\in R_\gamma=K_\gamma,$ then $\phi(x)\in S_\gamma=K_{\delta\gamma\delta^{-1}}$ and $x\phi(1)\in K_\gamma.$ By considering $\Gamma$ and $K$ such that $K_{\delta\gamma\delta^{-1}}\neq K_\gamma,$ for example 
$\Gamma=D_3=\langle a,b | a^3=1, b^2=1, ba=a^2b\rangle,$  $K=\Cset[\Gamma],$ $\gamma=a$ and $\delta=b,$ we arrive at a contradiction. 

\section{Dimension \texorpdfstring{$\Gamma$}{TEXT}-group extensions}\label{appendix_extensions}

In the case when $\Gamma$ is trivial, every dimension group with a generating interval $D$ has a dimension group extension with an order-unit by the group $(\Zset, 1).$ In this section, we prove a generalization of this statement for an ultrasimplicial $\Gamma$-group and, consequently, for a dimension $\Gamma$-group in the case when $\Zset[\Gamma]$ is Noetherian. 

If $G, H$ and $K$ are ordered $\Gamma$-groups then $H$ is an {\em ordered $\Gamma$-group extension} of $G$ by $K$ if 
\[\xymatrix{ 0\ar[r] & G\ar[r]^i& H\ar[r]^{p}& K\ar[r] &0}\]
is a short exact sequence in the category $\OG_\Gamma$ (thus $p, i$ are order-preserving $\Gamma$-group homomorphisms such that $i$ is injective, $p$ is surjective, and $i(G)=\ker p$), $i^{-1}(H^+)=G^+,$ and $p(H^+)=K^+.$
 
If $D$ is a generating interval of $G$, $u$ an order-unit of $H,$ and $v$ an order-unit of $K$, then $(H, u)$ is an {\em ordered $\Gamma$-group extension} of $(G,D)$ by $(K,v)$
if $H$ is a $\Gamma$-group extension of $G$ by $K$, $p$ is a morphism of $\OG_\Gamma^u,$ and $i^{-1}([0, u])=D$ (which also implies that $i$ is a morphism of $\OG_\Gamma^D$). This is written as follows. 
\[\xymatrix{ 0\ar[r] & (G,D)\ar[r]^{i} & (H,u)\ar[r]^{p}& (K,v)\ar[r] &0}\]

\begin{proposition}
Let $\Gamma$ be a group with a subgroup $\Delta.$ If $G$ is an ordered and directed $\Gamma$-group with a generating interval $D$, then $(G,D)$ has an ordered $\Gamma$-group extension by $(\Zset[\Gamma/\Delta], \Delta).$
\label{extension_of_any} 
\end{proposition}
\begin{proof}
Let $H=G\oplus \Zset[\Gamma/\Delta],$ let $i: G\to H$ be the natural injection and let $p:H\to  \Zset[\Gamma/\Delta]$ be the natural projection so that the diagram
$\xymatrix{ 0\ar[r] & G\ar[r]^i & H \ar[r]^{p}  & \Zset[\Gamma/\Delta]\ar[r] &0}$
is a short exact sequence in the category of $\Zset[\Gamma]$-modules. Let us denote the set $$\{(x, a\Delta)\in H\; |\; a\in\Zset^+[\Gamma]\mbox{ and }x+ad\in G^+\mbox{ for some }d\in D\}$$ by $H^+.$ It is direct to check that $H^+$ is closed under the action of $\Gamma$ and that $(0,0)\in H^+.$ Let $(x,a\Delta),(y,b\Delta)\in H^+$ and $d_1, d_2\in D$ be such that $x+ad_1, y+bd_2\in G^+.$ Since $D$ is upwards directed, there is $d\in D$ such that $d\geq d_1, d\geq d_2$ so that $x+ad\geq x+ad_1\geq 0$ and $y+bd\geq y+bd_2\geq 0.$ Thus $x+y+(a+b)d\geq 0$ and so $H^+$ is additively closed. Hence, $H^+$ is a cone in $H.$ The cone $H^+$ is strict since the cone $G^+$ is strict. So, $H^+$ defines a partial order on $H.$ 
 
Next, we show that $(0, \Delta)$ is an order-unit of $H$. For any $(x, a\Delta)\in H,$ there is $y\in G^+$ with $x\leq y$ since $G$ is directed. Then $y=\sum_{i=1}^n b_id_i$ for some $n,$ $d_i\in D$ and $b_i\in \Zset^+[\Gamma]$ for $i=1,\ldots, n.$ Since  $D$ is upwards directed, one can find $d\in D$ such that $d_i\leq d$ for all $i=1,\ldots, n.$ Thus $x
\leq y\leq \sum_{i=1}^n b_id=bd$ for $b=\sum_{i=1}^n b_i.$ Let $c\in \Zset^+[\Gamma]$ be such that $c\geq a+b.$ This implies that $c-a\geq b$ so that $(c-a)d\geq bd$.
The relation $x\leq bd$ implies that $-x+bd\geq 0$ and so $-x+(c-a)d\geq -x+bd\geq 0.$ Thus, we have that $(-x, (c-a)\Delta)\in H^+$ and, hence, $(x, a\Delta)\leq (0, c\Delta).$ This demonstrates that $(0,\Delta)$ is an order-unit and also implies that $H$ is directed. 

By the definition of $H^+,$ we have that $i$ and $p$ are order-preserving. Since $(0, a\Delta)$ is in $H^+$ for every $a\in \Zset^+[\Gamma],$ we have that  $p(H^+)=\Zset^+[\Gamma/\Delta].$ The relation $i^{-1}(H^+)=G^+$ holds since $i(x)\in H^+$ for $x\in G$ if and only if $x$ is in $G^+.$ The relation $p(0, \Delta)=\Delta$ holds by the definition of $p$. So, it remains to show that $i^{-1}([(0,0), (0, \Delta)])=D.$ For $d\in D,$ $(-d,\Delta)\in H^+$ by the definition of $H^+.$ Hence $(0,0)\leq (d, 0)\leq (0, \Delta)$ and so $i(d)\in [(0,0), (0, \Delta)].$ Conversely, if $i(x)\in [(0,0), (0, \Delta)]$ for some $x\in G,$ then 
$(0,0)\leq (x, 0)\leq (0, \Delta)$ which implies that $(x,0), (-x,\Delta)
\in H^+.$ Hence $x\in G^+$ and $-x+d\in G^+$ for some $d\in D$ by the definition of $H^+.$ Thus $0\leq x\leq d$ which implies that $x\in D$ by the convexity of $D$.
\end{proof}

\begin{proposition}
Let $\Delta$ be a normal subgroup of $\Gamma.$ If $G$ is a simplicial $\Gamma$-group with a simplicial $\Gamma$-basis stabilized by $\Delta$ and a generating interval $D$, then $(G,D)$ has an ordered $\Gamma$-group extension $(H,u)$ by $(\Zset[\Gamma/\Delta], \Delta)$ such that $H$ satisfies (SDP$_\Delta$) and that $\Delta\subseteq \Stab(H).$ 
\label{extension_of_simplicial}
\end{proposition}
\begin{proof}
Let $H=G\oplus \Zset[\Gamma/\Delta]$ and $u=(0,\Delta).$ Then $(H, u)$ is an ordered $\Gamma$-group extension of $(G, D)$ by $(\Zset[\Gamma/\Delta], \Delta)$ by definition and by Proposition \ref{extension_of_any}. By Proposition \ref{Delta_normal_case}, $\Stab(H)=\Delta.$ So, it is sufficient to show that $H$ satisfies (SDP$_{\Delta}$). Let $X_i=(x_i, b_i\Delta)$ be in $H^+$ for $i=1,\ldots, n$ and assume that $\sum_{i=1}^n a_iX_i=(0,0)$ for some $a_i\in\Zset[\Gamma].$ This implies that $\sum_{i=1}^n a_ix_i=0,$ that the elements $b_i$ can be chosen in $\Zset^+[\Gamma],$ that $\sum_{i=1}^n \pi(a_ib_i)=0$ where $\pi$ is the natural map $\pi:\Zset[\Gamma]\to \Zset[\Gamma/\Delta],$ 
and that there are $d_i\in D$ such that $x_i+b_id_i\geq 0.$ Since $D$ is upwards directed, there is $d\in D$ such that $d_i\leq d$ for all $i=1,\ldots, n$ and, hence, $x_i+b_id\geq x_i+b_id_i\geq 0.$  

Let $\{y_1,\ldots, y_m\}\subseteq G^+$ be a simplicial $\Gamma$-basis stabilized by $\Delta.$ Thus, for $x_i+b_id, d\in G^+,$ there are $b_{ij}, c_j\in\Zset^+[\Gamma]$ for $i=1,\ldots n, j=1,\ldots, m$ such that $x_i+b_id=\sum_{j=1}^m b_{ij}y_j$ and $d=\sum_{j=1}^m c_jy_j$ for all $i=1,\ldots, n.$  The condition  $$0=\sum_{i=1}^n a_ix_i=\sum_{i=1}^n a_i(x_i+b_id)-\sum_{i=1}^n a_ib_id=\sum_{i=1}^n\sum_{j=1}^m a_ib_{ij}y_j-\sum_{i=1}^n\sum_{j=1}^m a_ib_ic_jy_j$$ implies that 
$\sum_{i=1}^n \left(\pi(a_ib_{ij})-\pi(a_ib_ic_j)\right)=0$ for all $j=1,\ldots, m.$ 

Since $\Delta$ is normal in $\Gamma$, the map $\pi$ is both a left and a right $\Zset[\Gamma]$-module homomorphism. So, $\sum_{i=1}^n \pi(a_ib_i)=0$ implies that $\sum_{i=1}^n \pi(a_ib_ic_j)=0$ for all $j=1,\ldots, m,$ and hence  
\[0=\sum_{i=1}^n \pi(a_ib_{ij})-\sum_{i=1}^n\pi(a_ib_ic_j)=\sum_{i=1}^n \pi(a_ib_{ij})-0=\sum_{i=1}^n  \pi(a_ib_{ij}).\]

Let $Y_j=(y_j, 0)$ for $j=1,\ldots, m$ and let $Y_{m+1}=(-d, \Delta).$ Then $Y_j$ and $Y_{m+1}$ are in $H^+$ by Proposition \ref{extension_of_any}. For all $i=1,\ldots n, j=1,\ldots, m,$ let $B_{ij}=b_{ij}\in \Zset^+[\Gamma]$ and $B_{i(m+1)}=b_i\in\Zset^+[\Gamma].$ For any $i=1,\ldots, n,$
\[X_i=(x_i, b_i\Delta)=(\sum_{j=1}^mb_{ij}y_j-b_id, b_i\Delta)=(\sum_{j=1}^mb_{ij}y_j,0)+b_i(-d, \Delta)=\]
\[\sum_{j=1}^m b_{ij}(y_j,0)+b_i(-d, \Delta)=\sum_{j=1}^m B_{ij}Y_j+B_{i(m+1)}Y_{m+1}=\sum_{j=1}^{m+1} B_{ij}Y_j.\]
For any $j=1,\ldots, m,$ \[\sum_{i=1}^n\pi(a_iB_{ij})=\sum_{i=1}^n\pi(a_ib_{ij})=0\;\;\mbox{  and }\;\;\sum_{i=1}^n \pi(a_iB_{i(m+1)})=\sum_{i=1}^n \pi(a_ib_i)=0.\]
\end{proof}

Proposition \ref{extension_of_simplicial} implies the following proposition. 

\begin{proposition}
Let $\Delta$ be a normal subgroup of $\Gamma.$ If $(G, D)$ is a direct limit of a directed system of simplicial $\Gamma$-groups with simplicial $\Gamma$-bases stabilized by $\Delta$ in the category $\OG^D_\Gamma,$ then $(G,D)$ has an ordered $\Gamma$-group extension $(H,u)$ by $(\Zset[\Gamma/\Delta], \Delta)$ such that $H$ satisfies (SDP$_\Delta$) and such that $\Delta\subseteq \Stab(H).$ 
\label{extension_of_ultrasimplicial}
\end{proposition}
\begin{proof}
Let $I$ be a directed set, let $(G, D)$ be a direct limit of simplicial $\Gamma$-groups $((G_i, D_i), g_{ij})$, $i,j\in I, i\leq j,$ with simplicial $\Gamma$-bases stabilized by $\Delta,$ and let $g_i, i\in I$ be the translational maps. By Proposition \ref{extension_of_simplicial}, we can find ordered $\Gamma$-group extensions $(H_i, u_i), i\in I,$ of $(G_i, D_i)$ by $(\Zset[\Gamma/\Delta], \Delta)$ which satisfy (SDP$_\Delta$) and such that $\Delta\subseteq \Stab(H_i).$ Let $\iota_i$ denote the inclusion of $(G_i,D_i)$ into $(H_i, u_i)$ and let $p_i$ denote the projection $(H_i, u_i)\to (\Zset[\Gamma/\Delta], \Delta).$ If $h_{ij}=g_{ij}\oplus 1_{\Zset[\Gamma/\Delta]}$ where the second term is the identity map on $\Zset[\Gamma/\Delta],$ then the system $((H_i,u_i), h_{ij}),$  $i,j\in I, i\leq j,$ is a directed system in $\OG^u_\Gamma.$ Note that $h_{ij}\iota_i=\iota_jg_{ij}$ and $p_jh_{ij}=p_i$  for all $i\leq j$ by the proof of Proposition \ref{extension_of_simplicial}. Let $(H, u)$ be a direct limit of this directed system which exists by Proposition \ref{direct_limits} and let $h_i=g_i\oplus 1_{\Zset[\Gamma/\Delta]}$. Since all $\Gamma$-groups $H_i$ satisfy (SDP$_{\Delta}$) and are stabilized  by $\Delta$, $H$ satisfies (SDP$_{\Delta}$) and it is stabilized by $\Delta.$ 

Define the maps $\iota: G\to H$ and $p: H\to \Zset[\Gamma/\Delta]$ by $\iota(g_i(x_i))=h_i\iota_i(x_i)$ and $p(h_i(y_i))=p_i(y_i).$ We claim that we obtain the required properties from the commutative diagram below. 
\[\xymatrix{ 0\ar[r] & (G_i,D_i)\ar[r]^{\iota_i}\ar[d]^{g_i} & (H_i,u_i)\ar[r]^{p_i}\ar[d]^{h_i}& (\Zset[\Gamma/\Delta], \Delta)\ar[r]\ar[d]^{=} &0\\
 0\ar[r] & (G,D)\ar[r]^{\iota} & (H,u)\ar[r]^{p}& (\Zset[\Gamma/\Delta], \Delta)\ar[r] &0}
\]
Indeed, one checks that the maps $\iota$ and $p$ are well-defined, that $\iota$ is injective and that $p$ is surjective using properties of a direct limit.  
Then one checks that $\ker p$ is equal to the image of $\iota,$ that $\iota$ and $p$ are order-preserving, that $\iota^{-1}(H^+)=G^+,$ that $p(u)=\Delta,$ and that $p(H^+)=\Zset^+[\Gamma/\Delta]$ using the definitions of the maps. It remains to check that $\iota^{-1}([0, u])=D.$ If $d\in D=\bigcup_{i\in I}g_i(D_i),$ then there are $i\in I$ and $d_i\in D_i=\iota_i^{-1}([0, u_i])$ such that $d=g_i(d_i).$ Hence the relation $0\leq \iota_i(d_i)\leq u_i$ holds in $H_i.$ Applying $h_i$ to this relation, we obtain that $0\leq h_i\iota_i(d_i)=\iota g_i(d_i)\leq h_i(u_i)=u.$ Hence $d=g_i(d_i)\in \iota^{-1}([0, u]).$ Conversely, if $x\in \iota^{-1}([0, u]),$ let $\iota(x)=h_i(y_i)$ for some $y_i\in H_i.$ Since $0=p\iota(x)=ph_i(y_i)=p_i(y_i),$ $y_i=\iota_i(x_i)$ for some $x_i\in G_i$ and then $\iota(x)=h_i(y_i)=h_i\iota_i(x_i)=\iota g_i(x_i)$ which implies $x=g_i(x_i)$ since $\iota$ is injective. 
The relation $0\leq \iota(x)=h_i\iota_i(x_i)\leq u=h_i(u_i)$ implies that there is $j\geq i$ such that  $0\leq h_{ij}\iota_i(x_i)=\iota_j g_{ij}(x_i)\leq h_{ij}(u_i)=u_j$ and so $g_{ij}(x_i)\in \iota^{-1}_j([0, u_j])=D_j.$ Thus $x=g_i(x_i)=g_jg_{ij}(x_i)\in g_j(D_j)\subseteq \bigcup_{i\in I} g_i(D_i)=D.$
\end{proof}

We show the main result of this appendix now. 

\begin{theorem}
If $\Gamma$ is such that $\Zset[\Gamma]$ is Noetherian and $(G, D)$ is a dimension $\Gamma$-group satisfying (SDP$_\Delta$) for some subgroup $\Delta$ of $\Gamma$ such that $\Delta\subseteq \Stab(G),$
then $(G,D)$ has an ordered $\Gamma$-group extension $(H,u)$ by $(\Zset[\Gamma/\ol\Delta], \ol\Delta),$ where $\ol\Delta$ is the normal closure of $\Delta$ in $\Gamma,$ such that $H$ satisfies (SDP$_{\ol\Delta}$) and such that $\ol\Delta\subseteq \Stab(H).$  
\label{extension_of_dimension}
\end{theorem}
\begin{proof}
By Theorem \ref{any_dim_gr}, we can find a directed system of simplicial $\Gamma$-groups with simplicial $\Gamma$-bases stabilized by $\ol\Delta$ such that $(G,D)$ is a direct limit of this directed system in $\OG^D_\Gamma.$ By Proposition \ref{extension_of_ultrasimplicial}, an ordered $\Gamma$-group extension  $(H,u)$ by $(\Zset[\Gamma/\ol\Delta], \ol\Delta)$ exists and it satisfies the required properties.  
\end{proof}

We ask whether the assumption that $\Zset[\Gamma]$ is Noetherian can be dropped from Theorem \ref{extension_of_dimension}.

\end{document}